\documentclass[12pt,amsfonts]{article} 
\usepackage[utf8x]{inputenc}
\usepackage{tikz} 
\usepackage[all]{xy}
\usepackage{authblk}
\usepackage[babel=true]{csquotes}
\usepackage{amsmath,amsthm, amssymb,amsfonts}
\usepackage[hmargin=1in,vmargin=1in]{geometry}
\numberwithin{equation}{subsection} 
\xyoption{arc}
\usepackage{eucal}
\usepackage{indentfirst}
\usepackage{mathrsfs}
\usepackage{verbatim}
\usepackage{makeidx}
\usepackage{graphicx}
\usepackage{yfonts} 
\usepackage{hyperref}
\usepackage{enumitem} 
\usepackage{extpfeil} 
\usepackage{dsfont}
\usepackage[T1]{fontenc}
\newtheorem{thm}{Theorem}[section]
\newtheorem{prop}[thm]{Proposition}
\newtheorem{pdef}[thm]{Proposition-Definition}
\newtheorem{lem}[thm]{Lemma}
\newtheorem*{thmsans}{Theorem}

\newtheorem{cor}[thm]{Corollary}
\newtheorem*{defsans}{Definition} 

\newtheoremstyle{bidule}
{3pt}
{3pt}
{}
{}
{\scshape}
{.}
{.5em}
{}
\newtheorem{df}[thm]{Definition}

\theoremstyle{definition}
\newtheorem{rmk}[thm]{Remark}
\newtheorem{ex}[thm]{Example}

\newtheorem*{note}{Note}
\newtheorem*{thank}{Acknowledgments}
\newtheorem*{warn}{Warning}

\newtheorem{nota}[thm]{Notation}

\newcommand{\Ev}{\tx{Ev}}
\newcommand{\C}{\mathcal{C}}
\newcommand{\Ub}{\mathcal{U}}

\newcommand{\F}{\mathcal{F}}

\renewcommand{\O}{\mathcal{O}}
\newcommand{\Ar}{\text{Arr}}

\newcommand{\D}{\mathcal{D}}
\newcommand{\Ba}{\mathcal{B}}

\newcommand{\Xa}{\mathcal{X}}

\newcommand{\A}{\mathcal{A}}
\newcommand{\Aa}{\mathcal{A}}

\newcommand{\M}{\mathscr{M}}

\newcommand{\Ja}{\mathbf{J}} 
\newcommand{\J}{\mathcal{J}} 

\newcommand{\T}{\mathcal{T}}

\newcommand{\Ea}{\mathcal{E}} 
\newcommand{\Qa}{\mathcal{Q}} 
\newcommand{\W}{\mathscr{W}}

\newcommand{\I}{\mathbf{I}}
\newcommand{\Un}{\mathbb{I}} 

\newcommand{\G}{\mathcal{G}}

\newcommand{\Z}{\mathbb{Z}}

\renewcommand{\L}{\text{L}} 
\newcommand{\Fb}{\mathbf{F}}
\newcommand{\Ya}{\mathcal{Y}}
\newcommand{\Sim}{\mathscr{S}}

\newcommand{\Pa}{\mathcal{P}}

\renewcommand{\to}{\longrightarrow}

\newcommand{\ul}{\underline}
\newcommand{\xrw}{\xrightarrow}
\newcommand{\xlw}{\xleftarrow} 
\newcommand{\hrw}{\hookrightarrow}
\newcommand{\xhrw}{\xhookrightarrow}
 
\newcommand{\N}{\mathbb{N}}

\newcommand{\Ob}{\text{Ob}}
\newcommand{\0}{\textbf{0}} 

\newcommand{\tx}{\text}
\newcommand{\tld}{\widetilde}

\renewcommand{\to}{\longrightarrow}
\DeclareMathOperator\Id{Id}
\DeclareMathOperator\Hom{Hom}

\DeclareMathOperator\Cat{\mathbf{Cat}}
\DeclareMathOperator\colim{\tx{$colim$}}

\DeclareMathOperator\cof{\mathbf{cof}}
\DeclareMathOperator\kb{\mathbf{K}} 

\DeclareMathOperator\fib{\mathbf{fib}} 

\DeclareMathOperator\Map{\tx{Map}}  
 
\DeclareMathOperator\Arr{\text{Arr}} %

\DeclareMathOperator\sset{\mathbf{sSet}} 
 
\DeclareMathOperator\oalg{\O\tx{-$Alg$}}
\DeclareMathOperator\ag{\mathscr{A}}
\DeclareMathOperator\mua{\M_\Ub[\ag]}
\DeclareMathOperator\mdua{(\M \downarrow \Ub)}

\DeclareMathOperator\gcr{[\G]}
\DeclareMathOperator\pig{\pi_{\G}}
\DeclareMathOperator\pif{\pi_{\F}}
\DeclareMathOperator\Piar{\Pi_{\Ar}}
\DeclareMathOperator\Pio{\Pi_{0}}
\DeclareMathOperator\Piun{\Pi_{1}}
\DeclareMathOperator\ogcr{[\G_0,\G_1,\pi_{\G}]}
\DeclareMathOperator\mdu{\M_{\Ub}[\ag]} %
\DeclareMathOperator\kbi{\kb_{\I}}
\DeclareMathOperator\ali{\alpha_{\I}}

\DeclareMathOperator\Fc{[\F]}
\DeclareMathOperator\Gc{[\G]}
\DeclareMathOperator\Pc{[\Pa]}
\DeclareMathOperator\Qc{[\Qa]}
\DeclareMathOperator\Ec{[\Ea]}

\DeclareMathOperator\arm{\Ar(\M)}
\DeclareMathOperator\armij{\Ar(\M)_{inj}}
\DeclareMathOperator\armpj{\Ar(\M)_{proj}}
\DeclareMathOperator\sg{\sigma}

\DeclareMathOperator\facm{(\mathscr{L}_{\M}, \mathscr{R}_{\M})}
\DeclareMathOperator\ra{\mathscr{R}_{\ag}}
\DeclareMathOperator\faca{(\mathscr{L}_{\ag}, \mathscr{R}_{\ag})}
\DeclareMathOperator\la{\mathscr{L}_{\ag}}
\DeclareMathOperator\lm{\mathscr{L}_{\M}}
\DeclareMathOperator\mr{\mathscr{R}_{\M}}

\DeclareMathOperator\pip{\pi_{\Pa}}
\DeclareMathOperator\piq{\pi_{\Qa}}
\DeclareMathOperator\muaij{\M_{\Ub}[\ag]_{inj}} %
\DeclareMathOperator\muapj{\M_{\Ub}[\ag]_{proj}} %
\DeclareMathOperator\Top{\mathbf{Top}} %
\DeclareMathOperator\Iam{\mathbf{I}_{\M}} %
\DeclareMathOperator\Jam{\mathbf{J}_{\M}} %
\DeclareMathOperator\Iag{\mathbf{I}_{\ag}} %
\DeclareMathOperator\Jag{\mathbf{J}_{\ag}} %
\DeclareMathOperator\Jcij{\Ja_{\muaij}^+}
\DeclareMathOperator\mum{\M_{\Ub}[\M]} %
\DeclareMathOperator\qi{\varepsilon} 
\DeclareMathOperator\spu{\tx{Sp}_{\Ub}(\M) } %
\DeclareMathOperator\spup{\tx{Sp}_{\Ub}^+(\M) } %
\DeclareMathOperator\zd{\Z_{disc}} 
\DeclareMathOperator\nd{\N_{disc}} 
\DeclareMathOperator\od{\mathbf{O}} 
\DeclareMathOperator\ods{\mathbf{O}_{disc}} 
\DeclareMathOperator\zmum{\Hom(\zd, \mum)} 

\title{Quillen-Segal algebras and \\ Stable homotopy theory} 
\author{Hugo Bacard \thanks{\textit{E-mail address}: \href{mailto:hugo.bacard@ac-versailles.fr}{hugo.bacard@ac-versailles.fr}
}}
\date{}
\begin{document}
\maketitle
\begin{abstract}
Let  $\M$  be a monoidal model category that is also combinatorial and left proper. If $\O$ is a monad, operad, properad, or a PROP; following Segal's ideas we develop a theory of Quillen-Segal $\O$-algebras and show that we have a Quillen equivalence between usual $\O$-algebras and Quillen-Segal algebras. We use this theory to get the \emph{stable homotopy category} by a similar method as Hovey. 
\end{abstract}
\setcounter{tocdepth}{1}
\tableofcontents
\newpage

\section{Introduction}
This paper is part of a project that aims to develop a homotopy theory of \emph{ weak algebraic structures} encoded by objects such as operads, properads, PROP's or monads. And we hope that this theory can be useful to understand in some way the algebra of higher categories and its applications.  The notion of symmetric operad was introduced by May \cite{May_geom} as a solution to the \emph{delooping problem} in stable homotopy theory.  Segal  \cite{Seg1}  gave a solution to the same problem that is of simplicial nature with the notion of $\Gamma$-spaces.  Both solutions have inspired different directions in the development of \emph{Higher Category Theory}, and we shall refer the reader to the book of Simpson \cite{Simpson_HTHC} for a detailed account on the subject.\ \\

We choose  a general formalism that we hope, will be somehow ``a bridge'' between the operadic and the simplicial method for the delooping problem.  Let $\M$ be a model category and let $\Ar(\M)$ be its category of morphisms (or arrows). For $n \in \N$, define inductively $\Ar^{n+1}(\M)= \Ar(\Ar^{n}(\M))$, the category of \emph{hyper-cubes} in $\M$; with $\Ar^0(\M)= \M$. In order to simplify the treatment of the homotopy theory in  ``Segal situations'' we start with the following definition. 
\begin{defsans}
Let $\C$ be an arbitrary category. 
\begin{enumerate}
\item An \emph{$\M$-valued Quillen-Segal theory $\T$} on $\C$ is a family of functors called \emph{Segal data functors}: $\{\T_i: \C \to \Ar(\M) \}_{i \in I},$
for some small set $I$.
\item Say that an object $c \in \C$ satisfies the \emph{generalized Segal conditions}, if for every $i \in I$ $\T_i(c) \in \Ar(\M)$ is a weak equivalence in $\M$. 
\item Inductively, an  \emph{$n$-fold $\M$-valued Quillen-Segal theory $\T$} on $\C$ is a family of functors, called \emph{Segal $n$-data functors}: $\{\T_i: \C \to \Ar^{n}(\M) \}_{i \in I},$
for some small set $I$.
\end{enumerate}
\end{defsans}

In the original paper of Segal \cite{Seg1}, $\C$ is the category of $\Gamma$-spaces,  $\M=\Top$ and the theory is given by the family functors $\{\T_n: \C \to \Ar(\Top) \}_{n \geq 1}$, where $\T_n(\Aa)$ is the $n$th \emph{Segal map}: $ p_n: \Aa(n) \to  \Aa(1) \times_{\Aa(0)} \cdots \times_{\Aa(0)} \Aa(1).$  The same formula defines the theory for \emph{Segal spaces}  as in Rezk \cite{Rezk_proper}, and classical Segal categories (see \cite{Bergner_mon_seg}, \cite{Pel},\cite{Simpson_HTHC}, \cite{Tam}). These theory have been extended to Segal $n$-categories and Segal enriched $\M$-categories (see \cite{Bacard_TWEC}, \cite{Simpson_HTHC}). One can iterate the process \emph{à la Simpson-Tamsamani} to consider a $\C$-valued theory on a category $\D$: $\{ \T_j':\D \to \Ar(\C)\}_{j\in J}$  to get a $2$-fold $\M$-valued theory $\{\Ar(\T_i) \circ \T_j':\D \to \Ar^2(\M)\}_{(i,j) \in I\times J}$ and so on.\\ 

There are many categories that are equipped with a relevant Quillen-Segal theory, and some of them will be reviewed later. But in this paper we want to extend Segal's formalism to the following situations.
\begin{enumerate}
\item Let $\C= \Omega-Spec$ be the category of $\Omega$-prespectrum with the theory given by the functors $\T_n: \C \to \Arr(\sset_{\ast})$ that takes a prespectrum $X$ to the connecting morphism $X_n \to \Omega (X_{n+1})$ which is adjoint to the map $S^1 \wedge X_n \to X_{n+1}$. Then a fibrant prespectrum satisfying the Segal conditions is simply an $\Omega$-spectrum, thus a generalized cohomology theory. 
\item  Let $\C= \mdua$ be  the comma category associated to a (right Quillen) functor $\Ub: \ag \to \M$. We remind the reader  that an object of $(\M\downarrow\Ub)$ is a triple  $\Fc = [\F_0, \F_{1}, \pi_{\F}: \F_0 \to \Ub(\F_1)] \in \M \times \ag \times \Ar(\M)$. The theory is given by the functor that projects the maps $\pif$:
 $$\Piar: \mdua \to \Arr(\M),\quad \tx{with} \quad \quad \Piar(\Fc) = \pif.$$
\end{enumerate}

If $\O$ is an operad, monad, or a PROP; a \emph{Quillen-Segal $\O$-algebra} is an object $\Fc \in \C=\mdua$, that satisfies the Segal condition for the forgetful functor $\Ub: \oalg(\M) \to \M$.
The first example of Quillen-Segal algebras comes from  \emph{co-Segal algebras} and their generalizations to co-Segal categories \cite{Bacard_TWEC}.\\ \ 

Before discussing further the theory of Quillen-Segal algebras we have the following theorem on Quillen-Segal theories in general (see Theorem \ref{local-inj-theory})
\begin{thmsans}
Let $\M$ be a cofibrantly generated model category and let $\T=(\T_i)$ be a Quillen-Segal theory on $\C$.
Assume that one the following conditions holds.

\begin{enumerate}
\item $\C$ has a model structure $(\W,\cof,\fib)$ which is cellular and left proper such that every $\T_i \in \Hom(\C, \arm_{inj})$ is right Quillen.
\item $\C$ has a model structure $(\W,\cof,\fib)$ which is combinatorial and left proper such that every $\T_i \in \Hom(\C, \arm_{inj})$ is right Quillen.
\end{enumerate}
Then there is a new model structure $\C_B(\T)=(\W_B,\cof_B,\fib_B)$ on $\C$ which is a left Bousfield localization, such that:
\begin{itemize}
\item $\T=\{\T_i: \C \to \armij \}_{i \in I}$ defines an Quillen-Segal theory such that every $\T_i$ is also right Quillen with respect to the model structure $\C_B(\T)$;
\item  every fibrant object in this new model structure satisfies the generalized Segal conditions.
\end{itemize}
\end{thmsans}
Here $\arm_{inj}$ is the \emph{injective model} structure that can be found for example in Hovey \cite{Hovey_Arr}. There is also a \emph{projective model} structure on $\arm$ denoted by $\arm_{proj}$. We will review them later in a more general context. 
We also have the following theorem if the cofibrations in $\M$ are generated by a set of maps between cofibrant objects. Example of such model categories are \emph{tractable} model categories in the sense of Barwick \cite{Barwick_localization} . The following theorem is valid outside tractable model categories (see Theorem \ref{local-proj-theory}).
\begin{thmsans}
Let $\M$ be a tractable model category and let $\T=(\T_i)$ be a Quillen-Segal theory on $\C$.
Assume that one the following conditions holds.

\begin{enumerate}
\item $\C$ has a model structure $(\W,\cof,\fib)$ which is cellular and left proper such that every $\T_i \in \Hom(\C, \arm_{proj})$ is right Quillen.
\item $\C$ has a model structure $(\W,\cof,\fib)$ which is combinatorial and left proper such that every $\T_i \in \Hom(\C, \arm_{proj})$ is right Quillen.
\end{enumerate}
Then there is a new model structure $\C_B(\T)=(\W_B,\cof_B,\fib_B)$ on $\C$ which is a left Bousfield localization, such that:
\begin{itemize}
\item $\T=\{\T_i: \C \to \armpj \}_{i \in I}$ defines an Quillen-Segal theory such that every $\T_i$ is also right Quillen with respect to the model structure $\C_B(\T)$;
\item  every fibrant object in this new model structure satisfies the generalized Segal conditions.
\end{itemize}
\end{thmsans}

These theorems require the left properness of the model structure on  $\C$. But it's possible to have a localization of a theory without left properness using a result of Beke \cite{Beke_2} which is itself a consequence of a theorem of Jeff Smith. This is what we do for the comma category $\mua:=\mdua$ and in particular for Quillen-Segal algebras for a combinatorial model category $\M$.\ \\
First we prove that:
\begin{enumerate}
\item There is an injective model structure on the comma category $\mua:= \mdua$ (Theorem \ref{inj-thm})
\item There is a projective model structure on $\mua:= \mdua$ (Theorem \ref{proj-thm}).
\item If $\Ub= \Id$ we recover the injective and projective model structure on $\arm$.
\end{enumerate} 

Category theory provides an embedding $\iota: \ag \to \mdua$ with $\iota(\Pa)=[\Ub(\Pa), \Pa, \Id_{\Ub(\Pa)}]$ (see Proposition \ref{embed-prop}) ; and this is how usual strict algebras are $QS$-algebras ($\Id_{\Ub(\Pa)}$ is always a weak equivalence). We have a functor $\Piun:\mdua \to \ag$ that is simultaneously a left adjoint and a retraction for $\iota$. The object $\F_1$ is a usual $\O$-algebra and with the weak equivalence $\pif:\F_0 \to \Ub(\F_1)$, one may want to lift the $\O$-algebra structure to $\F_0$. This is a classical problem of \emph{homotopy invariance for algebras} that goes back to Boardman and Vogt \cite{board_vogt}, Dwyer, Kan and Smith \cite{DKS} and others. There are many results in this direction that can be found for example in Berger-Moerdijk \cite{Ber_Moer_axio_hop}, Johnson-Yau \cite{Jo_Yau} and the many references therein. We will discuss it in a future work.\ \\

One of the central results on the homotopy theory on $\mdua$ is that we can localize directly the injective and projective model structure along the functor $\Piun: \mdua \to \ag$. 
\begin{thmsans}
Let $\Ub:\ag \leftrightarrows \M: \Fb$ be Quillen adjunction between combinatorial model categories where $\Ub$ is right adjoint.   
Then the following hold.
\begin{enumerate}
\item There is a model structure on the category $\mdua=\mua$ such that:
\begin{enumerate}
\item fibrant objects satisfy the Segal condition and 
\item  we have a   a Quillen equivalence $\iota: \ag \leftrightarrows \mua: \Piun $, where $\iota$ is right Quillen and $\Piun(\Fc)=\F_1$.
\end{enumerate}
\item If the Quillen pair $\Ub:\ag \leftrightarrows \M: \Fb$ is a Quillen equivalence then the functor $$\Pio: \mdu \to \M$$ is a right Quillen equivalence and $\Ub$ is a composite of Quillen equivalences:
$$\ag \xrw{\iota} \mua \xrw{\Pio} \M$$
\end{enumerate}
\end{thmsans}

If $\M$ is tractable we can show that this localization is in fact a left Bousfield localization. The proof is just a consequence of Ken Brown's lemma. One gets the homotopy theory of Quillen-Segal $\O$-algebras when $\Ub$ is the forgetful functor.\ \\

The relation with Stable Homotopy Theory comes when $\ag= \M$ and we consider  the product $\prod_{n \in \Z} \mum$ for an endofunctor $\Ub: \M \to \M$ such as the loop space functor $\Omega: \sset_{\ast} \to \sset_{\ast}$. This product is in fact a functor category $\Hom(\zd, \mum)$ where $\zd$ is the set of integers regarded as a discrete category and \emph{not} as posetal category.  The category $\spu$, of $\Ub$-prespectrum, is equivalent to a category of objects in $\Hom(\zd, \mum)$ that are \emph{linked} in the sense of Definition \ref{def-spectra}. This link condition is just a $2$-pullback condition, that can be considered as a \emph{descent} condition in some sense. 

Our goal here is to have a picture where the operations on spectra are seen in the diagram category $\Hom(\zd, \mum)$. We take the benefit of the existing results on the homotopy theory and the algebra of diagram categories that have been studied for decades.  For example if we index over $\N$ the \emph{strict projective model structure} that goes back to Bousfield and Friedlander \cite{Bous_Fried}, follows directly from the projective (=Reedy) model structure on $(\M \downarrow \Ub)= \mum$ in Theorem \ref{proj-thm}.\ \\

The functor $P: \spu \to \Hom(\zd, \mum)$ that forgets the link has a left adjoint (Proposition \ref{prop-adj-P}); and we use this (monadic) adjunction  to get various  model structure on the category of $\Ub$-prespectrum. If $\M$ is locally presentable, there are 4 model structures on $\spu$ that follow from the injective and projective model structures on $\Hom(\zd, \mum_{inj})$ and $\Hom(\zd, \mum_{proj})$. All of them are combinatorial and left proper if $\M$ is left proper. In this paper we restrict to the case where $\M$ is combinatorial and left proper to keep the paper short. We will discuss in a different paper the case where $\M$ is cellular. The arguments remain essentially the same. 

We show in virtue of Theorem \ref{local-proj-theory} that:
\begin{thmsans}
Let $\C$ be the category of prespectrum in simplicial sets  with the  Quillen-Segal theory $\T_n: \C \to \Ar(\sset_{\ast})$ given  $\T_n(X)= [X_n \to \Omega(X_{n+1})]$. 
\begin{enumerate}
\item Then there is a model structure on $\C$ such that the fibrant objects are the $\Omega$-spectrum $X$  that are level fibrant (Kan).
\item The homotopy category is equivalent to the stable homotopy category of Bousfield-Friedlander obtained by Hovey \cite{Hov_stable} and Schwede \cite{Schwede_stable}.
\end{enumerate}
\end{thmsans}

Finally we set up a definition of generalized $\Ub$-chain complexes for future applications as linked $\Z$-sequences that satisfy the chain condition.  It's not new that chain complexes and spectra behave similarly.

\subsection*{Applications} 
There are various applications of the theory that is being developed here. We outline for the moment some  directions of interest. They will appear in the subsequent papers. 
\begin{enumerate}
\item Our interest in comma categories comes when we consider a right Quillen  $\Ub : \ag \to \M$ that is a  Quillen equivalence. An important example is the homotopy coherent nerve of Cordier and Porter \cite{PC-hcc} from simplicial categories to quasicategories.  Joyal showed that this right Quillen functor is in fact a Quillen equivalence (see Bergner  \cite{Bergner_inf,Bergner_scat}). Our theorem shows that the comma category that lives in between is Quillen equivalent to both simplicial categories and quasicategories. On the one hand, Lurie and Joyal gave a significant amount of work on quasicategories. And on the other hand, simplicial categories are $\sset$-enriched categories, and there is a lot in the literature on the subject. A classical reference is the book of Kelly \cite{Ke}. We would like to understand how the existing constructions for enriched categories such as weighted limits, colimits, adjunction, Cauchy completion, etc; interact with the corresponding notions introduced by Joyal and Lurie.

\item Following the ideas of Jardine \cite{Jardine_simpresh}, Hirshowitz-Simpson \cite{HS}, Morel-Voevodsky \cite{Mo_Vo_A1}, Toën-Vezzozi \cite{ToVe1} and others the study of presheaves on a Grothendieck site with coefficient in $\Hom(\zd, \mum)$ has application in motivic homotopy theory. Indeed we can attain presheaves of spectra this way.
\item The theory of operads, properads, PROP's and algebras over them has applications in several fields of mathematics (see for example Loday-Vallette \cite{Lo_Val}, Markl-Shnider-Stasheff \cite{Markl_SS}). Many of the existing construction for usual algebras can be transported to Quillen-Segal algebras. When we consider a general model category,  the problem of homotopy invariance for algebras doesn't always have a solution. It turns out that this problem is intimately related to a conjecture of Simpson about \emph{weak units} in Higher categories (see Joyal-Kock \cite{Joyal_Kock_wu}). We will address  this later.

\item One of the consequences  of this work is a \emph{strictification} theorem for co-Segal algebras and co-Segal categories. This is the analogue of a theorem of Bergner \cite{Bergner_rigid} for the strictification of Segal categories.

\end{enumerate}

\subsection*{Organization of the paper}
\begin{itemize}
\item Section $2$ is about Quillen-Segal theories. The main result is the ``Bousfield localization of the theory''.
\item We introduce Quillen-Segal algebras in Section $3$.
\item Section $4$ contains the category theory of comma constructions.
\item  In Section $5$, we develop the homotopy theory on the comma category $\mdua$. We follow closely the same method as Hovey \cite{Hovey_Arr} to get an \emph{injective} and \emph{projective} model structure. 
\item In Section $6$, we set up some material on Stable Homotopy Theory.
\end{itemize}

\subsection*{Related works} 
Model structures on comma categories have been discussed  by  Stanculescu \cite[Section 6.5]{Stanculescu_multi}, Toën \cite{To_hall} and others. But they consider instead the comma categories $(\Ub \downarrow \M)$, where $\Ub$ is a left Quillen functor. Considering other motivations, Stanculescu focused on the dual notion of what we call Quillen-Segal object, in that one would demand  $\Ub(\F_0) \to \F_1$ to be a weak equivalence. But in the case of algebras,  we will capture the homotopy theory of free algebras.  He also mentioned that it could be interesting to study the comma category $\mdua$ and consider what we call Quillen-Segal object. 

The general philosophy of this paper is also close to that of Hollander \cite{Hollander_stack}, Hirschowitz-Simpson \cite{HS}, Jardine \cite{Jardine_simpresh}, Joyal-Tierney \cite{Joy-Tier},  Stanculescu \cite{Stanculescu_stack} and many others. By this we mean that the theory of \emph{higher stacks} is a theory of fibrant objects. And these fibrant objects satisfy some generalized Segal conditions with respect to a certain Quillen-Segal theory.

\begin{thank}
Part of this paper was done when I was a postdoctoral fellow in the research group of Rick Jardine, and I would like to thank him for his support during those years. I would like to thank David White for some helpful comments  and for referring me to a paper of  Mark Hovey that was inspirational. 
\end{thank}

\newpage
  
\section{Quillen-Segal theories}
The goal of this section is to provide a general framework that is commonly used to produce a model structure on some category $\C$ such that the fibrant objects satisfy \emph{some generalized Segal conditions}.\ \\

Let $\M$ be a general model category, that will be in many cases a monoidal model category in the sense of Hovey \cite{Hov-model}. We will denote by $\Ar(\M)$ its category of morphisms, also called the \emph{arrow category}. Hovey \cite{Hovey_Arr} provides the \textbf{injective and projective} model structures on $\Ar(\M)$. In both model structure the weak equivalences are the objective-wise weak equivalences. We will denote them by $\armij$ and $\armpj$ and they will be reviewed in Section \ref{inj-proj-comma}. The category $\Arr(\M)$ is a diagram category and these model structures are particular case of Reedy model structures which have been widely discussed in the literature (see  \cite{DHK},\cite{Jardine-Goerss},\cite{Hirsch-model-loc},\cite{Simpson_HTHC}).

\begin{df}\label{qs-theory}
Let $\C$ be an arbitrary category. 
\begin{enumerate}
\item An \emph{$\M$-valued Quillen-Segal theory $\T$} on $\C$ is a family of functors called \emph{Segal data functors}: $\{\T_i: \C \to \Ar(\M) \}_{i \in I},$
for some small set $I$.
\item Say that an object $c \in \C$ satisfies the \emph{generalized Segal conditions}, if for every $i \in I$ $\T_i(c) \in \Ar(\M)$ is a weak equivalence in $\M$. 
\item Inductively, an  \emph{$n$-fold $\M$-valued Quillen-Segal theory $\T$} on $\C$ is a family of functors, called \emph{Segal $n$-data functors}: $\{\T_i: \C \to \Ar^{n}(\M) \}_{i \in I},$
for some small set $I$.
\end{enumerate}
\end{df}

\begin{ex}
The first example comes of course from $\Gamma$-spaces and Segal $n$-categories. We shall outline the idea for classical Segal categories. Let $\sset$ be the category of simplicial sets and set $I= \Ob(\Delta)$ and let $\C= \Hom(\Delta^{op}, \sset)$ be the category of simplicial spaces.  For $n \in \Delta$, let $\T_n: \C \to \Ar(\sset) $ be the functor that takes a simplicial space $\Aa$ and  projects the $n$th Segal map: $p_n: \Aa(n) \to  \Aa(1) \times_{\Aa(0)} \cdots \times_{\Aa(0)} \Aa(1).$
As mentioned before, it's the same formulas for Segal spaces \cite{Rezk_mh}. 
\begin{ex}
Leinster \cite{Lei3} introduced up-to-homotopy algebras as as colax monoidal functor $\Xa: (\Delta^{+},+,0) \to (\ul{M}, \otimes ,I)$ such that for every $n, m \in \Delta^+$ the colaxity map $\varphi_{n,m}: \Xa(n+m) \to  \Xa(n) \otimes \Xa(m)$ and $\varphi_0: \Xa(\0) \to I$ are weak equivalences. In this case the theory is given by the functor $\T_{n,m}$ defined by $\T_{n,m}(\Xa)=\varphi_{n,m}$ and $\T_0(\Xa)= \varphi_0$. We've generalized this notion to Segal enriched category \cite{SEC1}. And the homotopy theory is difficult because precisely the functors $\T_{n,m}$ are hardly right adjoints for a general tensor products $\otimes \neq \times$. We will discuss it in \cite{SEC2}.
\end{ex}
\end{ex}
\subsection{Injective and projective  theory}

\begin{df}\label{inj-proj-qs-theory}
Let $\C$ be a model category. 

\begin{enumerate}
\item An \textbf{injective  $\M$-valued Quillen-Segal theory } on $\C$ is a family of right Quillen functors, called \emph{Segal data functors}: $$\{\T_i: \C \to \armij \}_{i \in I}$$
for some small set $I$.
\item A \textbf{projective $\M$-valued Quillen-Segal theory} on $\C$ is a family of right Quillen functors, called \emph{Segal data functors}: $$\{\T_i: \C \to \armpj \}_{i \in I}$$
for some small set $I$.
\item We will say that there is an $\M$-valued  Quillen-Segal theory on $\C$ if there is a projective or an injective $\M$-valued Quillen-Segal theory on $\C$.
\end{enumerate}
\end{df}

We remind the reader that part of being a right Quillen functor means that there is a left adjoint $\Upsilon_i: \arm \to \C$ to each of the Segal data functor $\T_i$. It also implies that $\Upsilon_i$ is automatically a left Quillen, in that it preserves the cofibrations and the trivial cofibrations in $\arm$.

\begin{nota}
If $\C$ is model category we will denote by $\W$, $\cof$, $\fib$, the three classes of weak equivalences, cofibrations, and fibrations, respectively.
\end{nota}


\begin{thm}\label{local-inj-theory}
Let $\M$ be a cofibrantly generated model category and let $\T=(\T_i)$ be a Quillen-Segal theory on $\C$.
Assume that one the following conditions holds.

\begin{enumerate}
\item $\C$ has a model structure $(\W,\cof,\fib)$ which is cellular and left proper such that every $\T_i \in \Hom(\C, \arm_{inj})$ is right Quillen.
\item $\C$ has a model structure $(\W,\cof,\fib)$ which is combinatorial and left proper such that every $\T_i \in \Hom(\C, \arm_{inj})$ is right Quillen.
\end{enumerate}
Then there is a new model structure $\C_B(\T)=(\W_B,\cof_B,\fib_B)$ on $\C$ which is a left Bousfield localization, such that:
\begin{itemize}
\item $\T=\{\T_i: \C \to \armij \}_{i \in I}$ defines an Quillen-Segal theory such that every $\T_i$ is also right Quillen with respect to the model structure $\C_B(\T)$;
\item  every fibrant object in this new model structure satisfies the generalized Segal conditions.
\end{itemize}
\end{thm}

We will give  the proof at the end of the section.  For a projective Quillen-Segal theory, we need to recall the definition of a tractable model category as in Barwick \cite{Barwick_localization}.
\begin{df}
A cofibrantly generated model category $\M$ is \textbf{tractable} if $\M$ is locally presentable and both cofibrations and trivial cofibrations are generated by a set of morphisms between cofibrant objects. 
\end{df}

\begin{thm}\label{local-proj-theory}
Let $\M$ be a tractable  model category and let $\C$ be a category.
Assume that one the following conditions holds.

\begin{enumerate}
\item $\C$ has a model structure $(\W,\cof,\fib)$ which is cellular and left proper.
\item $\C$ has a model structure $(\W,\cof,\fib)$ which is combinatorial and left proper.
\end{enumerate}
Then for any projective theory $\T=\{\T_i: \C \to \armpj \}_{i \in I}$ on $\C$ with  respect to the model structure $(\W,\cof,\fib)$, there is a new model structure $\C_B(\T)=(\W_B,\cof_B,\fib_B)$ on $\C$ which is a left Bousfield localization such that:
\begin{itemize}
\item $\T=\{\T_i: \C \to \armpj \}_{i \in I}$ defines a projective theory with respect to the model structure $\C_B(\T)$ and
\item  every fibrant object in this new model structure satisfies the generalized Segal conditions.
\end{itemize}
\end{thm}

The guiding principle in proving Theorem \ref{local-inj-theory} and Theorem \ref{local-proj-theory} is to develop some techniques that allow one to perform two tasks: the first task is to be able to detect when an object satisfies the Segal conditions; and  the second task is to have a functorial process that takes an object $c$ and creates an object $\Sim(c)$ that satisfies the Segal conditions. For the first task, this amounts to know when a morphism $f=\T_i(c) \in \arm$ is a weak equivalence in $\M$. There are three known techniques to detect this  when $\M$ is cofibrantly generated and we outline them briefly hereafter.
\begin{enumerate}
\item $f$ will be a weak equivalence if we can show that $f$ has the RLP with respect to the set of generating cofibrations $\I$ of $\M$. In that case $f$ is in fact a trivial fibration. 
\item  If $f$ is a map between fibrant objects, $f$ will be a weak equivalence if we can show that it satisfies the  \emph{Homotopy Extension Lifting Property (HELP)} with respect to the elements of $\I$.
\item If $f :X \to Y$, we can use \emph{function complexes}  and check whether  $\Map(C,f): \Map(C,X) \to \Map(C,Y)$ is a weak equivalence of simplicial sets, as $C$ runs through the set of domain and the codomain of maps in $\I$.
\end{enumerate}
With these techniques in mind, it suffices to provide a set of maps $\kbi$ and consider the left Bousfield localization with respect to $\kbi$ such  that by adjunction, being $\kbi$-local as in \cite[Definition 3.2.4]{Hirsch-model-loc} forces our map $f= \T_i(c)$ to be a weak equivalence in $\M$ through one of the above techniques. A $\kbi$-localization in the sense of \cite[Definition 3.3.11]{Hirsch-model-loc} will allow by \emph{the small object argument} of Quillen to perform the second task i.e, produce a functor $\Sim$ with a natural transformation $\Id \to \Sim$  (a fibrant replacement functor) such that $\Sim(c)$ satisfies the generalized Segal conditions.\ \\ 

These techniques are now standard in Homotopy theory and go back to Bousfield-Friedlander \cite{Bous_Fried}, Jardine \cite{Jardine_simpresh}, Joyal \cite{Joyal_simpsh}, Kan \cite{Kan_exinf} and others. Simpson \cite[Chapter 7.7]{Simpson_HTHC} used the first two techniques for Segal categories where the functor $\Sim$ is a \emph{Segalification functor}. Hovey \cite[Proposition 3.2]{Hov_stable} used the third technique for spectra, where the functor $\Sim$ is weakly equivalent to the \emph{$\Omega$-spectrification} as in Bousfield-Friedlander \cite{Bous_Fried}.\ \\

In this paper we use the first two techniques like Simpson to prove Theorem \ref{local-inj-theory} and Theorem \ref{local-proj-theory}. As a consequence we get, among other things, a generalized \emph{Segalification functor} that encompasses the \emph{$\Omega$-spectrification}. Hovey's method can also be used in our context and we will explain later the connection between his approach and ours.
We will produce two sets of maps $\kbi_{inj}$ and $\kbi_{proj}$ in $\arm$ that will be used in the small object argument to \emph{force and to detect}, by adjunction, the RLP and the HELP against elements of $\Iam$. The set $\kbi_{inj}$ is suitable for an injective Quillen-Segal theory and $\kbi_{proj}$ is for  a projective Quillen-Segal theory.



\subsection{Lifting properties and detection of the Segal conditions}
\subsubsection{Homotopy theory of the arrow category}
Let's review briefly the \textbf{injective}  and the \textbf{projective} model structure on the category $\arm$ that can be found in Hovey \cite{Hovey_Arr}. Let $\Un=[0 \to 1]$ be a the walking morphism category. Then $\Ar(\M)$ is the functor category $\Hom(\Un, \M)$ and these model structures are special cases of Reedy model structures (see \cite{Hov-model}). Indeed, one can consider the category $\Un=[0 \to 1]$ as inverse category, and in that case the Reedy structure is the injective model structure. We can also consider $\Un=[0 \to 1]$ as a direct category and  we get the projective model structure (see \cite{Hovey_Arr},\cite{Hov-model} for details).  It's important to observe $\M$ need not to be a cofibrantly generated to get the Reedy model structure. \ \\

Given two objects of $\Ar(\M)$, $f: X_0 \to X_1$ and  $g: Y_0 \to Y_1$, a map $\alpha: f \to g$ in $\Ar(\M)$ consists of two morphisms $\alpha_i : X_i \to Y_i$ such that we have a commutative square  in $\M$ :
 \[
 \xy
(0,18)*+{X_0}="W";
(0,0)*+{ X_1}="X";
(30,0)*+{Y_1}="Y";
(30,18)*+{Y_0}="E";
{\ar@{->}^-{\alpha_1}"X";"Y"};
{\ar@{->}^-{f}"W";"X"};
{\ar@{->}^-{\alpha_0}"W";"E"};
{\ar@{->}^-{g}"E";"Y"};
\endxy
\]

\begin{df}
Let $\alpha: f \to g$ be a map in $\Ar(\M)$. With the previous notation we will say that:
\begin{enumerate}
\item $\alpha$ is a \textbf{injective cofibration} if $\alpha_0$ and $\alpha_1$ are cofibrations in $\M$. 
\item $\alpha$ is a  \textbf{projective  cofibration} if
\begin{itemize}
\item $\alpha_0$ is a cofibration  in $\M$ and 
\item  the unique induced morphism $ X_1 \cup^{X_0} Y_0 \to Y_1$ is a  cofibration in $\M$.
\end{itemize} 
\item $\alpha$ is a \textbf{level-wise weak equivalence} (resp. \textbf{level-wise fibration}) if $\alpha_0$ and $\alpha_1$ are weak equivalences (resp. fibrations) in $\M$.
\item $\alpha$ is an \textbf{injective fibration} if $\alpha_1: X_1 \to Y_1$ is a fibration and if the induced map : 
$$ X_0 \to X_1 \times_{Y_1} Y_0 $$
is a fibration. 
\end{enumerate}
\end{df}
\newpage
\begin{thm}\label{hovey-arr-thm}
Let $\M$ be a model category.  Then with the previous definitions, the following hold.
\begin{enumerate}
\item The three classes of injective fibrations, injective cofibrations and level-wise weak equivalences determine a model structure on $\Ar(\M)$ called \textbf{the injective model structure}.
\item The three classes of projective fibrations, projective cofibrations and level-wise weak equivalences determine a model structure on $\Ar(\M)$ called \textbf{the injective model structure}.
\item If $\M$ is cofibrantly generated (resp combinatorial) then so are the injective and projective model structures.
\end{enumerate}
\end{thm}
\begin{proof}
See Hovey \cite{Hovey_Arr}.
\end{proof}
\subsubsection{Localizing sets}

We borrow here some notation in Hovey \cite{Hovey_Arr}.
\begin{nota}\label{nota-alpha}
Let $\M$ be a cofibrantly generated model category with $\Iam$  and $\Jam$ the respective sets of generating cofibrations and trivial cofibrations. 
\begin{enumerate}
\item For any morphism $ s : A \to B$ of $\M$ we will denote by $\alpha_s: s \to \Id_B$ the map in $\Ar(\M)$ corresponding to the following commutative square.
 \[
 \xy
(0,18)*+{A}="W";
(0,0)*+{B}="X";
(30,0)*+{B}="Y";
(30,18)*+{B}="E";
{\ar@{->}^-{\Id}"X";"Y"};
{\ar@{->}^-{s}"W";"X"};
{\ar@{->}^-{s}"W";"E"};
{\ar@{->}^-{\Id}"E";"Y"};
\endxy
\]

\item We will denote by $\alpha_{\Iam}$ the set of all $\alpha_i$ for $i \in \Iam$, that is, $\alpha_{\Iam}= \{\alpha_i\}_{i \in \Iam}$.
\item Similarly will denote by  $\alpha_{\Jam}$ the set of all $\alpha_j$ for $j \in \Jam$ : $\alpha_{\Jam}= \{\alpha_j\}_{j \in \Jam}$.

\end{enumerate}
\end{nota}

\begin{df}\label{univ-inj-set}
Define the \textbf{localizing injective set} as the set $\alpha_{\Iam}$.
\end{df}

\begin{prop}
Let $s: A \to B$ be a morphism of $\M$ and let $\Id_{\emptyset}$ be the identity morphism of  the initial object of $\M$
\begin{enumerate}
\item If  $s: A \to B$ is a weak equivalence in $\M$ then $\alpha_s: s \to \Id_B$ is a level-wise weak equivalence in $\Ar(\M)$.
\item If  $s: A \to B$ is a (trivial) cofibration  in $\M$ then $\alpha_s: s \to \Id_B$ is a level-wise (trivial) cofibration in $\Ar(\M)$. 
\item If $A$ is cofibrant in $\M$, and $s$ is a cofibration,  then the unique map $\Id_{\emptyset} \to s$ is a level-wise cofibration in $\arm$.
\end{enumerate}
\end{prop}

\begin{proof}
Like any isomorphism of $\M$, the identity $\Id_B$  is a simultaneously a cofibration, a trivial cofibration, and a weak equivalence.  The maps $s$ and $\Id_B$ are  the components defining the map $\alpha_s$.  Therefore $\alpha_s$ is a level-wise weak equivalence (resp. (co)fibration) if and only if $s$ is a weak equivalence (resp. (co) fibration).

The third assertion is clear because $B$ is also cofibrant since by composition the unique map $\emptyset \to A \xrw{s} B$ is a cofibration.
\end{proof}

\paragraph{Relative cylinder object} For a cofibration $s: A \to B$ the map $\alpha_s$ is related to a \emph{relative cylinder object} for the cofibration $s$ (see \cite[Chapter 9.4]{Simpson_HTHC}). Consider the pushout of $s$ along itself with the pushout data $B \xlw{s} A \xrw{s} B$ and write $B \cup^{A} B$ for the pushout-object. From the commutative square that defines the map $\alpha_s:s \xrw{(s,\Id_B)} \Id_B$, we get a unique map $\phi:B \cup^{A} B \to  B$ using the universal property of the pushout. In particular everything  commutes below.

 \[
 \xy
(0,20)*+{A}="W";
(0,0)*+{B}="X";
(40,0)*+{B}="Y";
(40,20)*+{B}="E";
(20,6)*+{B \cup^{A} B}="Z";
{\ar@{->}_-{\Id}"X";"Y"};
{\ar@{->}^-{s}"W";"X"};
{\ar@{->}^-{s}"W";"E"};
{\ar@{->}^-{\Id}"E";"Y"};
{\ar@{->}^-{\epsilon_0}"X";"Z"};
{\ar@{->}^-{\phi}"Z";"Y"};
{\ar@{->}^-{\epsilon_1}"E";"Z"};
\endxy
\]

Now use the axiom of the model category $\M$ to factor the map $\phi$ as a cofibration followed by a trivial fibration:
$$B \cup^{A} B \xhrw{\delta} E_s  \xtwoheadrightarrow[\sim]{p} B.$$
This diagram determines a \emph{relative cylinder object} for the cofibration $s : A \to B$. 
The maps $\epsilon_1$ and $\epsilon_2$ are cofibrations because cofibration are closed under cobase change; and by composition $\delta \circ \epsilon_0$ and $\delta \circ \epsilon_1$ are also cofibrations. Moreover by $3$-for-$2$ with respect to the equality $p \circ [\delta\circ \epsilon_i] = \Id_B$, they are also weak equivalences. It turns out that  $\delta \circ \epsilon_1$ and $\delta \circ \epsilon_2$ are trivial cofibrations.

\begin{nota}\label{nota-zeta}
Let $s: A \to B$ be a cofibration as previously. 
\begin{enumerate}
\item Denote by $j_0$ the composite $\delta \circ \epsilon_0$ will write $j_1= \delta \circ \epsilon_1$. Both are trivial cofibrations.
\item Let $\zeta_s: s \xrw{(s,j_0)} j_1$ be the map in $\arm$ given by the commutative square:
 \[
 \xy
(0,18)*+{A}="W";
(0,0)*+{B}="X";
(30,0)*+{E_s}="Y";
(30,18)*+{B}="E";
{\ar@{->}^-{j_0}_{\sim}"X";"Y"};
{\ar@{->}^-{s}"W";"X"};
{\ar@{->}^-{s}"W";"E"};
{\ar@{->}^-{j_1}_{\sim}"E";"Y"};
\endxy
\]
\item Let $\ell_s: j_1 \xrw{(\Id_B, p)} \Id_B$ be the level wise weak equivalence in $\arm$ given by the commutative square: 
 \[
 \xy
(0,18)*+{B}="W";
(0,0)*+{E_s}="X";
(30,0)*+{B}="Y";
(30,18)*+{B}="E";
{\ar@{->>}^-{p}_{\sim}"X";"Y"};
{\ar@{->}^-{j_1}"W";"X"};
{\ar@{->}^-{\Id_B}"W";"E"};
{\ar@{->}^-{\Id_B}"E";"Y"};
\endxy
\]
\end{enumerate}
\end{nota}

\begin{df}\label{univ-proj-set}
Define the \textbf{localizing projective set} as the set $\zeta_{\Iam}= \{ \zeta_i \}_{i \in \Iam}$.
\end{df}

\begin{prop}\label{alph-zeta}
With the previous notation the following hold. 
\begin{enumerate}
\item For every cofibration $s$, $\zeta_s$ is a cofibration in $\armpj$
\item For every $s$ we have $\alpha_s= l_s \circ \zeta_s$
\item If $s: A \to B$ is a cofibration with $A$ cofibrant then,  $\ell_s$ is a level-wise weak equivalence between cofibrant objects in $\armpj$ and thus $\armij$. 
\item Let $\Upsilon: \arm \to \C$ be a left Quillen functor with respect to either  $\arm_{inj}$ or $\arm_{proj}$. If $s$ is a cofibration between cofibrant objects, then  $\Upsilon(\ell_s)$ is a weak equivalence in $\C$.
\item Assume that  $\Upsilon: \arm \to \C$ is a left Quillen functor with respect to either  $\arm_{inj}$ or $\arm_{proj}$. Then for any cofibration  $s$ between cofibrant objects, $\Upsilon(\alpha_s)$ is a weak equivalence if and only if $\Upsilon(\zeta_s)$ is a weak equivalence.
\end{enumerate}
\end{prop}

\begin{proof}
Assertion $(1)$ and $(2)$ are  immediate by construction. 
Observe that the cofibrant object in the projective model structure $\arm_{proj}$ are the cofibrations $s : A \to B$ whose domain is cofibrant. Moreover any projective cofibration is  also an injective cofibration, thus $s$ is also a cofibrant object in the injective model structure. Now clearly $\ell_s$ is a level-wise weak equivalence. This gives Assertion $(3)$. \ \\

Assertion $(4)$ is a consequence of Ken Brown's Lemma \cite[Lemma 1.1.12 ]{Hov-model}. Indeed, $s$ and $\Id_B$ are cofibrant and $\ell_s$ is weak equivalence between cofibrant objects, so by Ken Brown's Lemma any left Quillen functor must send it to a weak equivalence between cofibrant objects.

Assertion $(5)$ is a consequence of Assertion $(4)$ and the $3$-for-$2$ property of weak equivalences in the model category $\C$.
\end{proof}

\begin{df}
Let $s:A \to B$ be a cofibration and let  $u$ and $v$ be two elements in $\Hom_{\M}(B, X)$. Say that $f$ and $g$ are \emph{homotopic relative to $s$} if there is a map $h : E_s \to X$ such that $h\circ j_0= u$ and $h \circ j_1= v$.
\end{df}

\begin{rmk}
It's important to observe the following.
\begin{enumerate}
\item Note that the maps $j_0$ and $j_1$ restrict to the same map on $A$. Therefore we have an equality $h \circ (j_0 \circ s) =  h \circ (j_1 \circ s)$.
\item Any map $h: E_s \to X$ defines tautologically a homotopy between $u:=h\circ j_0$ and $v:=h \circ j_1$.
\end{enumerate}

\end{rmk}

\paragraph{Homotopy lifting problem}
\begin{df}
Let $f:X_0 \to X_1$ be an arbitrary morphism and let $s: A \to B$ be a cofibration in $\M$.  Consider a lifting problem of solid arrows defined by $s$ and $f$ through a map $\theta: s \xrw{(\theta_0,\theta_1)} f$  in $\arm$ as follows. 
\[
 \xy
(0,18)*+{A}="W";
(0,0)*+{B}="X";
(30,0)*+{X_1}="Y";
(30,18)*+{X_0}="E";
{\ar@{->}^-{\theta_1}"X";"Y"};
{\ar@{->}^-{s}"W";"X"};
{\ar@{->}^-{\theta_0}"W";"E"};
{\ar@{->}^-{f}"E";"Y"};
{\ar@{-->}^-{}"X";"E"};
\endxy
\]
A solution  \textbf{up-to-homotopy} to this problem consists of a map $r: B \dashrightarrow X_0$ such that:
\begin{enumerate}
\item $r \circ s = \theta_0$ i.e, $r$ solves strictly the upper triangle; and 
\item $f \circ r$ and $\theta_1$ are homotopic relative to $s$, in that, there exists a map $h : E_s \to X_1$ such that $h \circ j_0= \theta_1$ and $h \circ j_1=f \circ r$. In other words $r$ solves the lower triangle up-to a relative homotopy.
\end{enumerate} 
Say that $f$ has the \emph{Homotopy Extension Lifting  Property} with respect to $s$ if there is a solution up-to-homotopy to any lifting problem defined by $s$ and $f$.
\end{df}

\begin{prop}
Let $f : X_0 \to X_1$ be an object in $\arm$ and let $f \to \ast$ be the unique map going to the terminal object. Consider a lifting problem of solid arrows defined by $\zeta_s$ and $f \to \ast$ as follows. 
 \[
 \xy
(0,18)*+{s}="W";
(0,0)*+{j_1}="X";
(30,0)*+{\ast}="Y";
(30,18)*+{f}="E";
{\ar@{->}^-{!}"X";"Y"};
{\ar@{->}_-{\zeta_s=(s,j_0)}"W";"X"};
{\ar@{-->}^-{(\gamma_0, \gamma_1)}"X";"E"};
{\ar@{->}^-{(\theta_0,\theta_1)}"W";"E"};
{\ar@{->}^-{!}"E";"Y"};
\endxy
\]

Then we have an equivalence between the following.
\begin{enumerate}
\item A solution $\gamma=(\gamma_0,\gamma_1): j_1 \to f$ to this lifting problem.
\item An up-to-homotopy solution to the lifting problem in $\M$ defined by $s$ and $f$ by the map $\theta: s \xrw{(\theta_0,\theta_1)} f$.
\end{enumerate}
\end{prop}

\begin{proof}
One simply writes down everything. 
Given a solution $\gamma=(\gamma_0,\gamma_1): j_1 \to f$, we get a solution up-to-homotopy by letting $r=\gamma_0$ and $h=\gamma_1$ and conversely given a solution up-to-homotopy $r: B \to X_0$ and $h: E_s \to X_1$ defines a map $\gamma=(r,h): j_1 \to f$ that is a solution to the lifting problem.  
\end{proof}

\subsubsection{Detecting the Segal conditions}

With the sets of maps $\alpha_{\Iam}$ and $\alpha_{\Jam}$ we are able to detect in $\Ar(\M)$ whether a map $f: X_0 \to X_1$ is a trivial fibration or  fibration by lifting property just like we do with the set $\Iam$ and $\Jam$ in $\M$. Hovey \cite[Proposition 2.2]{Hovey_Arr} showed  that $\alpha_{\Iam}$ is a subset of the generating set of cofibrations in the \textbf{injective model structure} on $\Ar(\M)$.

\begin{lem}\label{lifting-lem}
Let $\M$ be a  model category and let $S$ be a set of maps. 
\begin{enumerate}
\item A morphism $f: X_0 \to X_1$ has the RLP in $\M$  with respect to $S$, if and only if the unique map  $f \to \ast$ going to the terminal object has the RLP in $\Ar(\M)$ with respect to the set $\alpha_S = \{\alpha_s\}_{s \in S}$.
\item More generally, if $g: Y_0 \to Y_1$, then a map $\beta: f \to g$ in $\Ar(\M)$ has the RLP with respect to $\alpha_S = \{\alpha_s\}_{s \in S}$, if and only if the induced map 
$ X_0 \to X_1 \times_{Y_1} Y_0 $ has the RLP in $\M$ with respect to the set $S$.
\item In particular a morphism $f: X_0 \to X_1$ is a trivial fibration , if and only if, $f \to \ast$ has the RLP with respect to the set $\alpha_{\Iam}$.
\end{enumerate}
\end{lem}

The last  assertion is also mentioned in Rosicky \cite[Proof of Proposition 3.3]{Rosicky_combinatorial}.

\begin{proof}

It suffices to prove Assertion $(2)$. 
Consider a lifting problem defined by $\alpha_s$ and $\beta$:
 \[
 \xy
(0,18)*+{s}="W";
(0,0)*+{\Id_B}="X";
(30,0)*+{g}="Y";
(30,18)*+{f}="E";
{\ar@{->}^-{(\gamma_0, \gamma_1)}"X";"Y"};
{\ar@{->}_-{(s,\Id_B)}"W";"X"};
{\ar@{->}^-{(\theta_0,\theta_1)}"W";"E"};
{\ar@{->}^-{(\beta_0,\beta_1)}"E";"Y"};
\endxy
\]

This commutative square corresponds to a commutative cube in $\M$. And if we unfold it, we see that the maps $\gamma_0 : B \to Y_0$ and $\theta_1 : B \to X_1$ determine a commutative square with the maps $g$ and $\beta_1$; in that $\beta_1 \circ \theta_1 = g \circ \gamma_0$. The universal property of the pullback square gives a unique map $\delta: B \to X_1 \times_{Y_1} Y_0$ that makes everything compatible. In particular $\theta_1$ factors through $\delta$. If we put this in the original cube, everything below commutes.  
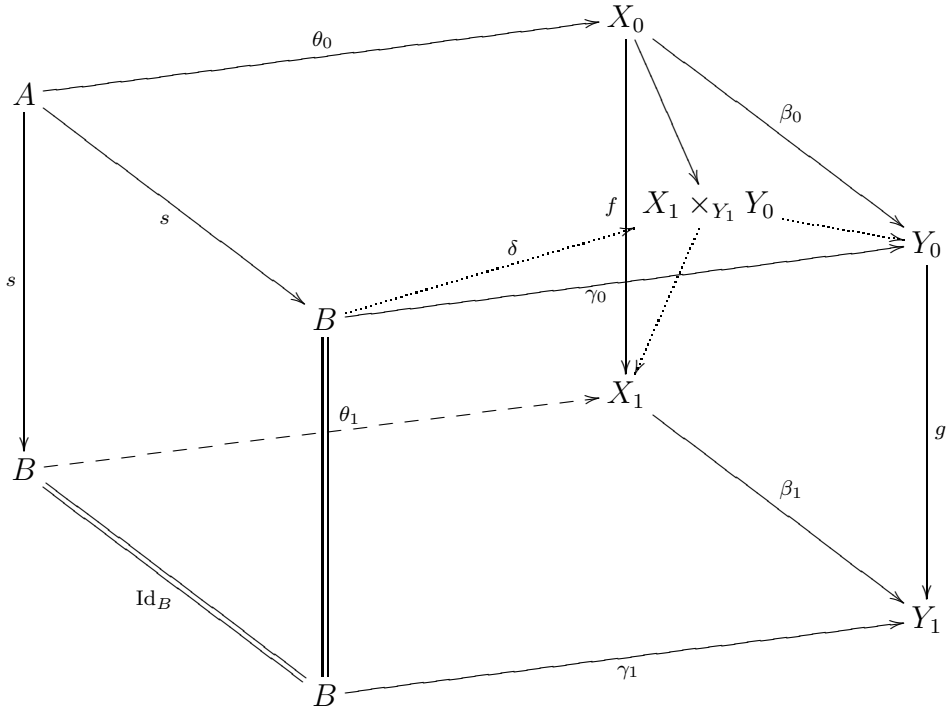
\[
\xy
(-60,20)*+{A}="A";
(20,30)*+{X_0}="B";
(-20,-10)*+{B}="C";
(60,0)*+{Y_0}="D";
{\ar@{->}^{\theta_0}"A";"B"};
{\ar@{->}_{s}"A";"C"};
{\ar@{->}^{\beta_0}"B";"D"};
{\ar@{->}_{\gamma_0~~~~~~~}"C";"D"};
(-60,-30)*+{B}="X";
(20,-20)*+{X_1}="Y";
(-20,-60)*+{B}="Z";
(60,-50)*+{Y_1}="W";
{\ar@{-->}^{\quad \quad \theta_1}"X";"Y"};
{\ar@{->}_{s}"A";"X"};
{\ar@{->}^{\beta_1}"Y";"W"};
{\ar@{=}^{}"C";"Z"};
{\ar@{->}_-{f}"B";"Y"};
{\ar@{->}^{g}"D";"W"};
{\ar@{=}_{\Id_B}"X";"Z"};
{\ar@{->}_{\gamma_1}"Z";"W"};
(20,30)+(11,-25)*+{X_1 \times_{Y_1} Y_0 }="E";
{\ar@{->}^{}"B";"E"};
{\ar@{.>}^{}"E";"D"};
{\ar@{.>}^{}"E";"Y"};
{\ar@{.>}^{\delta}"C";"E"};
\endxy
\]

As we see, we get a commutative square that corresponds to a lifting problem defined by the map $s:A \to B$ and the map $ X_0 \to X_1 \times_{Y_1} Y_0 $:
 \[
 \xy
(0,18)*+{A}="W";
(0,0)*+{B}="X";
(30,0)*+{X_1 \times_{Y_1} Y_0}="Y";
(30,18)*+{X_0}="E";
{\ar@{->}^-{\delta}"X";"Y"};
{\ar@{->}_-{s}"W";"X"};
{\ar@{->}^-{\theta_0}"W";"E"};
{\ar@{->}^-{}"E";"Y"};
\endxy
\]

Now it suffices to observe that this lifting problem has a solution if and only if our original lifting problem has a solution. Indeed if $\chi :  B \to X_0$ is a solution to the previous lifting problem, then the map $(\chi,\theta_1): \Id_B \to f$ in $\Ar(\M)$ determines a solution to the original lifting problem. And reciprocally given a solution to the original lifting problem $(\chi,\theta_1): \Id_B \to f$, then the component $\chi: B \to X_0$ is a solution to the lifting problem defined by $s$ and $ X_0 \to X_1 \times_{Y_1} Y_0 $.
\end{proof}

\paragraph{HELP Lemma} There are various versions of the HELP Lemma and we refer the reader to Boardman-Vogt \cite{board_vogt}, May \cite{May_dual_white}, Simpson \cite{Simpson_HTHC}, Vogt \cite{Vogt_help} and the references therein. We will use the  version used by Simpson \cite[Lemma 7.5.1]{Simpson_HTHC}. This version is  for tractable model categories but the argument remains exactly the same for a cofibrantly generated model category where the cofibrations are generated by a set of maps between cofibrant objects.

\begin{lem}
Let $\M$ be a tractable model category with a set of cofibration $\Iam$. Then any morphism $f$ between fibrant objects and that satisfies the HELP with respect to every element in $\Iam$ is a weak equivalence. 
\end{lem}

Equivalently we have:
\begin{lem}\label{lifting-lem-proj}
Let $\M$ be a tractable model category with a set of cofibration $\Iam$. If $f$ is a morphism in  $\M$ such that the unique map $f \to \ast$ is a level fibration in $\arm$ and  has the RLP with respect to every element in $\zeta_{\Iam}$,  then $f$ is a weak equivalence.
\end{lem}
\begin{proof}
See \cite[Lemma 7.5.1]{Simpson_HTHC}.
\end{proof}

\subsection{Localizing a Quillen-Segal theory}
We now give the proof of Theorem \ref{local-inj-theory} and Theorem \ref{local-proj-theory}. 

\begin{df}
Let $\T=\{\T_i: \C \to \Ar(\M) \}_{i \in I}$ be an $\M$-valued Quillen-Segal theory on $\C$ such that for every $i \in I$ we have a left adjoint $\Upsilon_i: \arm \to \C$. 
\begin{itemize}
\item Define the injective localizing set for $\T$ as:
$$\kbi_{inj}(\T):= \coprod_{i \in I} \Upsilon_i(\alpha_{\Iam}) $$
\item Define the projective localizing set for $\T$ as:
$$\kbi_{proj}(\T):= \coprod_{i \in I} \Upsilon_i(\zeta_{\Iam}) $$
\end{itemize}
\end{df}

\begin{prop}
With the previous notation we have the following. 
\begin{enumerate}
\item If $\T$ is an injective theory then every element in $\kbi_{inj}$ is a cofibration in $\C$
\item If $\T$ is a projective theory then every element in $\kbi_{proj}$ is a cofibration in $\C$
\end{enumerate}
\end{prop}

\begin{proof}
Indeed, every element in $\alpha_{\Iam}$ (resp. $\zeta_{\Iam}$) is a injective (reps. projective) cofibration in $\arm$; therefore its image by the left Quillen functor $\Upsilon_i$ is a cofibration in $\C$.
\end{proof}

\begin{proof}[Proof of Theorem \ref{local-inj-theory}]
The hypotheses of the theorem allow us to consider the left Bousfield localization  with respect to the set $\kbi_{inj}$ (see Hirschhorn \cite{Hirsch-model-loc} and Barwick \cite{Barwick_localization}). Then thanks to the previous proposition every element in $\kbi_{inj}$ is a cofibration and a weak equivalence; thus a new trivial cofibrations. It follows that for any fibrant object $c$ in this new model structure the unique map $c \to \ast$ in $\C$ must be $\Upsilon_i(\alpha_{\Iam})$-injective for every $i \in I$. And by adjunction we find that  the map $\T_i(c) \to \ast$ in $\arm$ is $\alpha_{\Iam}$-injective and $\T_i(c)$ is a trivial fibration by Lemma \ref{lifting-lem}. In particular for every $i \in I$, $\T_i(c)$ is a weak equivalence (between fibrant objects); thus $c$ satisfies the generalized Segal conditions. 
\end{proof}

\begin{proof}[Proof of Theorem \ref{local-proj-theory}]
The proof is the same as in the injective case. 
Consider the left Bousfield localization  with respect to the set $\kbi_{proj}$. Then thanks to the previous proposition every element in $\kbi_{proj}$ is a cofibration and a weak equivalence; thus a new trivial cofibrations. It follows that for any fibrant object $c$ in this new model structure the unique map $c \to \ast$ in $\C$ must be $\Upsilon_i(\zeta_{\Iam})$-injective for every $i \in I$. And by adjunction we find that  the map $\T_i(c) \to \ast$ in $\arm$ is $\zeta_{\Iam}$-injective and $\T_i(c)$ is a weak equivalence by the HELP Lemma \ref{lifting-lem-proj}. In particular for every $i \in I$, $\T_i(c)$ is a weak equivalence (between fibrant objects); thus $c$ satisfies the generalized Segal conditions. 
\end{proof}

\section{Quillen-Segal algebras}

\subsection{Comma categories and Quillen-Segal objects}
\subsubsection{Comma categories}
We recall here some definitions and properties on comma categories. These results are well known in Category Theory (see for example \cite{borceux_1}, \cite{Mac}, \cite{Stanculescu_multi}). We include some of them for completeness.
\begin{df}
Let $\Ub:\ag \to \M$ be a functor. The \emph{comma category} $(\M \downarrow \Ub)$ is the category described as follows.
\begin{enumerate}
\item Objects are triples $[\F] = [\F_0, \F_{1}, \pi_{\F}: \F_0 \to \Ub(\F_1)] \in \M \times \ag \times \Ar(\M)$
\item Given two objects $[\F] = [\F_0, \F_{1}, \pi_{\F}]$ and $[\G] = [\G_0, \G_{1}, \pi_{\G}]$, a map $\sigma : [\F] \to [\G]$ is a pair $\sigma=[\sigma_0,\sigma_1] \in \Hom_{\M}(\F_0, \G_0) \times \Hom_{\ag}(\F_1, \G_1)$ such that we have a commutative square in $\M$:
\[
\xy
(0,18)*+{\F_0}="W";
(0,0)*+{\Ub(\F_1)}="X";
(30,0)*+{\Ub(\G_1)}="Y";
(30,18)*+{\G_0}="E";
{\ar@{->}^-{\Ub(\sigma_1)}"X";"Y"};
{\ar@{->}^-{\pi_{\F}}"W";"X"};
{\ar@{->}^-{\sigma_0}"W";"E"};
{\ar@{->}^-{\pi_{\G}}"E";"Y"};
\endxy
\]
That is:
$$  \Hom_{(\M\downarrow\Ub)}([\F],[\G])=\{ [\sigma_0, \sigma_1] \in \Hom_{\M}(\F_0,\G_0) \times \Hom_{\ag}(\F_1,\G_1) \quad | \quad \pi_{\G} \circ \sigma_0 = \Ub(\sigma_1)\circ \pi_{\F} \}$$
\end{enumerate}
\ \\
We have three obvious functors :
\begin{itemize}
\item  $\Pi_0: \mdu \to \M $, with $\Pi_0([\F])= \F_0$;
\item  $\Pi_1: \mdu \to \ag $, with  $\Pi_1([\F])= \F_1$;
\item  $\Pi_{\Ar}: \mdu \to \Ar(\M)$, with $\Pi_{\Ar}([\F])= \pi_{\F}$.
\end{itemize}
\end{df}

\begin{nota}
It's sometimes convenient to adopt another notation for $\mdua$ because we don't see the category $\ag$. So we will also write $\mua := \mdua$.
\end{nota}
\subsubsection{Quillen-Segal algebras and objects}
We are now able to formulate the definition of Quillen-Segal algebras and  Quillen-Segal objects.. 
\begin{df}
Let $\M$ be a model category and let $\Ub: \ag \to \M$ be a functor. \ \\
An object \emph{$[\F]=[\F_0, \F_{1}, \pi_{\F}] \in \mdu $ is a Quillen-Segal $\Ub$-object} if \emph{$[\F]$ satisfies the Segal condition}, that is if  $\pi_{\F}: \F_0 \to \Ub(\F_1)$ is a weak equivalence in $\M$. 
\end{df}

\begin{df}
Let $\M$ be a model category and let $\O$ be a monad on $\M$ or an operad enriched over $\M$.

A Quillen-Segal $\O$-algebra is a Quillen-Segal $\Ub$-object $[\F]=[\F_0, \F_{1}, \pi_{\F}] \in \mdu $  for the forgetful functor $\Ub: \oalg(\M) \to \M$.
\end{df}

In the previous definition, $\O$ is not restricted to be a monad or an operad. Indeed $\O$ can be a PROP, properad, cyclic operad etc. These objects can be viewed themselves as algebras over some  monad $\O'$ defined on another category $\M'$ (see \cite{Markl_prop}, \cite{Yau_HDA}, \cite{Vallette_properad}).  Moreover if we fix such `operator-object' $\O$, its category of algebras with coefficient in $\M$, is equivalent to a category of a monad  defined on  another category $\M'$,  where $\M'$ is in many cases  a product of copies of $\M$ (see for example \cite{Robertson_Hack}) or more generally a diagram category $\Hom(\C,\M)$.
\begin{rmk}
\begin{enumerate} 

\item The previous definition makes sense in any category with weak equivalences. 
\item  Let $\Ar(\W) \hookrightarrow \Ar(\M)$ be the subcategory of weak equivalences in $\M$. Then the category of Quillen-Segal objects is precisely the full subcategory of $\mdu$ of objects that are sends to $\Ar(\W)$ through the functor $\Pi_{\Ar} : \mdu \to \Ar(\M)$
\end{enumerate}

\end{rmk}

\section{Properties of comma constructions}
We give here the required definitions and properties on comma categories that are necessary to deploy the homotopy theory on $\mdua$. Most of the material is classical in category theory and many results are straightforward.
\subsection{Embedding} 

A direct consequence of the definition is the following:
\begin{prop}\label{embed-prop}
We have a functor $\iota : \ag \hrw \mdu $ defined as follows. 
\begin{enumerate}
\item For $\Pa \in \ag$,  $\iota(\Pa)=[\Ub(\Pa), \Pa, \Id_{\Ub(\Pa)}]$. 
\item The functor $\iota$ takes a map $\theta : \Pa \to \Qa$ of $\ag$ to the map $\iota(\theta)=[\Ub(\theta), \theta] $.
\item The composite $\Pi_1 \circ \iota$ is the identity. In other words $\ag$ is a retract of $\mdu$ in the category $\Cat$ of small categories. 
\item The functor $\iota$ is injective on objects and we have an isomorphism of hom-sets
$$\Hom_{\ag}(\Pa,\Qa) \cong \Hom_{\mdu}(\iota(\Pa), \iota(\Qa)) $$
In particular the functor $\iota: \ag \to \mdu$ exhibits $\ag$ as a full subcategory of $\mdu$ 
\end{enumerate}

\end{prop}

\begin{proof}
The fact that $\iota$ is injective on objects is clear since $\Pa$ appears alone in the triple $[\Ub(\Pa), \Pa, \Id_{\Ub(\Pa)}]$. \ \\
On morphisms we have $\iota(\theta_1) = \iota(\theta_2) \Leftrightarrow [\Ub(\theta_0), \theta_0] = [\Ub(\theta_1), \theta_1]  \Leftrightarrow \theta_1 = \theta_2$ and  $\Ub(\theta_0)=\Ub(\theta_1)$; which means that $\iota$ is also injective on morphisms (= $\iota$ is  faithful).\ \\

Let's now show that $\iota$ is also surjective on morphisms (= $\iota$ is  full). 
If $\sigma= [\sigma_1, \sigma_1]$ is a map in $\mdu$ from $\iota(\Pa)$ to $\iota(\Qa)$, we have:
$$  \Id_{\Ub(\Qa)} \circ \sigma_0 = \Ub(\sigma_1) \circ  \Id_{\Ub(\Pa)} \Leftrightarrow \sigma_0 = \Ub(\sigma_1) \Leftrightarrow \sigma = [\Ub(\sigma_1), \sigma_1] = \iota(\sigma_1)$$ with $\sigma_1: \Pa \to \Qa$.
\end{proof}
\subsection{Arrow-category as comma category.}
Recall that $\Un= [0 \to 1]$  is the \emph{walking-morphism category} and therefore $\arm$ is the functor category $\Hom(\Un,\M)$. If we put together this fact and the previous definition of comma category we get the following obvious results. \ \\
\begin{prop}\label{arr-is-comma}
Let $\Ub : \M \to \M$ be the identity functor. Then the following hold.
\begin{enumerate}
\item The category of morphisms  $\Ar(\M)$ is isomorphic to the comma category $\mum:= \mdua$. And the functor $\Pi_{\Ar}: \mdua \to \Ar(\M)$ is an isomorphism.
\item The category of  morphisms  $\Ar(\M)$ is isomorphic to the functor category $\Hom(\Un, \M)$.  
\item The functor $\Pi_0 : \mum \to \M$ is the source-functor and corresponds to the evaluation functor $\Ev_0: \Hom(\Un, \M) \to \M$.
\item The functor $\Pi_1 : \mum \to \M$ is the target-functor and corresponds to the evaluation functor $\Ev_1: \Hom(\Un, \M) \to \M$.

\end{enumerate}

\hfill $\qed$
\end{prop}

If we apply Proposition \ref{embed-prop} for the functor $\Ub= \Id_{\M}$ then $\mum$ is the category of $\Ar(\M)$. In particular one gets the following results that can also be found in Hovey  \cite[Lemma 1.1]{Hovey_Arr}.

\begin{prop}\label{m-in-arm}
Let $\M$ be any category. Then the following hold.

\begin{enumerate}
\item We have an embedding $\iota_{\M} : \M \hrw \Ar(\M)$ that takes an object $X \in \M$ to the identity $\Id_X$ and a morphism $f: X \to X'$ to the map $\Id_X \to \Id_{X'}$ whose components are both equal to $f$. 
\item The functor $\iota_{\M}$ is left adjoint to the functor $\Ev_0: \Ar(\M) \to \M$. We will denote by $L_0= \iota_{\M}$. 
\item If there is an initial object $\emptyset \in \M$, then the functor $\Ev_1: \Ar(\M) \to \M$ has also a left adjoint $L_1: \M \to \Ar(\M)$. 
We have $L_1(X)=\emptyset \to X$. 
\end{enumerate}

\end{prop}
\newpage
\subsection{Adjunctions, limits and colimits}

\subsubsection{Adjunctions}
We give here the adjunctions that will be needed to get the various model structures on our comma category. The constructions and the results that will follow are also classical in category theory. We include the proof for completeness. Recall that we have an embedding $\iota: \ag \to \mdu$ and a functor $\Pi_1: \mdu \to \ag$. Our goal in this section is to establish the following theorem.

\begin{thm}\label{adjunction-thm}
Let $\Ub: \ag \to \M$ be a functor. With the previous notation the following hold. 

\begin{enumerate}
\item The functor $\Pi_1: \mdu \to \ag$ is left adjoint to the embedding $\iota : \ag \to \mdu$. In particular $\ag$ is equivalent to a full reflective subcategory of $\mdu$.
\item The functor $\Pi_1: \mdu \to \ag$ has also a left adjoint $L_1: \ag \to \mdu$ with $L_1(\Pa)= [\emptyset, \Pa, \emptyset \to  \Ub(\Pa)]$, where $\emptyset \to \Ub(\Pa)$ is the unique map.
\item  If furthermore $\Ub$ has a left adjoint $\Fb: \M \to \ag$, then
\begin{enumerate}
\item The functor $\Pi_{\Ar}: \mdu \to \Ar(\M)$  has a left adjoint 
$$\Gamma : \Ar(\M) \to \mdu$$
\item The functor $\Pi_0: \mdu \to \M$ has a left adjoint $\Fb^+ : \M \to \mdu$ such that the composite $\Pi_1 \circ \Fb^+$ is the functor $\Fb$. \\
We have $\Fb^+= \Gamma \circ L_0$. 
\end{enumerate}
\end{enumerate}
\end{thm}

\begin{proof}
Assertion $(1)$ is the content of Proposition \ref{prop-gen-embed}.\ \\
Assertion $(2)$ can be easily verified.\ \\
Assertion $(3)$ is given by Proposition \ref{prop-gamma}  and Proposition \ref{adj-prop-gamma}.

\end{proof}

\begin{prop}\label{prop-gen-embed}
Let $\Ub: \ag \to \M$ be a functor. Then the functor  $\iota : \ag \to \mdu$ is right adjoint to the  functor  $\Pi_1: \mdu \to \ag$. \ \\
In particular $\ag$ is equivalent to a full reflective subcategory of $\mdu$.
\end{prop}

\begin{proof}
Let $[\F]$ be an object of $\mdu$ and $\Pa$ be an object of $\ag$. We wish to show that we have a functorial isomorphism of hom-sets :
$$\Hom_{\ag}(\Pi_1([\F]), \Pa) \cong  \Hom_{\mdu}([\F], \iota(\Pa))$$ 
Recall that $\Pi_1([\F])= \F_1$ and that by definition 
a map $\sigma=[\sigma_0,\sigma_1] : [\F] \to \iota(\Pa)$ consists of two maps $\sigma_0 \in \Hom_{\M}(\F_0, \Ub(\Pa))$ and  $\sigma_1 \in \Hom_{\ag}(\F_1, \Pa)$ such that :
 $$\Id_{\Ub(\Pa)} \circ \sigma_0 = \Ub(\sigma_1) \circ \pi_{\F} \Longleftrightarrow \sigma_0 = \Ub(\sigma_1) \circ \pi_{\F} \Longleftrightarrow \sigma = [\Ub(\sigma_1) \circ \pi_{\F}, \sigma_1].$$ 
 It follows from the above equivalence that the function$$\varphi : \Hom_{\ag}(\F_1, \Pa) \to \Hom_{\mdu}([\F], \iota(\Pa)),$$ defined by $\varphi(\sigma_1)=[\Ub(\sigma_1) \circ \pi_{\F}, \sigma_1]$ is surjective.  The presence of $\sigma_1$ alone in the couple $[\Ub(\sigma_1) \circ \pi_{\F}, \sigma_1]$ implies that $\varphi$ is also injective. Therefore  $\varphi$ is an isomorphism as required. The functoriality in both variables $[\F]$ and $\Pa$ is clear. 
\end{proof}

\begin{prop}\label{prop-unit-adj}
With the previous notation the following hold.
\begin{enumerate}
\item For any $[\F] \in \mdu$, the unit  $\eta_1([\F]): [\F] \to \iota(\Pi_1([\F]))$ of the previous adjunction is given by the pair $[\pi_{\F}, \Id_{\F_1}]$. 
\item For any $\Pa \in \ag$ the counit $\varepsilon_1(\Pa): \Pi_1(\iota(\Pa)) \to \Pa$ of the adjunction is the identity.
\end{enumerate}

\end{prop}

\begin{proof}
Clear.
\end{proof}

\begin{prop}\label{prop-gamma}
Let $\Ub : \ag \to \M$ be a functor that posses a left adjoint $\Fb: \M \to \ag$. Let $\eta : \Id_{\M} \to \Ub\Fb$ be the unit of this adjunction.
Then we have a functor $$\Gamma : \Ar(\M) \to \mdu$$
 defined as follows.
\begin{enumerate}
\item If $f : X_0 \to X_1$, then $\Gamma(f) = [X_0, \Fb(X_1), \eta_{X_1}\circ f]$ with $\eta_{X_1} : X_1 \to \Ub\Fb(X_1)$
\item If $g : Y_0 \to Y_1$ and  $\alpha=(\alpha_0,\alpha_1) : f \to g$ is a map in $\Ar(\M)$, then $\Gamma(\alpha)= [\alpha_0, \Fb(\alpha_1)]$.
\end{enumerate}
In particular
\begin{itemize}
\item $\Pi_0(\Gamma(f))=X_0$, $\Pi_0(\Gamma(\alpha))=\alpha_0$;
\item $\Pi_1(\Gamma(f))= \Fb (X_1)$, $\Pi_1(\Gamma(\alpha))= \Fb (\alpha_1)$
\item  $\Pi_{\Ar}(\Gamma(f))= \eta_{X_1}\circ f : X_0 \to X_1 \to \Ub\Fb(X_1)$
\end{itemize}

\end{prop}

\begin{proof}
For the second assertion  it suffices to show that we have an equality $$\eta_{Y_1}\circ g \circ \alpha_0 = \Ub\Fb(\alpha_1) \circ \eta_{X_1}\circ f .$$ 
We get this equality by observing that the naturality of $\eta : \Id_{\M} \to \Ub\Fb$ and part of $\alpha$ being a map in $\arm$ imply that all three squares below are commutative.

\[
\xy
(0,18)*+{X_0}="W";
(0,0)*+{\Ub\Fb(X_0)}="X";
(30,0)*+{\Ub\Fb(X_1)}="Y";
(30,18)*+{X_1}="E";
{\ar@{->}^-{\Ub\Fb(f)}"X";"Y"};
{\ar@{->}^-{\eta_{X_0}}"W";"X"};
{\ar@{->}^-{f}"W";"E"};
{\ar@{->}^-{\eta_{X_1}}"E";"Y"};
\endxy
\quad \quad 
\xy
(0,18)*+{X_1}="W";
(0,0)*+{\Ub\Fb(X_1)}="X";
(30,0)*+{\Ub\Fb(Y_1)}="Y";
(30,18)*+{Y_1}="E";
{\ar@{->}^-{\Ub\Fb(\alpha_1)}"X";"Y"};
{\ar@{->}^-{\eta_{X_1}}"W";"X"};
{\ar@{->}^-{\alpha_1}"W";"E"};
{\ar@{->}^-{\eta_{Y_1}}"E";"Y"};
\endxy
\quad \quad 
\xy
(0,18)*+{X_0}="W";
(0,0)*+{Y_0}="X";
(30,0)*+{Y_1}="Y";
(30,18)*+{X_1}="E";
{\ar@{->}^-{g}"X";"Y"};
{\ar@{->}^-{\alpha_0}"W";"X"};
{\ar@{->}^-{f}"W";"E"};
{\ar@{->}^-{\alpha_1}"E";"Y"};
\endxy
\]

And we get:
\begin{equation*}
\begin{split}
 \eta_{Y_1}\circ g \circ \alpha_0 &= \eta_{Y_1}\circ (g \circ \alpha_0)\\ 
 &= \eta_{Y_1}\circ (\alpha_1 \circ f)\\ 
  &= (\eta_{Y_1}\circ \alpha_1) \circ f\\ 
    &= (\Ub\Fb(\alpha_1) \circ \eta_{X_1}) \circ f\\ 
\end{split}
\end{equation*}
\end{proof}

We are now able to establish the following adjunction. 

\begin{prop}\label{adj-prop-gamma}
The functor $\Gamma : \Ar(\M) \to \mdu$ is left adjoint to the functor $\Pi_{\Ar} : \mdu \to \Ar(\M)$.
\end{prop}

\begin{proof}
The proof is tedious but straightforward so we shall give an outline of it.\ \\
Let $f: X_0 \to X_1$ be an object of $\Ar(\M)$ and let $\gcr = \ogcr$ be an object of $\mdu$. We wish to establish that we have a functorial isomorphism of hom-sets:
$$\Hom_{\mdu}(\Gamma(f), \gcr) \cong \Hom_{\Ar(\M)}(f, \pi_{\G})$$

We define a function $\phi: \Hom_{\Ar(\M)}(f, \pi_{\G}) \to \Hom_{\mdu}(\Gamma(f), \gcr)$ as follows. \ \\
\begin{itemize}
\item Let $\alpha: f \to \pig$ a map in $\Ar(\M)$ displayed by the  commutative square:
 \[
 \xy
(0,18)*+{X_0}="W";
(0,0)*+{ X_1}="X";
(30,0)*+{\Ub(\G_1)}="Y";
(30,18)*+{\G_0}="E";
{\ar@{->}^-{\alpha_1}"X";"Y"};
{\ar@{->}^-{f}"W";"X"};
{\ar@{->}^-{\alpha_0}"W";"E"};
{\ar@{->}^-{\pig}"E";"Y"};
\endxy
\]
From the adjunction $\Fb \dashv \Ub$ there is a functorial isomorphism of hom-sets:
$$\varrho : \Hom_{\M}(X_1, \Ub(\G_1))  \xrightarrow{\cong} \Hom_{\ag}(\Fb X_1, \G_1) $$
\item Consider $\varrho(\alpha_1) : \Fb(X_1) \to \G_1$,  the adjoint transpose of $\alpha_1 : X_1 \to \Ub(\G_1)$.
\item  We define $\phi(\alpha)=[\alpha_0, \varrho(\alpha_1)] \in \Hom_{\M}(X_0, \G_0) \times \Hom_{\ag}(\Fb X_1, \G_1)$.
\end{itemize}
 Let's check  that the pair $\phi(\alpha)=[\alpha_0, \varrho(\alpha_1)]$ defines indeed a morphism in $\mdu$, that is, we must prove that we have an equality: $\Ub(\varrho(\alpha_1)) \circ \Pi_{\Ar}\Gamma(f) = \pig \circ \alpha_1.$\\
By definition, $\Pi_{\Ar}\Gamma(f)= \eta_{X_1} \circ f$ and from the adjunction $\Fb \dashv \Ub$, we have: $$\alpha_1= \Ub(\varrho(\alpha_1)) \circ \eta_{X_1}.$$
Putting all together we have a diagram in which everything commutes: 

 \[
 \xy
(0,18)+(0,10)*+{X_0}="W";
(0,0)*+{ X_1}="X";
(0,0)+(20,10)*+{\Ub\Fb(X_1)}="O";
(30,0)+(20,0)*+{\Ub(\G_1)}="Y";
(30,18)+(20,10)*+{\G_0}="E";
{\ar@{->}^-{\alpha_1}"X";"Y"};
{\ar@{->}^-{f}"W";"X"};
{\ar@{->}^-{\alpha_0}"W";"E"};
{\ar@{->}^-{\pig}"E";"Y"};
{\ar@{->}^-{\eta_{X_1}}"X";"O"};
{\ar@{->}^-{\Pi_{\Ar}\Gamma(f)}"W";"O"};
{\ar@{->}^-{\Ub(\varrho(\alpha_1))}"O";"Y"};
\endxy
\]
The inner commutative square gives a map $\Pi_{\Ar}\Gamma(f) \to \pig$  in $\Ar(\M)$. This means that $\phi(\alpha)=[\alpha_0, \varrho(\alpha_1)]$ is a map in $\mdu$. \ \\

This function is clearly an injection since $\varrho$ is an isomorphism. And we have an inverse function $\phi^{-1}$ that takes a map $\sigma=[\sigma_0,\sigma_1]: \Gamma(f) \to \G$ to the map $$\phi^{-1}(\sigma)= [\sigma_0, \varrho^{-1}(\sigma_1)]: f \to \pig.$$
\end{proof}

Let's now put together the various adjunctions and see how they interact. 
\begin{prop}\label{all-adjunction}
Let  $\Ub: \ag \to \M$ be a functor. Then the following hold.
\begin{enumerate}
\item We have a commutative diagram.
\[
 \xy
(0,18)*+{\ag}="W";
(0,0)*+{ \M}="X";
(30,0)*+{\Ar(\M)}="Y";
(30,18)*+{\mdu}="E";
{\ar@{<-}^-{\Ev_0}"X";"Y"};
{\ar@{->}_-{\Pi_0}"E";"X"};
{\ar@{->}^-{\Ub}"W";"X"};
{\ar@{->}^-{\iota}"W";"E"};
{\ar@{->}^-{\Pi_{\Ar}}"E";"Y"};
\endxy
\]
\item If moreover  $\Ub$ has a left adjoint $\Fb: \M \to \ag$ then we have a commutative diagram of left adjoints functors.
\[
 \xy
(0,18)*+{\ag}="W";
(0,0)*+{ \M}="X";
(30,0)*+{\Ar(\M)}="Y";
(30,18)*+{\mdu}="E";
{\ar@{->}^-{L_0}"X";"Y"};
{\ar@{->}^-{\Fb^+}"X";"E"};
{\ar@{<-}^-{\Fb}"W";"X"};
{\ar@{<-}^-{\Pi_1}"W";"E"};
{\ar@{<-}^-{\Gamma}"E";"Y"};
\endxy
\]
With $\Fb^+= \Gamma \circ L_0$.
\end{enumerate}
\end{prop}

\subsubsection{Limits and colimits}
Given a diagram $D: \J \to \mua$ we have three induced diagrams:
\begin{itemize}
\item $D_0=\Pi_0(D): \J \to \M$ 
\item $D_1=\Pi_1(D): \J \to \ag$ 
\item $\pi_D=\Piar(D):\J \to \Ar(\M).$
\end{itemize}
For every $j \in \J$ we have $D(j)= [D_0(j), D_0(j), \pi_D(j)]$. Our goal here is to see how we compute limits and colimits in the comma category $\mua:=\mdua$.
\newpage

\begin{prop}\label{prop-limit-colimit}
Let $\Ub: \ag \leftrightarrows \M: \Fb$ be an adjunction between complete and cocomplete categories. 
Then with the previous notation the following hold. 
\begin{enumerate}
\item There is an induced map in $\M$: $$\pi_{\colim}:\colim D_0 \to \Ub(\colim D_1)$$ 

And the triple $[\colim D_0, \colim D_1, \pi_{\colim}]$ equipped with the obvious maps is the colimit of the diagram  $D: \J \to \mua$.
\item As the functor $\Ub$ preserves limits, there is also an induced map in  $\M$ 
$$\pi_{\lim}: \lim D_0 \to \Ub(\lim D_1)$$
And the triple $[\lim D_0, \lim D_1, \pi_{\lim}]$ equipped with the obvious maps is the limit of the diagram  $D: \J \to \mua$.
\end{enumerate}
The category $\mdu$ is also complete and cocomplete.
\end{prop}

An immediate consequence is that: 
\begin{cor}
\begin{enumerate}
\item The functors $\Pi_0: \mdu \to \M$ and $\Pi_1: \mdu \to \ag$ preserve limits and colimits. In particular they preserve pushouts.
\item The functor $\Piar: \mdu \to \Ar(\M)$ preserves limits.
\end{enumerate}
\end{cor}

In other words the functor $(\Pi_0,\Pi_1): \mdu \to \M \times \ag $ creates limits and colimits (see \cite{Low_heart})

\begin{proof}
For every $j \in \J$ denote by $\epsilon_1(j): D_1(j) \to \colim D_1$ the canonical map in $\ag$  going to the colimit and consider its image $\Ub(\epsilon_1(j)): \Ub(D_1(j)) \to \Ub(\colim D_1).$ Then the maps $$\{ \Ub(\epsilon_1(j)) \circ \pi_D(j) : D_0(j) \to \Ub(\colim D_1) \}_{j \in \J}$$
 determine a natural transformation from $D_0$ to the constant diagram of value $\Ub(\colim D_1)$. Indeed,  for every structure map $l:j \to j'$ in $\J$ we can apply the functor $\Ub$ to the equality $\epsilon_1(j) = \epsilon_1(j') \circ D_1(l)$ and get:
$$\Ub(\epsilon_1(j) )= \Ub(\epsilon_1(j'))  \circ \Ub(D_1(l)).$$
Moreover, we also have a commutative diagram in $\M$:

\[
\xy
(0,18)*+{D_0(j)}="W";
(0,0)*+{\Ub(D_1(j))}="X";
(30,0)*+{\Ub(D_1(j'))}="Y";
(30,0)+(50,0)*+{\Ub(\colim D_1) }="Z";
(30,18)*+{D_0(j')}="E";
{\ar@{->}^-{\Ub(D_1(l))}"X";"Y"};
{\ar@{->}^-{\pi_D(j)}"W";"X"};
{\ar@{->}^-{D_0(l)}"W";"E"};
{\ar@{->}^-{\pi_D(j')}"E";"Y"};
{\ar@{.>}^-{\Ub(\epsilon_1(j'))}"Y";"Z"};
\endxy
\]
Using the universal property of the colimit of $D_0$, we find a unique map  $$\pi_{\colim}:\colim D_0 \to \Ub(\colim D_1)$$
that makes everything compatible.  If we write $\epsilon_0(j) : D_0(j) \to colim D_1$, the canonical map in $\M$ going to the colimit, then the pair $[\epsilon_0(j),\epsilon_1(j)]$ determines a map in $\mdu$: 
$$[D_0(j), D_0(j), \pi_D(j)] \to  [\colim D_0, \colim D_1, \pi_{\colim}].$$

It is tedious but straightforward to check that $[\colim D_0, \colim D_1, \pi_{\colim}]$ equipped with these maps satisfies the universal property of the colimit of the diagram $D$. This proves Assertion $(1)$. \ \\

For Assertion $(2)$ we proceed in a dual manner using the fact that $\Ub$ preserves limits like any right adjoint functor.  Using this, we know that $\Ub(\lim D_1)$ equipped with the obvious maps is the limit of the diagram $\Ub(D_1)$. \ \\
Let $p_0(j): \lim D_0 \to D_0(j)$ be the canonical projection so that for every structure map $l : j \to j'$ we have $p_0(j') = D_0(l) \circ p_0(j)$. Applying $\Ub$ to this equality yields:
$$\Ub(p_0(j')) = \Ub(D_0(l)) \circ \Ub(p_0(j)).$$
Then the maps: $\{ \pi_D(j) \circ p_0(j): \lim D_0 \to \Ub(D_1(j)) \}_{j \in \J} $,
determine a natural transformation from the constant diagram of value $\lim D_0$, to the diagram $\Ub(D_1)$.\\ 
By the universal property of the limit $\Ub(\lim D_1)$ we find a unique map 
$$\pi_{\lim}: \lim D_0 \to \Ub(\lim D_1)$$
that makes everything compatible. In particular for every structure map $l : j \to j'$ we have a commutative diagram:
\[
\xy
(0,18)*+{D_0(j)}="W";
(0,0)*+{\Ub(D_1(j))}="X";
(30,0)*+{\Ub(D_1(j'))}="Y";
(30,18)+(-70,0)*+{\lim D_0 }="Z";
(30,0)+(-70,0)*+{ \Ub(\lim D_1) }="T";
(30,18)*+{D_0(j')}="E";
{\ar@{->}^-{\Ub(D_1(l))}"X";"Y"};
{\ar@{->}^-{\pi_D(j)}"W";"X"};
{\ar@{->}^-{D_0(l)}"W";"E"};
{\ar@{->}^-{\pi_D(j')}"E";"Y"};
{\ar@{.>}^-{p_0(j)}"Z";"W"};
{\ar@{.>}^-{\Ub(p_1(j))}"T";"X"};
{\ar@{.>}^-{\pi_{\lim}}"Z";"T"};
\endxy
\]
It's not hard to check that the object $[\lim D_0, \lim D_1, \pi_{\lim}]$ equipped with the maps defined by the maps 
$$[p_0(j), p_1(j)] : [\lim D_0, \lim D_1, \pi_{\lim}] \to D(j)$$
 satisfies the universal property of the limit for the diagram $D$.
\end{proof}

\subsubsection{Accessibility of comma categories}
\begin{note}
We list below some technical results on locally presentable categories. Good references on the subject include \cite{Adamek-Rosicky-loc-pres}, \cite{Chorny_Rosi}, \cite{Low_heart}. 

\end{note}
\begin{prop}
Let  $\Ub: \ag \to  \M$ be an accessible functor between accessible categories. 
Then the category $\mua= (\M \downarrow \Ub)$ is also accessible.
\end{prop}

\begin{proof}
This is a special case of Theorem 2.43 in the book of Ad\'amek and Rosick\'y \cite{Adamek-Rosicky-loc-pres}  for the comma category  $\Id_{\M} \downarrow \Ub \cong \mdu$.
\end{proof}

\begin{cor}\label{cor-presentable}
Let $\Ub : \ag \rightleftarrows \M : \Fb$ be an adjunction between locally presentable categories. Then the category $\mdu$ is also locally presentable. 
\end{cor}

\begin{proof}
By the previous result we know that $\mdu$ is an accessible category. So it remains to show that $\mdu$ is also cocomplete. But this is given by Proposition \ref{prop-limit-colimit}. 
\end{proof}

\section{Injective and projective model structures}\label{inj-proj-comma}
In this section we consider a cofibrantly generated model category $\M$ without requiring the axioms of a monoidal model category yet.  
\subsection{Detecting the Segal conditions for comma categories}
Consider the adjunction $\Piar: \mua \rightleftarrows \Ar(\M): \Gamma$ given in Proposition \ref{prop-gamma}. 
\begin{df}\label{localization-set}
Define the \textbf{injective localizing set}   $\kb_{\Iam}= \{\Gamma(\alpha_i) \}_{i \in \Iam} = \Gamma (\alpha_{\Iam})$ as the image of $\alpha_{\Iam}$ under the left adjoint $\Gamma$. 
\end{df}
 An immediate consequence of Lemma \ref{lifting-lem} is:
\begin{prop}\label{detect-segal}
In the category $\mua$ the following hold. 

\begin{enumerate}
\item An object $[\F]=[\F_0,\F_1,\pif]$ is $\kbi$-injective, that is $[\F] \to \ast$ has the RLP with respect to $\kbi$, if and only if, the map $\pif : \F_0 \to \Ub(\F_1)$ is a trivial fibration in $\M$. In particular $[\F]$ satisfies the Segal condition.
\item For any $\Pa \in \ag$, $\iota(\Pa)=[\Ub(\Pa), \Pa, \Id_{\Ub(\Pa)}]$ is $\kbi$-injective. 
\end{enumerate}
\end{prop}

\begin{proof}
We proceed by adjunction. $[\F] \to \ast$ has the  RLP with respect to $\kbi$, if and only if, the map $\pif : \F_0 \to \Ub(\F_1)$ has the RLP with respect to $\alpha_{\I}$, if and only if $\pif$ is a trivial fibration, thanks to Lemma \ref{lifting-lem}. This gives the first assertion. The second assertion is clear since every identity morphism is a trivial fibration. 
\end{proof}

\begin{df}\label{localization-set-proj}
Let $\M$ be a tractable model category. Define the \textbf{projective localizing set}   $\kbi_{proj}= \{\Gamma(\zeta_i) \}_{i \in \Iam} = \Gamma (\zeta_{\Iam})$ as the image of $\zeta_{\Iam}$ under the left adjoint $\Gamma$. 
\end{df}
The following proposition will be used for the projective model structure.
\begin{prop}\label{detect-segal-pj}
Let $\M$ be a tractable model category. Let $[\F]=[\F_0,\F_1,\pif]$ be an object in the category $\mua$ and assume that:
\begin{enumerate}
\item $[\F_0,\F_1,\pif]$ is $\kbi_{proj}$-injective, that is $[\F] \to \ast$ has the RLP with respect to $\kbi_{proj}$;
\item $\pif$ is a map between fibrant object i.e, the unique map $\Fc \to \ast$ is a level fibration.
\end{enumerate}
Then $\Fc$ satisfies the Segal conditions i.e, $\pif$ is a weak equivalence. 
\end{prop}

\begin{proof}
Just proceed by adjunction with Lemma \ref{lifting-lem-proj}.
\end{proof}

\subsection{Injective and projective model structures on $\mua$}
We are going to mimic Hovey's result and have an \textbf{injective model structure} on $\mua$ such that when $\Ub=\Id_{\M}$ we recover Hovey's theorem. Stanculescu \cite{Stanculescu_multi} and Toën \cite{To_hall} considered model structures on the (equivalent) category $(\Fb \downarrow \ag)$. \ \\
 
Recall that from the adjunction $\Fb \dashv \Ub$ there is a functorial isomorphism of hom-sets:
$$\varrho : \Hom_{\M}(\F_0,  \Ub(\F_1))  \xrightarrow{\cong} \Hom_{\ag}(\Fb (\F_0), \F_1).$$
Then if $\sigma=[\sigma_0,\sigma_1]: \Fc \to \Gc$ is a map in $\mua$ we have two commutative squares that are mutually adjoint:
\[
\xy
(0,18)*+{\F_0}="W";
(0,0)*+{\Ub(\F_1)}="X";
(30,0)*+{\Ub(\G_1)}="Y";
(30,18)*+{\G_0}="E";
{\ar@{->}^-{\Ub(\sigma_1)}"X";"Y"};
{\ar@{->}^-{\pi_{\F}}"W";"X"};
{\ar@{->}^-{\sigma_0}"W";"E"};
{\ar@{->}^-{\pi_{\G}}"E";"Y"};
(45,9)*+{\Longleftrightarrow}="Z";
\endxy
\xy
(0,18)*+{\Fb(\F_0)}="W";
(0,0)*+{\F_1}="X";
(30,0)*+{\G_1}="Y";
(30,18)*+{\Fb(\G_0)}="E";
{\ar@{->}^-{\sigma_1}"X";"Y"};
{\ar@{->}^-{\varrho(\pi_{\F})}"W";"X"};
{\ar@{->}^-{\Fb(\sigma_0)}"W";"E"};
{\ar@{->}^-{\varrho(\pi_{\G})}"E";"Y"};
\endxy
\]
\begin{df}\label{injective-data}
Let $\sigma: \Fc \to \Gc$ be a map in $\mua$. With the previous notation we will say that:
\begin{enumerate}
\item $\sigma$ is a \textbf{injective (trivial) cofibration} if $\sigma_0$ and $\sigma_1$ are (trivial) cofibrations in $\M$. 
\item $\sigma$ is a \textbf{level-wise weak equivalence} (resp. \textbf{level-wise fibration}) if:
\begin{itemize}
\item  $\sigma_0$ is a weak equivalence (resp. fibration)  in $\M$ and
\item  $\sigma_1$ is a weak equivalence (resp. fibration) in $\ag$.
\end{itemize}
\item $\sigma$ is an \textbf{injective (trivial) fibration}  if:
\begin{itemize}
\item $\sigma_1: \F_1 \to \G_1$ is a  (trivial) fibration in $\ag$ and 
\item the induced map $ \F_0 \to \Ub(\F_1) \times_{\Ub(\G_1)} \G_0 $
is a (trivial) fibration in $\M$. 
\end{itemize}

\item $\sigma$ is an \textbf{projective (trivial) cofibration}  if:
\begin{itemize}
\item $\sigma_0: \F_0 \to \G_0$ is a  (trivial) cofibration in $\M$ and 
\item the induced map $ \F_1 \cup^{\Fb(\F_0)} \Fb(\G_0) \to  \G_1 $ is a (trivial) cofibration in $\ag$.
\end{itemize}
\end{enumerate}
\end{df}

We need the following result to establish the Reedy model structure. The proof is straightforward but we include it for completeness. 
\begin{lem}\label{lem-fib-inj}
Let $\Ub: \ag \to \M$ be a right Quillen functor. Then a map $\sigma: \Fc \to \Gc$ is an injective trivial fibration if and only if, it's an injective fibration and a level-wise weak equivalence. 
\end{lem}
\begin{proof}
To prove the only if part, it suffices to show that $\sigma$ is a level-wise weak equivalence.  If $\sigma=[\sigma_0,\sigma_1]$ is an injective trivial fibration, then $\sigma_1$ is a trivial fibration in $\ag$ by definition, and therefore $\Ub(\sigma_1)$ is also trivial fibration in $\M$, since $\Ub$ preserves trivial fibrations. Now in any model category, trivial fibrations are closed under pullback, and we see that the canonical map $p: \Ub(\F_1) \times_{\Ub(\G_1)} \G_0 \to \G_0 $ is also a trivial fibration.  The other part of being an injective trivial fibration means that the map $ \F_0 \to \Ub(\F_1) \times_{\Ub(\G_1)} \G_0 $ is a trivial fibration. The map $\sigma_0$ is the composite the previous map followed by the map $p$, and both are trivial fibrations, thus $\sigma_0$ is a trivial fibrations as well.  In the end $\sigma_1$ and $\sigma_0$ are trivial fibrations in the respective model category, in particular each of them is a  weak equivalence.  \ \\

For the if part, we simply need to show that the map $ \F_0 \to \Ub(\F_1) \times_{\Ub(\G_1)} \G_0 $ is a weak equivalence since it's already a fibration.  The argument is based on the $3$-for-$2$ property of weak equivalences in $\M$. Indeed, assume $\sigma=[\sigma_0,\sigma_1]$ is an injective fibration and a level-wise weak equivalence. Then $\sigma_1$ is a trivial fibration in $\ag$ by definition, and therefore $\Ub(\sigma_1)$ and then  its pullback $p$ are also a trivial fibrations.  As mentioned above, we have a factorization $\sigma_0$ as the map $\F_0 \to \Ub(\F_1) \times_{\Ub(\G_1)} \G_0 $ followed by the map $p: \Ub(\F_1) \times_{\Ub(\G_1)} \G_0 \to \G_0 $. Since $\sigma_0$ and $p$ are weak equivalences, then by $3$-for-$2$, the map 
$ \F_0 \to \Ub(\F_1) \times_{\Ub(\G_1)} \G_0 $ is also a weak equivalence as desired. 
\end{proof}

The dual statement is:
\begin{lem}\label{lem-cof-proj}
 A map $\sigma: \Fc \to \Gc$ is a projective trivial cofibration if and only if, it's a level-wise weak equivalence and a projective cofibration. 
\end{lem}
The proof is dual to the previous one so we just need to adapt it. 
\begin{proof}
To prove the only if part, it suffices to show that $\sigma$ is a level-wise weak equivalence.  If $\sigma=[\sigma_0,\sigma_1]$ is an projective trivial cofibration, then $\sigma_0$ is a trivial cofibration in $\M$ by definition, and therefore $\Fb(\sigma_0)$ is also trivial cofibration in $\ag$, since $\Fb$ preserves trivial cofibrations. Now in any model category, trivial cofibrations are closed under cobase change. It follows that the canonical map $q: \F_1 \to   \F_1 \cup^{\Fb(\F_0)} \Fb(\G_0) $ is a trivial cofibration in $\ag$. The  other part of being an projective trivial cofibration means that the map $ \F_1 \cup^{\Fb(\F_0)} \Fb(\G_0) \to  \G_1 $ is a trivial cofibration. The map $\sigma_1$ is the composite of the map $q$ followed  by the previous map, and both are trivial cofibrations. Then $\sigma_0$ is a trivial cofibrations as well.  In the end $\sigma_0$ and $\sigma_1$ are trivial cofibrations in the respective model category, in particular each of them is a  weak equivalence.  \ \\

For the if part, we simply need to show that the map  $ \F_1 \cup^{\Fb(\F_0)} \Fb(\G_0) \to  \G_1 $ is a weak equivalence since it's already a cofibration.  The argument is also based on the $3$-for-$2$ property of weak equivalences in $\ag$. Indeed, assume $\sigma=[\sigma_0,\sigma_1]$ is a projective cofibration and a level-wise weak equivalence. Then $\sigma_0$ is a trivial cofibration in $\M$ by definition,  therefore $\Fb(\sigma_0)$ and  its cobase change $q$ are also a trivial cofibrations in $\ag$.  As mentioned above, we have a factorization $\sigma_1$ as the map $q$ followed by the map $ \F_1 \cup^{\Fb(\F_0)} \Fb(\G_0) \to  \G_1 $. Since $\sigma_1$ and $q$ are weak equivalences, then by $3$-for-$2$, the map $$\F_1 \cup^{\Fb(\F_0)} \Fb(\G_0) \to  \G_1 $$ is also a weak equivalence as desired.
\end{proof}

\subsection{Factorizations}
We generalize here to comma categories the required factorizations that are necessary to get the Reedy model structure. We simply follow the classical method.

\subsubsection{Injective factorizations}
\begin{prop}\label{inj-fact-reedy}
Let $\Ub: \ag \to \M$ be a right Quillen functor. Then with the previous definition the following hold. 
\begin{enumerate}
\item Any map $\sigma: \Fc \to \Gc$ can be factored as an injective cofibration followed by a an injective trivial fibration.
\item Any map $\sigma: \Fc \to \Gc$ can be factored as an injective trivial cofibration followed by an injective fibration. 
\end{enumerate}
\end{prop}
If we decide to prove each assertion separately, this will be very long. So we will use a generic language of \emph{(weak) factorization systems} (see for example \cite{Riehl_algcat}). In any model category we have two weak factorizations systems $(\cof, \fib \cap \W)$ and $(\cof \cap \W, \fib)$, where $\cof$, $\fib$ and $\W$ are respectively the classes of cofibrations, fibrations and weak equivalences.  Then we have two different factorization  systems $\faca$ on $\ag$ and $\facm$ on $\M$ that determine the model structure on each model category. The functor $\Ub$ is such that $\Ub(\ra ) \subseteq \mr$ for each of the  corresponding factorization systems. 

\begin{proof}[Proof of Proposition \ref{inj-fact-reedy}]
Let  $\sigma=[\sg_0,\sg_1]: \Fc \to \Gc$ be a map in $\mua$.\ \\
Use the axiom of the model category $\ag$ with respect to the factorization system $\faca$ to write $\sigma_1$ as $\sg_1= r(\sg_1) \circ l(\sg_1)$:
$$\F_1 \xrightarrow{\sg_1} \G_1 = \F_1\xhookrightarrow{l(\sg_1)}  E_1 \xtwoheadrightarrow{r(\sg_1)} \G_1,$$ 
with $ r(\sg_1) \in \ra $ and $ l(\sg_1) \in \la$. The image under $\Ub$ of this factorization, gives a factorization  $\Ub(\sg_1)= \Ub(r(\sg_1)) \circ \Ub(l(\sg_1))$, with $\Ub(r(\sg_1)) \in \mr$ since  $\Ub(\ra ) \subseteq \mr$.

Form the pullback square in $\M$ defined by the pullback data: $$\Ub(E_1) \xrightarrow{\Ub(r(\sg_1))} \Ub(\G_1) \xleftarrow{\pig} \G_0,$$
and let $p_1 : \Ub(E_1) \times_{\Ub(\G_1)} \G_0 \to \G_0$ and $p_2:   \Ub(E_1) \times_{\Ub(\G_1)} \G_0  \to \Ub(\Ea_1)$ be the canonical maps. Then $p_1 \in \mr$ because $\mr$ is closed under pullbacks.  
The universal property of the pullback square gives a unique map $\delta : \F_0 \to \Ub(E_1) \times_{\Ub(\G_1)} \G_0,$
such that everything below commutes. 
\[
\xy
(0,30)*+{\F_0}="W";
(0,0)*+{\Ub(\F_1)}="X";
(60,0)*+{\Ub(\G_1)}="Y";
(60,30)*+{\G_0}="E";
{\ar@{->}^-{\pi_{\F}}"W";"X"};
{\ar@{->}^-{\sigma_0}"W";"E"};
{\ar@{->}^-{\pi_{\G}}"E";"Y"};
(30,20)*+{\Ub(E_1) \times_{\Ub(\G_1)} \G_0}="U";
(30,0)*+{\Ub(E_1)}="V";
{\ar@{->}^-{\Ub(l(\sg_1))}"X";"V"};
{\ar@{->>}^-{\Ub(r(\sg_1))}"V";"Y"};
{\ar@{->}^-{p_2}"U";"V"};
{\ar@{->>}^-{p_1}"U";"E"};
{\ar@{->}^-{\delta}"W";"U"};
\endxy
\]

Now we use the factorization system $\facm$ to factor the map $\delta$:
$$ \delta : \F_0 \to \Ub(E_1) \times_{\Ub(\G_1)} \G_0 = \F_0  \xhookrightarrow{l(\delta)}  m_0 \xtwoheadrightarrow{r(\delta)}\Ub(E_1) \times_{\Ub(\G_1)} \G_0,$$ 
with $ r(\delta) \in \mr $ and $ l(\delta) \in \lm$.
Let $\Ec=[\Ea_0,\Ea_1,\pi_{\Ea}]$ be the object of $\mua$ defined by
$$ \Ea_0= m_0, \quad \Ea_1=E_1, \quad \pi_{\Ea}=p_2 \circ r(\delta).$$

We  have a map $l(\sigma) : \Fc \to \Ec$ given by the couple $[l(\delta),  l(\sigma_1)] \in \lm \times \la$ and a map $r(\sigma) : \Ec \to \Gc$ given by the couple $[p_1 \circ r(\delta),  r(\sigma_1)] $  with $[r(\delta),  r(\sigma_1)] \in \mr \times \ra$,  such that $\sg= r(\sg) \circ l(\sg)$.  This gives the assertions. 
\end{proof}

\subsubsection{Projective factorizations}
\begin{prop}\label{proj-fact-reedy}
Let $\Ub: \ag \to \M$ be a right Quillen functor. Then with the previous definition the following hold. 
\item Any map $\sigma: \Fc \to \Gc$ can be factored as a projective cofibration followed by a  projective trivial fibration.
\item Any map $\sigma: \Fc \to \Gc$ can be factored as a projective trivial cofibration followed by a projective fibration. 
\end{prop}

\begin{proof}
Consider a factorization system $\faca$ on $\ag$ and a factorization system $\facm$ on $\M$ such that $\Ub(\ra) \subseteq \mr$ and $\Fb(\lm) \subseteq \la$. \ \\

Let  $\sigma=[\sg_0,\sg_1]: \Fc \to \Gc$ be a map in $\mua$.\ \\
Use the axiom of the model category $\ag$ with respect to the factorization system $\faca$ to write $\sigma_0$ as $\sg_0= r(\sg_0) \circ l(\sg_0)$:
$$\F_0 \xrightarrow{\sg_0} \G_0 = \F_1\xhookrightarrow{l(\sg_0)}  m_0 \xtwoheadrightarrow{r(\sg_0)} \G_0,$$ 
with $ r(\sg_0) \in \mr $ and $ l(\sg_0) \in \lm$.\ \\
The image under $\Fb$ of this factorization, gives a factorization  $\Fb(\sg_0)= \Fb(r(\sg_0)) \circ \Fb(l(\sg_0))$, with $\Fb(l(\sg_0)) \in \la$ since  $\Fb(\lm ) \subseteq \la$. 

Form the pushout square in $\ag$ defined by the pushout data: $$\F_1 \xleftarrow{\varrho(\pif)} \Fb(\F_0) \xrightarrow{\Fb(l(\sg_0)) }\Fb (m_0), $$
and let $i_2 : \F_1 \to  \F_1 \cup^{\Fb(\F_0)} \Fb(\G_0)$ and $i_1: \Fb (m_0)   \to  \F_1 \cup^{\Fb(\F_0)} \Fb(\G_0)$ be the canonical maps. Then $i_2 \in \la$ because $\la$ is closed under pushouts.
The universal property of the pushout square gives a unique map 
$$ \zeta : \F_1 \cup^{\Fb(\F_0)} \Fb(m_0) \to  \G_1,$$
such that everything below commutes. 
\[
\xy
(0,30)*+{\Fb(\F_0)}="W";
(0,0)*+{\F_1}="X";
(60,0)*+{\G_1}="Y";
(60,30)*+{\Fb(\G_0)}="E";
{\ar@{->}^-{\varrho(\pif)}"W";"X"};
{\ar@{->}^-{\sg_1}"X";"Y"};
{\ar@{->}^-{\varrho(\pig)}"E";"Y"};
(30,10)*+{\F_1 \cup^{\Fb(\F_0)} \Fb(m_0)}="U";
(30,30)*+{\Fb(m_0)}="V";
{\ar@{->}^-{\Fb(l(\sg_0))}"W";"V"};
{\ar@{->}^-{\Fb(r(\sg_0))}"V";"E"};
{\ar@{->}^-{i_2}"X";"U"};
{\ar@{->}^-{i_1}"V";"U"};
{\ar@{->}^-{\zeta}"U";"Y"};
\endxy
\]
Now we use the factorization system $\faca$ to factor the map $\zeta$:
$$ \delta : \F_1 \cup^{\Fb(\F_0)} \Fb(m_0) \to \G_1 = \F_1 \cup^{\Fb(\F_0)} \Fb(m_0)  \xhookrightarrow{l(\zeta)}  E_1 \xtwoheadrightarrow{r(\zeta)} \G_1,$$ 
with $ l(\zeta) \in \la $ and $ r(\zeta) \in \ra$.
Let $\Ec=[\Ea_0,\Ea_1,\pi_{\Ea}]$ be the object of $\mua$ defined by
$$ \Ea_0= m_0, \quad \Ea_1=E_1, \quad \pi_{\Ea}= \varrho^{-1} (l(\zeta) \circ i_1) \in  \Hom_{\M}(m_0,\Ub(E_1)).$$

We  have a map $l(\sigma) : \Fc \to \Ec$ given by the couple $[l(\sg_0),  l(\zeta) \circ i_2]$, with  $[l(\sg_0),  l(\zeta)] \in \lm \times \la$, and a map $r(\sg) : \Ec \to \Gc$ given by the couple $[ r(\sg_0),  r(\zeta)] \in \mr \times \ra $ ;  such that $\sg= r(\sg) \circ l(\sg)$.  These gives the assertions. 

\end{proof}

\subsection{Lifting properties}
On the model category $\ag$ we will use here again the generic notation $\faca$ for both factorization systems. And similarly we will denote by $\facm$ both factorization systems on $\M$. In both cases we have $\Ub(\ra) \subseteq \mr$ and $\Fb(\lm) \subseteq \la$. Recall that $\sg=[\sg_0,\sg_1]: \Fc \to \Gc$ is an injective (trivial) cofibration in $\mua$ if $[\sg_0,\sg_1] \in \lm \times \la$. And $\beta=[\beta_0,\beta_1]: [\Pa] \to [\Qa]$ is an injective (trivial) fibration if $[ \delta , \beta_1] \in \lm \times \la$, where $\delta:\Pa_0 \to \Ub(\Pa_1) \times_{\Ub(\Qa_1)} \Qa_0$ is the induced map.
\subsubsection{Injective lifting properties}
\begin{prop}\label{inj-lifting}
Let $\Ub: \ag \to \M$ be a right Quillen functor. Then with the previous definition the following hold. 
\begin{enumerate}
\item Any lifting problem defined by an injective cofibration and  an injective trivial fibration has a solution.
\item Any lifting problem defined by an injective trivial cofibration and  an injective fibration has a solution.
\end{enumerate}
\end{prop}

\begin{proof}
The proof is the same for both assertions by considering the appropriate factorization system on $\ag$ and $\M$.\ \\
Consider a lifting problem in $\mua$ defined by $\sigma : \Fc \xrw{[\sg_0,\sg_1]} \Gc$ and $\beta: \Pc \xrw{[\beta_0,\beta_1]} \Qc $ as follows.
 \[
 \xy
(0,18)*+{\Fc}="W";
(0,0)*+{\Gc}="X";
(30,0)*+{\Qc}="Y";
(30,18)*+{\Pc}="E";
{\ar@{->}^-{(\gamma_0, \gamma_1)}"X";"Y"};
{\ar@{->}_-{(\sg_0,\sg_1)}"W";"X"};
{\ar@{->}^-{(\theta_0,\theta_1)}"W";"E"};
{\ar@{->}^-{(\beta_0,\beta_1)}"E";"Y"};
\endxy
\]
The image of this lifting problem under $\Piun: \mua \to \ag$ is a lifting problem defined by $\sg_1$ and $\beta_1$. Therefore if $[\sg_1, \beta_1]\in \la \times \ra$, then there is a solution $s_1: \G_1 \to \Pa_1$ of this lifting problem in $\ag$. Then by functoriality of $\Ub$, the map $\Ub(s_1)$ is a solution to the induced lifting problem defined by $\Ub(\sg_1)$ and $\Ub(\beta_1)$ in $\M$. Part of $\Ub(s_1)$ being a solution gives an equality $\Ub(\gamma_1)= \Ub(\beta_1) \circ \Ub(s_1)$. And  $[\gamma]=[\gamma_0,\gamma_1]$ being a morphism in $\mua$ gives the equality $\pi_{\Qa} \circ \gamma_0 = \Ub(\gamma_1) \circ \pig$.\ \\

Now consider the map $\Ub(s_1) \circ \pig \in \Hom_{\M}(\G_0,\Ub(\Pa_1))$ and the map $\gamma_0 \in \Hom_{\M}(\G_0,\Qa_0)$. Then by the above, it's not hard to see that these maps complete the pullback data $$\Ub(\Pa_1) \xrightarrow{\Ub(\beta_1)} \Ub(\Qa_1) \xleftarrow{\pi_{\Qa}} \Qa_0$$ into a commutative square ($\pi_{\Qa} \circ \gamma_0= \Ub(\beta_1) \circ \Ub(s_1) \circ \pig$). Therefore, by the universal property of the pullback square there is a unique map: $\zeta:\G_0 \to \Ub(\Pa_1) \times_{\Ub(\Qa_1)} \Qa_0$,
making everything compatible. In particular $\gamma_0$ and $\Ub(s_1) \circ \pig$  factor through $\zeta$.\ \\

Our original lifting problem in $\mua$ defined by $[\sigma]$ and $[\beta]$ is represented by a commutative cube in $\M$. If we unfold it, we find that everything commutes in the diagram hereafter:
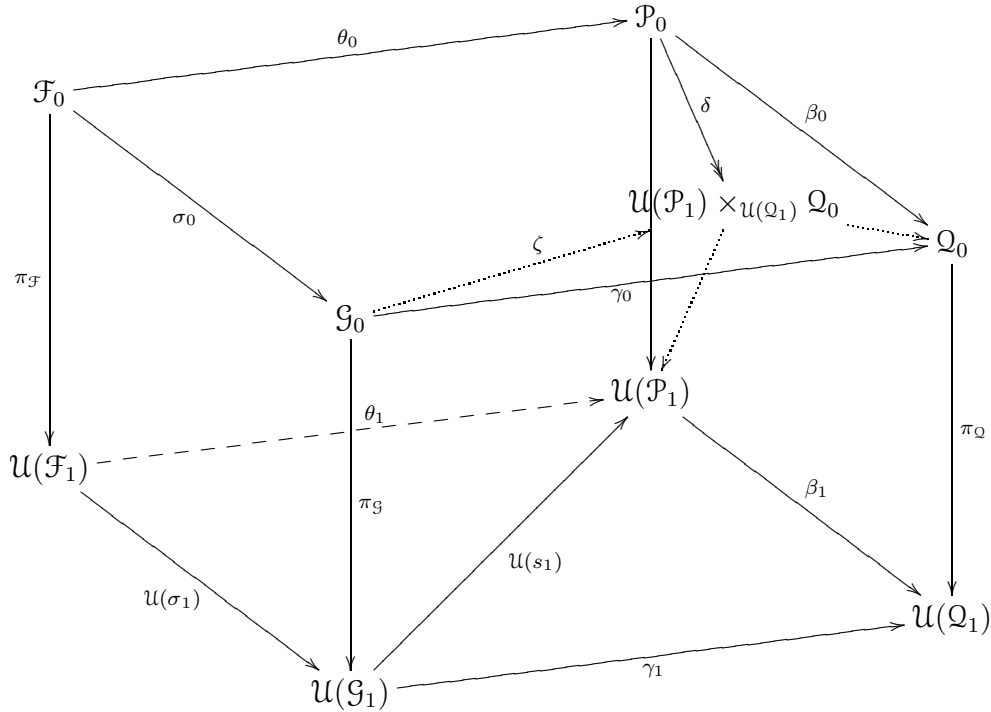
\[
\xy
(-60,20)*+{\F_0}="A";
(20,30)*+{\Pa_0}="B";
(-20,-10)*+{\G_0}="C";
(60,0)*+{\Qa_0}="D";
{\ar@{->}^{\theta_0}"A";"B"};
{\ar@{->}_{\sg_0}"A";"C"};
{\ar@{->}^{\beta_0}"B";"D"};
{\ar@{->}_{\gamma_0~~~~~~~}"C";"D"};
(-60,-30)*+{\Ub(\F_1)}="X";
(20,-20)*+{\Ub(\Pa_1)}="Y";
(-20,-60)*+{\Ub(\G_1)}="Z";
(60,-50)*+{\Ub(\Qa_1)}="W";
{\ar@{-->}^{\quad \quad \theta_1}"X";"Y"};
{\ar@{->}_{\pif}"A";"X"};
{\ar@{->}^{\beta_1}"Y";"W"};
{\ar@{->}^{\pig}"C";"Z"};
{\ar@{->}_-{}"B";"Y"};
{\ar@{->}^{\pi_{\Qa}}"D";"W"};
{\ar@{->}_{\Ub(\sg_1)}"X";"Z"};
{\ar@{->}_{\gamma_1}"Z";"W"};
(20,30)+(11,-25)*+{\Ub(\Pa_1) \times_{\Ub(\Qa_1)} \Qa_0 }="E";
{\ar@{->>}^{\delta}"B";"E"};
{\ar@{.>}^{}"E";"D"};
{\ar@{.>}^{}"E";"Y"};
{\ar@{.>}^{\zeta}"C";"E"};
{\ar@{->}_-{\Ub(s_1)}"Z";"Y"};
\endxy
\]

As we see, we get a commutative square that corresponds to a lifting problem defined by the map $\sg_0:\F_0 \to \G_0$ and the map $\delta: \Pa_0 \to  \Ub(\Pa_1) \times_{\Ub(\Qa_1)} \Qa_0 $:
 \[
 \xy
(0,18)*+{\F_0}="W";
(0,0)*+{\G_0}="X";
(30,0)*+{\Ub(\Pa_1) \times_{\Ub(\Qa_1)} \Qa_0}="Y";
(30,18)*+{\Pa_0}="E";
{\ar@{->}^-{\zeta}"X";"Y"};
{\ar@{->}_-{\sg_0}"W";"X"};
{\ar@{->}^-{\theta_0}"W";"E"};
{\ar@{->}^-{\delta}"E";"Y"};
\endxy
\]

Now it suffices to observe that this lifting problem has a solution if and only if our original lifting problem has a solution. Indeed, if $s_0 : \G_0 \to \Pa_0$ is a solution to the previous lifting problem, then we have map $[s]=[s_0,s_1]: \Gc \to \Pc$ that is a solution to the original problem. Conversely given a solution $[s]=[s_0,s_1]: \Gc \to \Pc$ to the original lifting problem , then the component $s_0: \G_0 \to \Pa_0$ is a solution to the lifting problem defined by $\sg_0$ and $\delta: \Pa_0 \to  \Ub(\Pa_1) \times_{\Ub(\Qa_1)} \Qa_0 $.\ \\

Finally one clear sees that the lifting problem defined by $\sg_0$ and $\delta$ has a  solution $s_0 \in \Hom_{\M}(\G_0,\Pa_0)$ since $[\sg_0,\delta] \in \lm \times \mr$.

\end{proof}

\subsubsection{Projective lifting properties}
\begin{prop}\label{proj-lifting}
Let $\Ub: \ag \to \M$ be a right Quillen functor. Then with the previous definition the following hold. 
\begin{enumerate}
\item Any lifting problem defined by a projective cofibration and a projective trivial fibration has a solution.
\item Any lifting problem defined by a projective trivial cofibration and  projective fibration has a solution.
\end{enumerate}
\end{prop}

\begin{proof}
We proceed in a dual manner as in the proof of Proposition \ref{inj-lifting} with the same notation for both factorizations systems on $\ag$ and $\M$. \ \\
Consider a lifting problem in $\mua$ defined by $\sigma : \Fc \xrw{[\sg_0,\sg_1]} \Gc$ and $\beta: \Pc \xrw{[\beta_0,\beta_1]} \Qc $ as follows.
 \[
 \xy
(0,18)*+{\Fc}="W";
(0,0)*+{\Gc}="X";
(30,0)*+{\Qc}="Y";
(30,18)*+{\Pc}="E";
{\ar@{->}^-{(\gamma_0, \gamma_1)}"X";"Y"};
{\ar@{->}_-{(\sg_0,\sg_1)}"W";"X"};
{\ar@{->}^-{(\theta_0,\theta_1)}"W";"E"};
{\ar@{->}^-{(\beta_0,\beta_1)}"E";"Y"};
\endxy
\]
The image of this lifting problem under $\Pio: \mua \to \M$ is a lifting problem defined by $\sg_0$ and $\beta_0$. Therefore if $[\sg_0, \beta_0]\in \lm \times \mr$, then there is a solution $s_0: \G_0 \to \Pa_0$ of this lifting problem in $\M$. Then by functoriality of $\Fb$, the map $\Fb(s_0)$ is a solution to the induced lifting problem defined by $\Fb(\sg_0)$ and $\Fb(\beta_0)$ in $\ag$. Part of $\Fb(s_0)$ being a solution gives an equality $\Fb(\theta_0)= \Fb(s_0) \circ \Fb(\sg_0)$. And  $[\theta]=[\theta_0,\theta_1]$ being a morphism in $\mua$ gives by adjunction a morphism $\varrho(\pif) \xrw{[\Fb(\sg_0), \theta_1]} \varrho(\piq)$ in $\Arr(\ag)$. In particular we have an the equality $\varrho(\pip) \circ \Fb(\theta_0) = \theta_1 \circ \varrho(\pif)$.\ \\

Now consider the map $\Fb(s_0) \circ \pip \in \Hom_{\ag}(\Fb(\G_0),\Pa_1)$ and the map $\theta_1 \in \Hom_{\ag}(\F_1,\Pa_1)$. Then using the previous equalities, it's not hard to see that these maps complete the pushout data $$\Fb(\G_0) \xlw{\Fb(\sg_0)} \Fb(\F_0) \xrw{\varrho(\pif)} \F_1$$ into a commutative square ($\theta_1 \circ \varrho(\pif)= \varrho(\pip) \circ \Fb(s_0) \circ  \Fb(\sg_0)$). Therefore, by the universal property of the pushout square, there is a unique map: $\xi:\F_1 \cup^{\Fb(\F_0)} \Fb(\G_0) \to  \Pa_1,$
making everything compatible. In particular  $\varrho(\piq) \circ \Fb(s_0)$ and $\theta_1$ factor through $\xi$.\ \\

The original lifting problem in $\mua$ defined by $[\sigma]$ and $[\beta]$ is represented by adjunction by a commutative cube in $\ag$. If we unfold it, we find that everything commutes in the diagram hereafter:
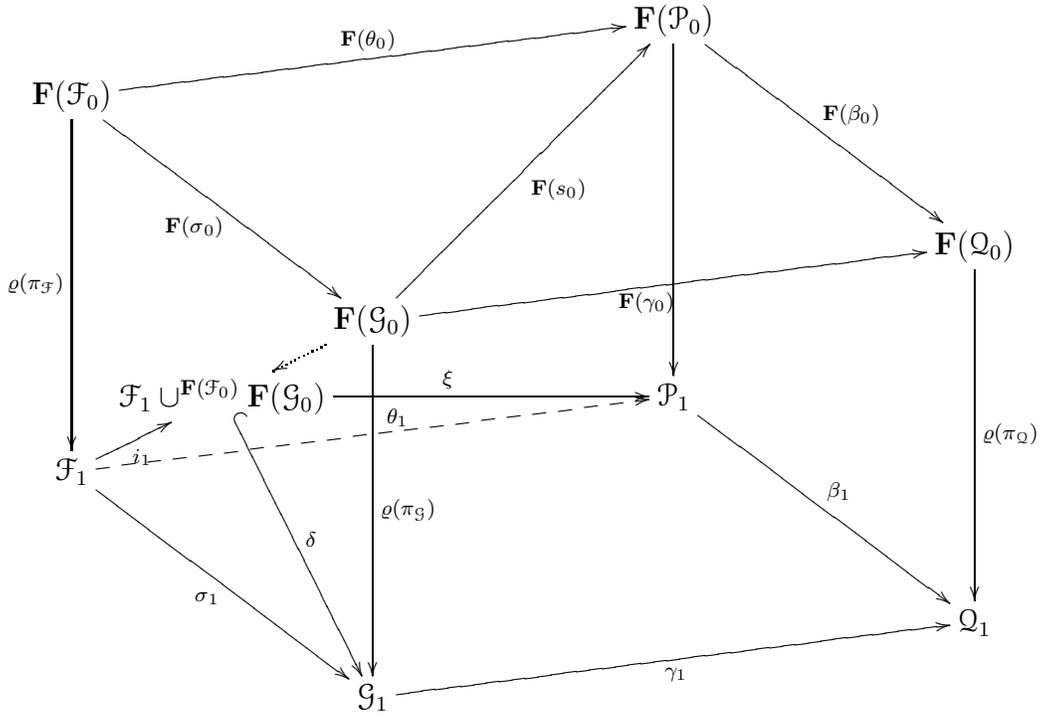
\[
\xy
(-60,20)*+{\Fb(\F_0)}="A";
(20,30)*+{\Fb(\Pa_0)}="B";
(-20,-10)*+{\Fb(\G_0)}="C";
(60,0)*+{\Fb(\Qa_0)}="D";
{\ar@{->}^{\Fb(\theta_0)}"A";"B"};
{\ar@{->}_{\Fb(\sg_0)}"A";"C"};
{\ar@{->}_{\Fb(s_0)}"C";"B"};
{\ar@{->}^{\Fb(\beta_0)}"B";"D"};
{\ar@{->}_{\Fb(\gamma_0)~~~~~~~}"C";"D"};
(-60,-30)*+{\F_1}="X";
(20,-20)*+{\Pa_1}="Y";
(-20,-60)*+{\G_1}="Z";
(60,-50)*+{\Qa_1}="W";
{\ar@{-->}^{\quad \quad \theta_1}"X";"Y"};
{\ar@{->}_{\varrho(\pif)}"A";"X"};
{\ar@{->}^{\beta_1}"Y";"W"};
{\ar@{->}^{\varrho(\pig)}"C";"Z"};
{\ar@{->}_-{}"B";"Y"};
{\ar@{->}^{\varrho(\piq)}"D";"W"};
{\ar@{->}_{\sg_1}"X";"Z"};
{\ar@{->}_{\gamma_1}"Z";"W"};
(-40,-20)*+{\F_1 \cup^{\Fb(\F_0)} \Fb(\G_0) }="E";
{\ar@{->}_-{i_1}"X";"E"};
{\ar@{^(->}^{\delta}"E";"Z"};
{\ar@{->}^{\xi}"E";"Y"};
{\ar@{.>}^{}"C";"E"};
\endxy
\]

As we see, we get a commutative square that corresponds to a lifting problem defined by the maps $\delta: \F_1 \cup^{\Fb(\F_0)} \Fb(\G_0) \to \G_1$ and $\beta_1: \Pa_1 \to \Qa_1$:
 \[
 \xy
(0,18)*+{ \F_1 \cup^{\Fb(\F_0)} \Fb(\G_0)}="W";
(0,0)*+{\G_1}="X";
(30,0)*+{\Qa_1}="Y";
(30,18)*+{\Pa_1}="E";
{\ar@{->}^-{\gamma_1}"X";"Y"};
{\ar@{->}_-{\delta}"W";"X"};
{\ar@{->}^-{\xi}"W";"E"};
{\ar@{->}^-{\beta_1}"E";"Y"};
\endxy
\]

As in the injective case, it suffices to observe that this lifting problem has a solution if and only if our original lifting problem has a solution. Indeed, if $s_1 : \G_1 \to \Pa_1$ is a solution to the previous lifting problem, then we have map $[s]=[s_0,s_1]: \Gc \to \Pc$ that is a solution to the original problem. Conversely given a solution $[s]=[s_0,s_1]: \Gc \to \Pc$ to the original lifting problem , then the component $s_1: \G_1 \to \Pa_1$ is a solution to the lifting problem defined by $\delta: \F_1 \cup^{\Fb(\F_0)} \Fb(\G_0) \to \G_1$ and $\beta_1$.\ \\

Finally one clear sees that the lifting problem defined by $\delta$ and $\beta_1$ has a  solution $s_1 \in \Hom_{\ag}(\G_1,\Pa_1)$ since $[\sg_0,\delta] \in \la \times \ra$.
\end{proof}

\subsection{The model structures}
We are able to state our theorems. 
\subsubsection{Injective model structure}

\begin{thm}\label{inj-thm}
Let $\Ub : \ag \to \M$ be a right Quillen functor. Then with the previous definitions, the following hold. 
There is a model structure on the category $\mdua= \mua$ which may be described as follows. 
\begin{itemize}
\item A map $\sigma: \Fc \to \Gc$ is a weak equivalence if and only if it's a level-wise weak equivalence.
\item A map $\sigma: \Fc \to \Gc$ is cofibration if it's an injective cofibration. 
\item A map $\sigma: \Fc \to \Gc$ is fibration if it's an injective fibration. 
\item We will denote this model category  by $\muaij$. 
\end{itemize}
\end{thm}

\begin{proof}
The class of level-wise weak equivalences clearly satisfies the $2$-out-of-$3$ property. The three classes of cofibrations, fibrations and weak equivalences are closed under composition and retracts. With Proposition \ref{inj-lifting},  Proposition \ref{inj-fact-reedy} and Lemma \ref{lem-fib-inj}, one can easily verify that the axioms of a model structure hold (see for example \cite[Definition 1.1.3]{Hov-model}).
\end{proof}
\begin{cor}
Let $\Ub: \ag \to \M$ be a right Quillen functor. 
\begin{enumerate} 
\item We have a Quillen adjunction  $\Piun: \muaij  \rightleftarrows  \ag : \iota$, where $\Piun$ is left Quillen. 
\item We also have a Quillen adjunction  $L_1: \ag  \rightleftarrows \muaij : \Piun$, where $L_1$ is left Quillen. 
\item We have a Quillen adjunction $\Gamma: \armij  \rightleftarrows \muaij: \Piar$, where $\Gamma$ is left Quillen.
\item The functors $\Piun$ and $\Pio$ preserve the weak equivalences.
\end{enumerate}
\end{cor}
\begin{proof}
The functor $\Piun: \muaij \to \ag$ preserves (trivial) cofibrations and (trivial) fibrations. So clearly it's simultaneously a left Quillen functor and a right Quillen functor. This gives the first two assertions. \ \\

For Assertion $(3)$, it suffices to observe that if $\sigma=[\sg_0, \sg_1]$ is an injective trivial fibration then $\Piar(\sg)=[\sg_0, \Ub(\sg_1)]$ is an injective (trivial) fibration in $\armij$ by definition ($\Ub$ being a right Quillen functor). Therefore $\Piar$ is right Quillen which means automatically that $\Gamma$ is left Quillen. \ \\
The last assertion follows from the definition of a level-wise weak equivalences. 
\end{proof}

\paragraph{Cofibrantly generated}
Let's now assume that $\M$ and $\ag$ are cofibrantly generated in the sense of \cite[Definition 2.1.17]{Hov-model}. We will denote by $\Iam$ and $\Jam$ the respective sets of the generating cofibrations and trivial cofibrations for $\M$. And similarly let $\Iag$ and $\Jag$ be the respective sets of generating cofibrations and trivial cofibrations of $\ag$. Given  $s: A \to B \in \Arr(\M)$, we've introduced in Notation \ref{nota-alpha} a map $\alpha_s: s \xrw{(s,\Id)} \Id_{B}$ in $\Arr(\M)$. 
\begin{thm}
If $\ag$ and $\M$ are cofibrantly generated (resp. combinatorial), then $\muaij$ is cofibrantly generated (resp. combinatorial). 
\begin{enumerate}
\item The set $\I_{\muaij}=L_1(\Iag) \coprod \Gamma(\alpha_{\Iam})$ is a set of generating cofibrations in $\muaij$. 
\item The set $\Ja_{\muaij}=L_1(\Jag) \coprod \Gamma(\alpha_{\Jam})$ is a set of generating trivial cofibrations  in $\muaij$. 
\end{enumerate}
\end{thm}

\begin{proof}
Corollary \ref{cor-presentable} says that the category $\mdu$ is locally presentable. So we just need to prove that it is cofibrantly generated if we want to prove that it's a combinatorial model category. \ \\

A map $\sg=[\sg_0,\sg_1]$ has the RLP with respect to all maps in $L_1(\Iag) \coprod \Gamma(\alpha_{\Iam})$ if and only if it is simultaneously $L_1(\Iag)$-injective and $\Gamma(\alpha_{\Iam})$-injective. One the one hand, using the adjunction $ L_1 \dashv \Piun$,  $\sg=[\sg_0,\sg_1]$ is $L_1(\Iag)$-injective if and only if $\sg_1$ is $\Iag$-injective, if and only if, $\sg_1$ is a trivial fibration. \ \\

On the other hand, using the adjunction $ \Gamma \dashv \Piun$,  $\sg: \Pc \xrw{[\sg_0,\sg_1]} \Qc$ is $\Gamma(\alpha_{\Iam})$-injective if and only if $\Piar(\sg)=[\sg_0,\Ub(\sg_1)]$ is $\alpha_{\Iam}$-injective. By Lemma \ref{lifting-lem}, $\Piar(\sg)=[\sg_0,\Ub(\sg_1)]$ is $\alpha_{\Iam}$-injective if and only if  $\delta:\Pa_0 \to \Ub(\Pa_1) \times_{\Ub(\Qa_1)} \Qa_0$ is $\Iam$-injective, if and only if $\delta$ is a trivial fibration. This gives Assertion $(1)$. The second assertion is proved the same way. 
\end{proof}
\subsubsection{Localizing the injective model structure}
The functor $\Piun: \mua \to \ag$ preserves the weak equivalences. Let $\W_{\mua}$ be the class of level-wise weak equivalences and let $\W_{\ag}$ be the class of weak equivalences in $\ag$. Denote by $\W_L=\Piun^{-1}(\W_{\ag})$. Clearly $\W_{\mua} \subseteq \W_L$.
\begin{df}
Elements in $\W_L$ will be called \emph{new weak equivalences}. 
\end{df}
\begin{prop}\label{prop-aij-new}
For any $i \in \Iam$, $\Gamma(\alpha_i)$ is an injective cofibration and a new weak equivalence. 
\end{prop}

\begin{proof}
If $i: U \to V$, then $\alpha_i: i \xrw{(i,\Id_V)} \Id_V$ is an injective cofibration in $\Arr(\M)$  whose components are $i$ and $\Id_V$. Then $\Gamma(\alpha_i): \Gamma(i) \xrw{[i, \Fb(\Id_V)]} \Gamma(\Id_V)$ is an injective cofibration in $\mua$ whose components are $i$ and $\Fb(\Id_V)$,  by definition of $\Gamma$ (see Proposition \ref{prop-gamma}). $\Piun(\Gamma(\alpha_i)) = \Fb(\Id_V)$ is an isomorphism, thus a weak equivalence and: 
$$\Piun(\Gamma(\alpha_i)) \in  \W_{\ag}\Leftrightarrow \Gamma(\alpha_i) \in \Piun^{-1}(\W_{\ag})= \W_L.$$ 
\end{proof}
We are now in the hypotheses of Smith's theorem to localize the injective model structure to get a model structure on $\mua$ with the class of new weak equivalences $\W_L$ and the same set of generating cofibrations $\I_{\muaij}$. We will use Smith's theorem and its consequences that can be found in  Beke \cite[Theorem 4.1, Proposition 4.2, Proposition 4.4]{Beke_2}. 

\begin{thm}\label{main-thm}
Let $\Ub: \ag \to \M$ be a right Quillen functor between combinatorial model categories. 
\begin{enumerate}
\item The data $(\I_{\muaij}, \W_L)$ define a combinatorial model structure on $\mua$ that will be denoted by $\muaij^+$ and which may be described as follows.
\begin{itemize}
\item A map $\sg=[\sg_0,\sg_1]$ is a weak equivalence if it's in $\W_L$, that is if $\sg_1$ is a weak equivalence in $\ag$.
\item The cofibrations are the injective cofibrations.
\item The fibrations are the maps satisfying the RLP with respect to every map that is simultaneously a cofibration and a weak equivalence. 
\end{itemize}
\item We have a Quillen equivalence $\Piun: \muaij^+  \rightleftarrows \ag : \iota$, where $\Piun$ is left Quillen. 
\item Any fibrant object $\Fc=[\F_0,\F_1,\pif]$ in $\muaij^+$ satisfies the Segal condition, that is $\pif: \F_0 \to \Ub(\F_1)$ is a weak equivalence.
\end{enumerate}
\end{thm}
\begin{proof}
Assertion $(1)$ is a direct consequence of a result of  Beke \cite[Proposition 4.2]{Beke_2}, which is itself a consequence of Smith recognition theorem. \ \\
The Quillen equivalence is given by Proposition 4.4 in Beke \cite{Beke_2} since $\iota:\ag \to \mua$ exhibits $\ag$ as a full reflective subcategory and the functor $\Piun$ preserves the weak equivalences. This gives Assertion $(2)$. \ \\

For Assertion $(3)$ we proceed by adjunction. If $\Fc$ is fibrant in $\muaij^+$ , then the  unique map $\Fc \to \ast$ has the RLP with respect to all trivial cofibrations. And by Proposition \ref{prop-aij-new}, all elements of $\Gamma(\alpha_{\Iam})$ are trivial cofibrations in $\muaij^+$,  therefore $\Fc \to \ast$ has the RLP with respect to every element in $\Gamma(\alpha_{\Iam})$ i.e, $\Fc$ is $\Gamma(\alpha_{\Iam})$-injective.  Now by Proposition \ref{detect-segal}, we know that $\Fc$ satisfies the Segal condition ($\pif$ is a trivial fibration) and the assertion follows. 
\end{proof}
The following gives an explicit description of the fibrations.
\begin{thm}
With the previous definition we have the following. 
\begin{enumerate}
\item The set $\Ja_{\muaij}^+=\Ja_{\muaij} \coprod \Gamma(\alpha_{\Iam})$ is a set of generating trivial cofibrations.
\item  A map $\theta: \Pc \xrw{[\sg_0,\sg_1]} \Qc$ is fibration in  $\muaij^+$, if it's an injective fibration such that the induced map $\delta:\Pa_0 \to \Ub(\Pa_1) \times_{\Ub(\Qa_1)} \Qa_0$ is a trivial fibration in $\M$.
\end{enumerate}
\end{thm}

\begin{proof}
The second assertion is a consequence of the first considering the fact that a map $\theta$  is $\Gamma(\alpha_{\Iam})$-injective if and only if the map $\delta$ is $\Iam$-injective, thus a trivial fibration in $\M$ ( see Lemma \ref{lifting-lem}).\ \\
 
To get the first assertion, we shall prove that $\cof(\I_{\muaij}) \cap \W_L \subseteq \cof(\Ja_{\muaij}^+).$\\
Let $\sg: \Fc \xrw{[\sg_0,\sg_1]} \Gc$ be an element in  $\cof(\I_{\muaij}) \cap \W_L$. Apply the small object argument to factor $\sg$ as a relative $\Jcij$-cell complex followed by a $\Jcij$-injective map, $\sg= r(\sg) \circ l(\sg)$, where $r(\sg)=[r(\sg)_0, r(\sg)_1]$ and $l(\sg)=[l(\sg)_0, l(\sg)_1]$. If we look at this factorization in $\ag$ it gives a factorization of $\sg_1= r(\sg)_1 \circ l(\sg)_1$ with $ l(\sg)_1 \in cell(\Jag)$ and  $r(\sg)_1$ is $\Jag$-injective map, thus a fibration. By assumption $\sg \in \W_L$ which means that $\sg_1$ is a weak equivalence in $\ag$ and so is $l(\sg)_1$; therefore by $3$-for-$2$, $r(\sg)_1$ is also a weak equivalence in $\ag$. So we find that $r(\sg)_1$ is a trivial fibration i.e, $r(\sg)$ is $L_1(\Iag)$-injective. \ \\

On the other hand, part of $r(\sg)$ being $\Jcij$-injective implies that $r(\sg)$ is in particular $\Gamma(\alpha_{\Iam})$-injective since $\Gamma(\alpha_{\Iam}) \subset \Jcij$. So in the end $r(\sg)$ is $\I_{\muaij}$-injective, which means that $r(\sg)$ is an injective trivial fibration. The map $\sg$ is an injective cofibration so it posses the LLP with respect to $r(\sg)$. Putting this together with the factorization $\sg= r(\sg) \circ l(\sg)$, the retract argument as in \cite[Lemme 1.1.9]{Hov-model} implies that  $\sg$ is a retract of $l(\sg) \in cell(\Jcij)$, thus $\sg \in \cof(\Jcij)$. 
\end{proof}

\subsubsection{Projective model structure}

\begin{thm}\label{proj-thm}
Let $\Ub : \ag \to \M$ be a right Quillen functor. Then with the previous definitions, the following hold. 
There is a model structure on the category $\mdua= \mua$ which may be described as follows. 
\begin{itemize}
\item A map $\sigma: \Fc \to \Gc$ is a weak equivalence if and only if it's a level-wise weak equivalence.
\item A map $\sigma: \Fc \to \Gc$ is cofibration if it's a projective cofibration. 
\item A map $\sigma: \Fc \to \Gc$ is fibration if it's a projective (= level) fibration. 
\end{itemize}
We will denote this model category  by $\muapj$. 
\end{thm}

\begin{proof}
The class of level-wise weak equivalences clearly satisfies the $2$-out-of-$3$ property. The three classes of cofibrations, fibrations and weak equivalences are closed under composition and retracts. With Proposition \ref{proj-lifting},  Proposition \ref{proj-fact-reedy} and Lemma \ref{lem-cof-proj}, one can easily verify that the axioms of a model structure hold.
\end{proof}
\begin{cor}
Let $\Ub: \ag \to \M$ be a right Quillen functor. 
\begin{enumerate} 
\item We have a Quillen adjunction  $\Piun: \muapj  \rightleftarrows  \ag : \iota$, where $\Piun$ is left Quillen. 
\item We also have a Quillen adjunction  $L_1: \ag  \rightleftarrows \muapj : \Piun$, where $L_1$ is left Quillen. 
\item We have a Quillen adjunction $\Gamma: \armpj  \rightleftarrows \muapj: \Piar$, where $\Gamma$ is left Quillen.
\item The functors $\Piun$ and $\Pio$ preserve the weak equivalences.
\end{enumerate}
\end{cor}
\begin{proof}
The functor $\Piun: \muapj \to \ag$ preserves (trivial) cofibrations and (trivial) fibrations. So clearly it's simultaneously a left Quillen functor and a right Quillen functor. This gives the first two assertions. \ \\

For Assertion $(3)$, it suffices to observe that if $\sigma=[\sg_0, \sg_1]$ is a projective trivial fibration then $\Piar(\sg)=[\sg_0, \Ub(\sg_1)]$ is a projective (trivial) fibration in $\armpj$ by definition ($\Ub$ being a right Quillen functor). Therefore $\Piar$ is right Quillen which means automatically that $\Gamma$ is left Quillen. \ \\
The last assertion follows from the definition of a level-wise weak equivalences. 
\end{proof}

\paragraph{Cofibrantly generated}
Let's now assume that $\M$ and $\ag$ are cofibrantly generated as before. We have a similar theorem as in the injective model structure. 
\begin{thm}
If $\ag$ and $\M$ are cofibrantly generated (resp. combinatorial), then $\muapj$ is cofibrantly generated (resp. combinatorial). 
\begin{enumerate}
\item The set $\I_{\muapj}=L_1(\Iag) \coprod \Gamma(L_0(\Iam))= L_1(\Iag) \coprod \Fb^+(\Iam)$ is a set of generating cofibrations in $\muapj$. 
\item The set $\Ja_{\muapj}=L_1(\Jag) \coprod \Gamma(L_0(\Jam))=L_1(\Jag) \coprod \Fb^+(\Jam)$ is a set of generating trivial cofibrations  in $\muapj$. 
\end{enumerate}
\end{thm}

\begin{proof}
Like in the injective case, Corollary \ref{cor-presentable} says that the category $\mdu$ is locally presentable. So  it remains to prove that it is cofibrantly generated to prove that it's a combinatorial model category. \ \\

A map $\sg=[\sg_0,\sg_1]$ has the RLP with respect to all maps in $L_1(\Iag) \coprod \Gamma(L_0(\Iam))$ if and only if it is simultaneously $L_1(\Iag)$-injective and $\Gamma(L_0(\Iam))$-injective. One the one hand, using the adjunction $ L_1 \dashv \Piun$,  $\sg=[\sg_0,\sg_1]$ is $L_1(\Iag)$-injective if and only if $\sg_1$ is $\Iag$-injective, if and only if, $\sg_1$ is a trivial fibration. \ \\

On the other hand, using the adjunction $ \Fb^+ \dashv (\Ev_0 \circ \Piar)$,  $\sg=[\sg_0,\sg_1]$ is $\Fb^+(\Iam)$-injective if and only if $\Ev_0\Piar(\sg)=\sg_0$ is $\Iam$-injective, if and only if $\sg_0$ is a trivial fibration. This gives Assertion $(1)$. The second assertion is proved the same way. 
\end{proof}
\subsubsection{Localizing the projective model structure}
Given  $s: A \to B \in \Arr(\M)$, we've introduced in Notation \ref{nota-zeta} a map $\zeta_s: s \xrw{(s,\Id)} \Id_{B}$ in $\Arr(\M)$ and we've defined the universal projective localizing set $\zeta_{\I}$. Let $\kbi_{proj}= \Gamma(\zeta_{\Iam})=\{ \Gamma(\zeta_i) \}_{i \in \Iam}$.\ \\ 

Like in the injective model structure, we consider the same class $\W_L=\Piun^{-1}(\W_{\ag})$ of  \emph{new weak equivalences}. 

\begin{prop}\label{prop-aproj-new}
For any $i \in \Iam$, $\Gamma(\zeta_i)$ is a projective cofibration and a new weak equivalence. 
\end{prop}

\begin{proof}
If $i: U \to V$, then $\zeta_i: i \xrw{(i,j_0)} j_1 $ is an projective cofibration in $\Arr(\M)$ by construction (see Proposition \ref{alph-zeta}). The components of $\zeta_i$ are $i$ and $j_0$. With the notation introduced in Notation \ref{nota-zeta}, $\Gamma(\zeta_i): \Gamma(i) \xrw{[i, \Fb(j_0)]} \Gamma(j_1)$ is therefore a projective cofibration in $\mua$ whose components are $i$ and $\Fb(j_0)$,  since $\Gamma$ is left Quillen (see Proposition \ref{prop-gamma}). $\Piun(\Gamma(\zeta_i)) = \Fb(j_0)$ is a trivial cofibration because $j_0$ is trivial cofibration and $\Fb$ is left Quillen. In particular, $\Fb(j_1)=\Piun(\Gamma(\zeta_i))$ is a weak equivalence; that is:
$$\Piun(\Gamma(\zeta_i)) \in  \W_{\ag}\Leftrightarrow \Gamma(\zeta_i) \in \Piun^{-1}(\W_{\ag})= \W_L.$$ 
\end{proof}
Just like in the injective case, we are in the hypotheses of Smith's theorem to localize the projective model structure to get a model structure on $\mua$ with the class of new weak equivalences $\W_L$ and the same set of generating cofibrations $\I_{\muapj}$. 
\begin{thm}\label{main-thm-proj}
Let $\M$ be a tractable model category and let $\Ub: \ag \to \M$ be a right Quillen functor between combinatorial model categories.
\begin{enumerate}
\item The data $(\I_{\muapj}, \W_L)$ define a combinatorial model structure on $\mua$ that will be denoted by $\muapj^+$ and which may be described as follows.
\begin{itemize}
\item A map $\sg=[\sg_0,\sg_1]$ is a weak equivalence if it's in $\W_L$, that is if $\sg_1$ is a weak equivalence in $\ag$.
\item The cofibrations are the projective cofibrations.
\item The fibrations are the maps satisfying the RLP with respect to every map that is simultaneously a cofibration and a weak equivalence. 
\end{itemize}
\item We have a Quillen equivalence $\Piun: \muapj^+  \rightleftarrows \ag : \iota$, where $\Piun$ is left Quillen. 
\item Any fibrant object $\Fc=[\F_0,\F_1,\pif]$ in $\muapj^+$ satisfies the Segal condition, that is $\pif: \F_0 \to \Ub(\F_1)$ is a weak equivalence.
\end{enumerate}
\end{thm}
\begin{proof}
Assertion $(1)$ is a direct consequence of Proposition 4.2 in Beke \cite{Beke_2}, which is itself a consequence of Smith recognition theorem. \ \\
The Quillen equivalence is given by Proposition 4.4 in Beke \cite{Beke_2} since $\iota:\ag \to \mua$ exhibits $\ag$ as a full reflective subcategory and the functor $\Piun$ preserves the weak equivalences. This gives Assertion $(2)$. \ \\

For Assertion $(3)$ we proceed by adjunction. If $\Fc$ is fibrant in $\muapj^+$ , then the  unique map $\Fc \to \ast$ has the RLP with respect to all trivial cofibrations. And by Proposition \ref{prop-aproj-new}, all elements of $\Gamma(\zeta_{\Iam})$ are trivial cofibrations in $\muapj^+$,  therefore $\Fc \to \ast$ has the RLP with respect to every element in $\Gamma(\zeta_{\Iam})$ i.e, $\Fc$ is $\Gamma(\zeta_{\Iam})$-injective.  Now by Proposition \ref{detect-segal-pj}, we know that $\Fc$ satisfies the Segal condition ($\pif$ is a weak equivalence between fibrant objects) and the assertion follows. 
\end{proof}

\begin{cor}
The identity functor $\Id: \muapj^+ \to \muaij^+$ is a left Quillen functor which is a Quillen equivalence.
\end{cor}

\begin{proof}
It's the identity functor and we have the same class of new weak equivalences. Moreover any projective cofibration is an injective cofibration. 
\end{proof}

Unlike the injective model structure, it's difficult to give an explicit description of the fibrations in general without any assumption of (left)  properness on $\M$. This will require some work that will make the paper very long. We will do this in a future work.

\begin{rmk}
We will close this section with some observations.
\begin{enumerate}
\item If $\M$ is tractable, we can show that $\muaij^+$ is the left Bousfield localization with respect $\Gamma(\alpha_{\Iam})$. Indeed, elements in $\Gamma(\alpha_{\Iam}))$ are maps between cofibrant objects so there is no need to take a cofibrant approximation of these maps to define their image under a left derived functor.
\item There are other model structure on $\mdua$ that one can get from the injective and the projective model structure on $\arm$ with the adjunction $\Gamma \dashv \Piar$. They will be considered in a future work.  
\end{enumerate}
\end{rmk}

\section{Generating the Segal conditions}

\subsection{Factorization lemma and the small object argument} 

Recall that for any object $[\F]=[\F_0,\F_1,\pif]$ we have $\F_1= \Piun([\F])$. In the adjunction $\iota: \ag \rightleftarrows \mua : \Piun$, the unit $\eta_1([\F]): [\F] \to \iota(\F_1)$ is given by the pair $[\pi_{\F}, \Id_{\F_1}] \in \Hom_{\M}(\F_0,\Ub(\F_1)) \times \Hom_{\ag}(\F_1,\F_1)$ (see Proposition \ref{prop-unit-adj}).  \\

Consider the set $\kbi= \{\Gamma(\alpha_i) \}_{i \in \I}$ introduced previously. 
Our goal here is to apply the \emph{small object argument} (see \cite{Dwyer_Spalinski}, \cite{Hov-model}) with respect to $\kbi$ and show that the unit  $$\eta_1([\F]): [\F] \to \iota(\F_1)$$ can be factored as a relative $\kbi$-cell complex followed by a $\kbi$-injective map. Now observe that since the codomain of each $\alpha_i$ is the identity, the codomain of $\Gamma(\alpha_i)$ is also the identity. Moreover as $\Gamma$ is a left adjoint, it preserves any kind of colimits; in particular it distributes over coproducts. It follows that the codomain of any coproduct $ \sqcup_{N}\Gamma(\alpha_i) \cong \Gamma(\sqcup_{N}\alpha_i)$ of elements of $\kbi$ is a essentially a coproduct of identities. Following  Proposition \ref{prop-limit-colimit}, if we take the pushout of such a product along an attaching map $Domain(\sqcup_{N}\Gamma(\alpha_i)) \to [\F]$, we won't change the object $\F_1$ since $\Piun$ preserves colimits and the pushout of an isomorphism is an isomorphism. \ \\ 

It follows from this description that the factorization given by the small object argument, amounts to factor the map $\pif: \F_0 \to \Ub(\F_1)$ as a relative $\I$-cell complex followed by an $\I$-injective map, that is a trivial fibration, leaving $\F_1$ essentially unchanged.\\

We shall now describe explicitly the pushout of a single element of $\kbi$. 

Given a morphism $\theta_1: B \to \Ub(\F_1)$ in $\M$ we will denote by $\varrho(\theta_1): \Fb B \to \F_1$ its adjoint-transpose in the adjunction $\Ub: \ag \rightleftarrows \M : \Fb$.
\begin{lem}\label{lem-cobase-change-mua}
Let  $\theta=(\theta_0,\theta_1) : s \to \pif$ be a map in $\Ar(\M)$ whose adjoint-transpose $\varphi(\theta)=[\theta_0, \varrho(\theta_1)]:\Gamma(s) \to [\F] $ is  represented by the commutative square:
 \[
 \xy
(0,18)*+{A}="W";
(0,0)*+{\Ub(\Fb B)}="X";
(30,0)*+{\Ub(\F_1)}="Y";
(30,18)*+{\F_0}="E";
{\ar@{->}^-{\Ub(\varrho(\theta_1))}"X";"Y"};
{\ar@{->}_-{\eta_B \circ s}"W";"X"};
{\ar@{->}^-{\theta_0}"W";"E"};
{\ar@{->}^-{\pif}"E";"Y"};
\endxy
\]
Then the following hold.
\begin{enumerate}
\item The pushout of  $\Gamma(\alpha_s): \Gamma(s) \to \Gamma(\Id_B)$ along $\varphi(\theta)$ is the map $$l(\Gamma \alpha_s)=[l_s(\pif), \Id_{\F_1}]: [\F] \to [\F] \cup^{\Gamma(s)} \Gamma(\Id_B)$$ displayed by the commutative square:
 \[
 \xy
(0,18)*+{\F_0}="W";
(0,0)*+{\Ub(\F_1)}="X";
(30,0)*+{\Ub(\F_1)}="Y";
(30,18)*+{\F_0 \cup^{A} B}="E";
{\ar@{=}^-{\Id}"X";"Y"};
{\ar@{->}^-{\pif}"W";"X"};
{\ar@{->}^-{l_s(\pif)}"W";"E"};
{\ar@{->}^-{r_s(\pif)}"E";"Y"};
\endxy
\]
where $l_s(\pif)$ is the pushout of $s$ along $\theta_0$ and $r_s(\pif)$ is the canonical map induced by universal property of the pushout.  
\item The pushout-object $[\Ea]= [\F] \cup^{\Gamma(s)} \Gamma(\Id_B)$ is given by $[\Ea]=[\Ea_0, \Ea_1, \pi_{\Ea}]$ with 
$$ \Ea_0=\F_0 \cup^{A} B, \quad \Ea_1= \F_1, \quad \pi_{\Ea}=r_s(\pif)  $$
\item We have a factorization of $\pif$: 
$\F_0 \xrightarrow{l_s(\pif)} \F_0 \cup^{A} B   \xrightarrow{r_s(\pif)}\Ub(\F_1) $
that induces in turn a factorization of the unit 
$$\eta_1([\F]): [\F] \xrightarrow{[\pif, \Id_{\F_1}]} \iota(\F_1)= [\F] \xrightarrow{[l_s(\pif), \Id_{\F_1}]} [\Ea] \xrightarrow{\eta_1([\Ea])} \iota(\F_1) $$

\end{enumerate}
\end{lem}

\begin{proof}
This directly follows from Proposition \ref{prop-limit-colimit} where we explain how we compute colimits, and in particular pushouts in  $\mua$. The pushout square is represented by a commutative cube in $\M$ that we display below for the reader's convenience. 

\[
\xy
(-30,20)*+{A}="A";
(20,20)+(0,10)+(0,-4)*+{\F_0}="B";
(-30,0)*+{\Ub(\Fb B)}="C";
(20,0)+(0,10)+(0,-4)*+{\Ub(\F_1)}="D";
{\ar@{->}^{\theta_0}"A";"B"};
{\ar@{->}_{\eta_B \circ  s}"A"+(0,-2);"C"};
{\ar@{.>}^<<<<<<<<<<<<<{\pif}"B"+(0,-2);"D"};
{\ar@{.>}^>>>>>>>>>{\Ub(\varrho(\theta_1))}"C";"D"};
(-30,20)+(-15,-10)+(40,0)+(0,5)*+{B}="E";
(20,20)+(-15,-10)+(40,0)+(0,10)+(0,5)+(0,-4)*+{\F_0 \cup^{A} B }="F";
(-30,0)+(-15,-10)+(40,0)+(0,5)*+{\Ub(\Fb B)}="G";
(20,0)+(-15,-10)+(40,0)+(0,10)+(0,5)+(0,-4)*+{\Ub(\F_1)}="H";
{\ar@{->}^{}"E";"F"};
{\ar@{->}_>>>>{\eta_B}"E"+(0,-2);"G"};
{\ar@{->}^{r_s(\pif)}"F";"H"};
{\ar@{->}^{}"G";"H"};
{\ar@{->}^{s}"A";"E"};
{\ar@{->}^{l_s(\pif)}"B";"F"};
{\ar@{=}^{}"C";"G"};
{\ar@{=}^{\Id}"D";"H"};
\endxy
\]  

In this cube the lower and upper faces are pushout squares in $\M$. Clearly we have a factorization of $\pif$:
$$\F_0 \xrightarrow{l_s(\pif)} \F_0 \cup^{A} B   \xrightarrow{r_s(\pif)}\Ub(\F_1).$$

The map $ B \to \F_0 \cup^{A} B $ is the canonical map going to the colimit and it gives a factorization of $\theta_1$:
$$ \theta_1 = B \xrightarrow{} \F_0 \cup^{A} B   \xrightarrow{r_s(\pif)} \Ub(\F_1). $$
The remaining assertions are clear. 
\end{proof}

\subsection{Identifying a pushout}
We also need a lemma that reverses the process.

\begin{lem}\label{lem-push-retract}
Let $[\F]=[\F_0,\F_1,\pif]$ be an object of $\mua$ and consider a factorization of $\pif= r(\pif) \circ l(\pif)$ displayed as follows:
$$\pif: \F_0 \to \Ub(\F_1) = \F_0 \xrightarrow{l(\pif)} E_0   \xrightarrow{r(\pif)} \Ub(\F_1)$$
Let $[\Ea]=[\Ea_0, \Ea_1, \pi_{\Ea}]$ be the object of $\mua$ defined by 
$$ \Ea_0=E_0, \quad \Ea_1= \F_1, \quad \pi_{\Ea}=r(\pif).$$
Then the following hold.
\begin{enumerate}
\item  The factorization induces a tautological map: $\psi_{\F}(l,r): [\F] \to [\Ea]$ in $\mua$:
\[
 \xy
(0,18)*+{\F_0}="W";
(0,0)*+{\Ub(\F_1)}="X";
(30,0)*+{\Ub(\F_1)=\Ub(\Ea_1)}="Y";
(30,18)*+{E_0=\Ea_0}="E";
{\ar@{=}^-{\Id}"X";"Y"};
{\ar@{->}^-{\pif}"W";"X"};
{\ar@{->}^-{l(\pif)}"W";"E"};
{\ar@{->}^-{r(\pif)}"E";"Y"};
\endxy
\] 
\item If $l(\pif)$ is a cobase of some map $s : A \to B$ then the map $\psi_{\F}(l,r)$ is a cobase change of the map $\Gamma(\alpha_s)$.
\end{enumerate}
\end{lem}

\begin{proof}
The first assertion is clear. For Assertion $(2)$ we just need to define the attaching map $ \Gamma (s) \to [\F]$ in $\mua$ that will be used to get a pushout data:
$$\Gamma(\Id_B) \xleftarrow{\Gamma(\alpha_s)} \Gamma (s) \to [\F].$$
 Now by hypothesis the map $l(\pif)$ fits in a pushout square: 
 \[
 \xy
(0,18)*+{A}="W";
(0,0)*+{B}="X";
(30,0)*+{E_0\cong \F_0 \cup^{A} B}="Y";
(30,18)*+{\F_0}="E";
{\ar@{->}^-{}"X";"Y"};
{\ar@{->}^-{s}"W";"X"};
{\ar@{->}^-{}"W";"E"};
{\ar@{->}^-{l(\pif)}"E";"Y"};
\endxy
\]

We extend this commutative square with the map $r(\pif): E_0 \to \Ub(\F_1)$.  This gives a commutative square that defines a map $\theta : s \to \pif$ in the category $\Ar(\M)$. The adjoint transpose of $\theta$ gives the attaching map $\varphi(\theta)=[\theta, \varrho(\theta_1)]: \Gamma(\alpha_s) \to [\F]$ as in the previous lemma. And by the previous lemma one clearly sees that $\psi_{\F}(l,r)$ is the cobase change of $\Gamma(\alpha_s)$ along $\varphi(\theta)$.\ \\
\end{proof}

\subsection{Cell complexes} The last piece of ingredient we need is the following lemma which is a consequence of the previous ones. It's the general form of the second assertion in Lemma \ref{lem-push-retract}.

\begin{lem}\label{lem-cx}
Let $S$ be a small set of maps of $\M$ and denote by $\alpha_S = \{ \alpha_s\}_{s \in S}$. Let $[\F]=[\F_0,\F_1,\pif]$ be an object of $\mua$ and consider a factorization of $\pif= r(\pif) \circ l(\pif)$ displayed as follows:
$$\pif: \F_0 \to \Ub(\F_1) = \F_0 \xrightarrow{l(\pif)} E_0   \xrightarrow{r(\pif)} \Ub(\F_1)$$
Let $[\Ea]=[\Ea_0, \Ea_1, \pi_{\Ea}]$ be the object of $\mua$ defined by 
$$ \Ea_0=E_0, \quad \Ea_1= \F_1, \quad \pi_{\Ea}=r(\pif).$$

If the morphism $l(\pif)$ is a relative $S$-cell complex, then the map $\psi_{\F}(l,r) : [\F] \to [\Ea]$ in Lemma \ref{lem-push-retract} is a relative $\Gamma(\alpha_S)$-cell complex.
\end{lem}
\begin{proof}
This is just a repetition of the second assertion of Lemma \ref{lem-push-retract}. Indeed, the map $l(\pif)$ is a transfinite composition of maps $l(\pif)_k$; where each of them is a cobase change of some map $s_k \in S$.  In particular for each $k$ we have a factorization of $\pif$ : $\pif=r(\pif)_k\circ l(\pif)_k$.

By the same construction as in Lemma \ref{lem-push-retract},  this factorization defines maps $\psi_{\F}(l_k,r_k): [\F]\to \Ea_k$ whose transfinite composite is the map $\psi_{\F}(l,r)$.  Thanks to Lemma \ref{lem-push-retract}  we know that each $\psi_{\F}(l_k,r_k)$ is a pushout of $\Gamma(\alpha_{s_k})$.

We see that $\psi_{\pif}(l,r)$ is a transfinite composition of maps that are pushouts of elements of $\Gamma(\alpha_S)$, thus it's a relative $\Gamma(\alpha_S)$-cell complex as claimed.
\end{proof}

\begin{rmk}
In the previous lemma we can generalize the statement when the map $\F_0 \xrightarrow{l(\pif)} E_0$ is a retract of a map $l'(\pif):\F_0' \xrightarrow{} E'_0 $ that is a relative $S$-cell complex. In that case the map $\psi_{\F}(l,r) : [\F] \to [\Ea]$ would be a retract of relative $\Gamma(\alpha_s)$-cell complex. 

To see this, first observe that from the definition of a retract, there is map $a:E_0 \to E'_0$ and a retraction $b:E_0' \to E_0$ such that $b\circ a= \Id_{E_0}$.  Similarly we have a map $u:\F_0 \to \F_0'$ and a retraction $v:\F'_0 \to \F_0$ such that $v \circ u = \Id_{\F_0}$. If we set $\pif'= r(\pif) \circ b \circ l'_{\pif}$ we get an object $[\F'] \in \mua$ to which we can apply the construction of the previous lemma and get a map $\psi_{\F'}(l',r')$.  
It's not hard to see that the map $\pif$ is a retract of $\pi_{\F'}$ and that $\psi_{\F}(l,r)$ is a retract  of  $\psi_{\F'}(l',r')$.
\end{rmk}

\subsection{Factorization Theorem}
We give here one of the central results of this section. We show that the unit in the adjunction $\Piun \dashv \iota$ can be factored as as relative  $\kbi$-cell complex followed by a level-wise trivial fibration. 
\begin{thm}\label{thm-factorization-unit}
Let $\M$ be a cofibrantly generated model category with $\I$ the generating set of cofibrations . Consider the sets $\alpha_{\I}=  \{ \alpha_i\}_{i \in \I} $  and let $\kbi= \Gamma(\ali)$. Then for any object $\Fc \in \mua$ the following hold.

\begin{enumerate}
\item  We can factor the unit $\eta_{1,\Fc}: \Fc \to \iota(\Piun(\Fc))$ as a relative $\kbi$-cell complex followed by a map that is a level-wise trivial fibration:
$$\Fc \xrightarrow{\psi} \Ec \xrightarrow{}    \iota(\Piun(\Fc)).$$
\item The map  $\Ec \xrightarrow{}    \iota(\Piun(\Fc))$ is the unit of the adjunction for the object $\Ec$. In particular the unit $\eta_{1,\Ec}$ is a level-wise weak equivalence in $\mua$.
\end{enumerate}

\end{thm}

\begin{proof}
Apply the small object argument in $\M$, to get a factorization of $\pif$ as a relative $\I$-cell complex followed by an $\I$-injective map; that is a trivial fibration: 
$$\pif: \F_0 \to \Ub(\F_1) = \F_0 \xhookrightarrow{l(\pif)} E_0    \xtwoheadrightarrow[\sim]{r(\pif)}  \Ub(\F_1)$$

By the same construction as in Lemma \ref{lem-push-retract} we define the object $\Ec \in \mua$ with:
$$ \Ea_0=E_0, \quad \Ea_1= \F_1, \quad \pi_{\Ea}=r(\pif).$$
In particular $\Piun(\Ec)=\F_1$ and the unit of the adjunction for the object $\Ec$ is the map $\eta_{1,\Ec}=[r(\pif), \Id_{\F_1}]: \Ec \to \iota(\Piun(\Fc))$ which is a level-wise trivial fibration since $r$ and $\Id_{\F_1}$ are. \ \\

Let's denote simply by $\psi$, the map $\psi(l,r)=[l(\pif),\Id_{\F_1}]: \Fc \to \Ec$ constructed  in Lemma \ref{lem-push-retract}. Now it's clear that the composite $\eta_{1,\Ec} \circ \psi$ is just $\eta_{1,\Fc}$ since:
$$\eta_{1,\Ec} \circ \psi=[r(\pif), \Id_{\F_1}] \circ [l(\pif),\Id_{\F_1}]= [r(\pif)\circ l(\pif), \Id_{\F_1}]= [\pif, \Id_{\F_1}]=\eta_{1,\Fc}.$$
This completes the proof of the Lemma.
\end{proof}
\newpage

\section{Stable Homotopy Theory}
In this section we assume that $\ag=\M$ and that $\Ub: \M \to \M$ is an endofunctor which is right Quillen. The example that we shall keep in mind is $\Ub= \Omega: \sset_{\ast} \to \sset_{\ast}$. The left adjoint $\Fb: \M \to \M$ is to be thought as the suspension functor $\Sigma = S^1 \wedge -$. We will write $\Ub^n$ for the composite of $\Ub$ with itself $n$ times and similarly we have $\Fb^n$; with the convention $\Fb^0= \Id_{\M}= \Ub^0$.\ \\

There many references in the literature on stable homotopy theory. Classical ones include Adams \cite{Adams_stable}, Bousfield-Friedlander \cite{Bous_Fried}, Lima \cite{Lima_dual}. Modern foundations are to be found in  Elmendorf-Kriz-Mandell-May (EKMM) \cite{EKMM_book}, Hovey \cite{Hov_stable}, Schwede \cite{Schwede_stable} and the many references therein. In the world of $\infty$-categories we shall refer the reader to Lurie \cite{Lurie_Stable}.

\begin{df}\label{df-bous_fried}
A spectrum $X$ is a sequence of pointed simplicial sets $(X_n)_{n \in \N}$ together with basepoint preserving maps $\gamma: S^1 \wedge X_n \to X_{n+1}$.  A map $f: X \to Y$ of spectrum is a sequence $(f_n)$ of maps $f_n:X_n \to Y_n$ in $\sset_{\ast}$ such that the following commutes for every $n \geq 0$.
 
 \[
 \xy
(0,18)*+{S^1 \wedge X_n}="W";
(0,0)*+{ X_{n+1}}="X";
(30,0)*+{Y_{n+1}}="Y";
(30,18)*+{S^1 \wedge Y_n}="E";
{\ar@{->}^-{f_{n+1}}"X";"Y"};
{\ar@{->}^-{\gamma_X}"W";"X"};
{\ar@{->}^-{\Id \wedge f_n}"W";"E"};
{\ar@{->}^-{\gamma_Y}"E";"Y"};
\endxy
\]
\end{df}
If we take a general left adjoint $T$ that generalizes the suspension functor, we find the definition of spectrum given by Hovey \cite[Definition 1.1]{Hov_stable}.
With the adjunction $(S^1 \wedge -) \dashv \Omega$, the previous definition is equivalent to the one below that we shall work with (see Adams \cite[Part III]{Adams_stable}).
\begin{df}\label{df_spec_dual}
A spectrum $X$ is a sequence of pointed simplicial sets $(X_n)_{n \in \N}$ together with basepoint preserving maps $\varepsilon: X_n \to \Omega(X_{n+1})$.  A map $f: X \to Y$ of spectrum is a sequence $(f_n)$ of maps $f_n:X_n \to Y_n$ in $\sset_{\ast}$ such that the following commutes for every $n \geq 0$.
 
 \[
 \xy
(0,18)*+{ X_n}="W";
(0,0)*+{ \Omega (X_{n+1})}="X";
(30,0)*+{\Omega(Y_{n+1})}="Y";
(30,18)*+{ Y_n}="E";
{\ar@{->}^-{\Omega(f_{n+1})}"X";"Y"};
{\ar@{->}^-{\varepsilon_X}"W";"X"};
{\ar@{->}^-{f_n}"W";"E"};
{\ar@{->}^-{\varepsilon_Y}"E";"Y"};
\endxy
\]
\end{df}
The category of (pre)spectra will be denoted by $\spu$.
\begin{rmk}
For every $n$ we have an object $[X_n] \in \M \downarrow \Omega$ given by $[X_n]=[X_n, X_{n+1}, \varepsilon]$. This gives a sequence of objects in  $\M \downarrow \Omega$. In virtue of this, we will study spectra through sequences of objects in $\mdua$. 
\end{rmk}

We have a natural forgetful functor $P: \spu \to \prod_{n\in \N} \mum$ that maps $X$ to $([X_n])_{n \in \N}$.  Projecting the $n$th factor yields a functor $P_n: \spu \to \arm$ that maps $X \mapsto \varepsilon_n$.  As $n$ runs through $\N$, we have a Quillen-Segal theory on the category $\spu$. Our goal here is to show that we can transfer the model structure from $\prod_{n\in \Z} \mum$ to the category $\spu$ with the classical argument of right induced model structure (see for example \cite{Beke_2}, \cite{Jardine-Goerss}). On $\mum$ we have an injective model structure $\mum_{inj}$(Theorem \ref{inj-thm}) and the projective model structure $\mum_{proj}$ (Theorem \ref{proj-thm}).
Each model structure will induce a model structure on $\spu$ called projective and injective model structure just like for (unbouded)  chain complexes (see \cite{Hov-model}). But first we need to set up some definitions and properties.  
\subsection{$\Z$-sequences and spectra}

As pointed out by Adams, the indexing set for spectra can be either $\Z$ or $\N$. We will index our sequences over $\Z$ and mention explicitly when we index over $\N$. Many of the constructions that will follow hold for any ordinal $\od \subseteq \Z$, possibly finite. We will denote by $\zd$ the discrete category associated to the set $\Z$. Similarly we have $\nd$ and $\ods$ the discrete category associated to $\N$ and $\od$ respectively.

\begin{warn}
\begin{enumerate}
\item So far if $\Xa$ is an object of $\mdua$, we've written $\Xa=[\Xa_0, \Xa_1, \pi_{\Xa}: \Xa_0 \to \Ub(\Xa_1)]$. But when we consider sequences $(\Xa_n)$ there will be multiple indexes, so we will write instead $\Xa=[\Xa^0, \Xa^1, \qi: \Xa^0 \to \Ub(\Xa^1)]$ to put the indexing integer as $\Xa_n^0$ . We've change the letter $\pi$ into $\qi$ because the letter $\pi$ refers to homotopy groups. 
\item In the upcoming definitions we will be talking about $\Z$-sequences but this simply means $\Z$-indexed sequences and \textbf{not} a $\lambda$-sequence as in \cite[Definition 2.1.1]{Hov-model}. 
\end{enumerate}

\end{warn}

\begin{df}
Let $\Ub:\M \to \M$ be a right Quillen functor. \\
\begin{enumerate}
\item A $\Z$-sequence with coefficient in $\mum$ is  an object $\Xa=(\Xa_n)_{n \in \Z} \in \prod_{n\in \Z} \mum$, in that for every $n \in \Z$, $$\Xa_n=[\Xa_n^0, \Xa_{n}^1, \qi_n: \Xa_n^0 \to \Ub(\Xa_{n}^1)].$$
Equivalently $\Xa$ determines a functor $\Xa: \Z_{disc} \to \mum$, that maps $n \mapsto \Xa_n$.
\item Say that a $\Z$-sequence $\Xa$ is \textbf{linked} if for every $n$ there is an isomorphism:
$$\tau_n:\Xa_n^1 \xrw{\cong} \Xa_{n+1}^0.$$
We will denote by $(\Xa,\tau)$ the linked sequence. 
\item Say that a morphism $\sg: (\Xa,\tau) \xrw{(\sg_n)} (\Ya,\tau')$ is a \textbf{linked morphism}, if for every $n$, we have $ \tau'_n \circ \sg_{n}^1 = \sg_{n+1}^0 \circ \tau_n $, or equivalently  $ \sg_{n+1}^0=\tau'_n \circ \sg_{n}^1 \circ \tau_n^{-1}$. In other words the following commutes.

 \[
 \xy
(0,18)*+{\Xa_n^1}="W";
(0,0)*+{ \Xa_{n+1}^0}="X";
(30,0)*+{\Ya_{n+1}^0}="Y";
(30,18)*+{\Ya_n^1}="E";
{\ar@{->}^-{\sg_{n+1}^0}"X";"Y"};
{\ar@{->}^-{\tau_n}_{\cong}"W";"X"};
{\ar@{->}^-{\sg_n^1}"W";"E"};
{\ar@{->}^-{\tau'_n}_{\cong}"E";"Y"};
\endxy
\]
\item Say that a linked sequence $(\Xa,\tau)$ is \textbf{strictly linked} if for every $n$, $\tau_n$ is the identity morphism.
\end{enumerate}
\end{df}


Pictorially, we can represent an unlinked $\Z$-sequence by a `stair diagram':
\[
 \xy
(-30,18)*+{\Xa_n^0}="W";
(0,0)*+{\Xa_{n+1}^0}="X";
(30,0)*+{\Ub(\Xa_{n+1}^1)}="Y";
(0,18)*+{\Ub(\underbrace{\Xa_{n}^1})}="E";
{\ar@{->}^-{\qi_{n+1}}"X";"Y"};
{\ar@{.}^-{\tx{?}}"E";"X"};
{\ar@{->}^-{\qi_n}"W";"E"};
\endxy
\] 
When the sequence is linked we can move up the lower stair and get a map $\Ub(\tau_n) \circ \qi_n : \Xa_n^0 \to \Ub(\Xa_{n+1}^0)$. Each of this map has an adjoint $\Fb( \Xa_n^0) \to \Xa_{n+1}^0$. If $\Fb$ is the suspension functor given by $\Fb= S^1 \wedge -$, then the previous map takes the form $S^1 \wedge \Xa_n^0 \to \Xa_{n+1}^0$, and we see that we have a spectrum as in Definition \ref{df-bous_fried} above, when  the sequence is indexed over $n \in \N$. 
The sequence $(\Xa_n^0,\Ub(\tau_n) \circ \qi_n)$ is the same spectrum as in Definition \ref{df_spec_dual}. 

\begin{pdef}\label{def-spectra}
With the previous notation we have:
\begin{enumerate}
\item The category of $\Z$-sequences is equivalent to the functor category $\Hom(\zd, \mum)= (\mum)^{\zd}$.
\item Linked $\Z$-sequences and linked morphisms form a category $Sp_{\Ub}^+(\M)$ whose objects are called $\Ub$-prespectrum.
\item Strictly linked $\Z$-sequences form a subcategory $ Sp_{\Ub}(\M) \subset Sp_{\Ub}^+(\M)$. 
\end{enumerate}
\end{pdef}

\begin{df}
For a linked $\Z$-sequence $(\Xa,\tau)$, define the associated $\Ub$-prespectrum to be the sequence $(\Xa_n^0,\Ub(\tau_n) \circ \qi_n)$.
\end{df}

\begin{prop}
The inclusion $ Sp_{\Ub}(\M) \hrw Sp_{\Ub}^+(\M)$ is an equivalence of categories. A quasi inverse is the functor $\Theta: Sp_{\Ub}^+(\M) \to Sp_{\Ub}(\M) $  defined as follows. 
\begin{enumerate}
\item $\Theta$ maps $(\Xa,\tau)$ to the strictly linked sequence $\tld{\Xa}=(\tld{\Xa}_n)$ given by $$\tld{\Xa}_n= [\Xa_n^0, \Xa_{n+1}^0,\Ub(\tau_n) \circ \qi_n ].$$
\item $\Theta$ maps $\sg=(\sg_n)$ to $\tld{\sg}=(\tld{\sg}_n)$ given by: 
$$\tld{\sg}_n= (\tld{\sg}_n^0,\tld{\sg}_n^1) :=(\sg_n^0, \tau'_n \circ \sg_{n}^1 \circ \tau_n^{-1}).$$
\end{enumerate}
\end{prop}

\begin{proof}
Clearly for every $n$ we have $\tld{\Xa}_n^1=\tld{\Xa}_{n+1}^0$ by construction. This means that $\Theta((\Xa,\tau)) \in Sp_{\Ub}(\M)$. By the same reasoning $\Theta(\sg)$ is a linked morphism if $\sg$ is. These constructions are clearly functorial and $\Theta$ is well defined.

The proposition will be proved as soon as we establish that  $\sg \mapsto \tld{\sg}$ is an isomorphism of hom-sets: 
$$\Theta: \Hom_{\spup}(\Xa,\Ya) \xrw{\cong} \Hom_{\spu}(\tld{\Xa},\tld{\Ya}).$$
And this is clear since we have an inverse  that takes $(\tld{\sg}_n^0,\tld{\sg}_n^1) \mapsto (\sg_n^0, {\tau'_n}^{-1} \circ \tld{\sg}_n^1 \circ \tau_n)=(\sg_n^0, {\tau'_n}^{-1} \circ \tld{\sg}_{n+1}^0 \circ \tau_n)$.  
\end{proof}

With the previous proposition we will focus on the strictly linked sequences to simplify the constructions.
\begin{rmk}
In the standard situation of spectra, $\M$ is a category of pointed objects, and therefore the initial object $\emptyset$ and the terminal object $\ast$ are uniquely isomorphic to form a zero object in $\M$. In the upcoming results we will only use this property when really needed. So we shall write distinctly $\emptyset$ for the initial object and $\ast$ for the terminal object. 
\end{rmk}

\begin{nota}
With the previous definition we will adopt the following notation. 
\begin{enumerate}
\item The above  forgetful functor $P: Sp_{\Ub}(\M) \to \zmum$ will take the form $P: Sp_{\Ub}(\M) \to (\mum)^{\zd}$.
\item Similarly we will denote by $P_n: Sp_{\Ub}(\M) \to \mum$, the composite of $P$ followed by the $n$th projection: $P_n(\Xa)=\Xa_n$.
\end{enumerate}
\end{nota}

\begin{prop}\label{prop-adj-P}
The category $Sp_{\Ub}(\M)$ is complete and cocomplete and the following hold. 

\begin{enumerate}
\item The functor $P:  Sp_{\Ub}(\M) \to  \zmum$ creates limits and colimits which are computed level-wise.
\item The functor $P$ has a left adjoint and this adjunction is monadic.
\end{enumerate} 
\end{prop}
We differ the proof for the moment because we need the following intermediate result that will simplify the proof of the proposition.

\begin{lem}\label{adj-pn}
The projection functor $P_n$ has a left adjoint  $\Upsilon_n: \mum \to \spu$  defined as follows.
\begin{enumerate}
\item For $\Aa=[\Aa_0,\Aa_1, \qi] \in \mum$ we define $\Upsilon_n(\Aa)=(\tld{\Aa}_k)_{k \in \Z} \in \prod_{k\in \Z} \mum$,  by the formulas:
\begin{itemize}
\item $\tld{\Aa}_k= L_1(\emptyset)= [\emptyset, \emptyset, \emptyset \xrw{!} \Ub(\emptyset)]$ if $k<n-1$
\item $\tld{\Aa}_k= L_1(\Aa_1)= [\emptyset, \Aa_0, \emptyset \xrw{!} \Ub(\Aa_0)]$ if $k=n-1$
\item $\tld{\Aa}_k= \Aa= [\Aa_0,\Aa_1, \Aa_0 \xrw{\qi} \Ub(\Aa_1)]$ if $k=n$
\item $\tld{\Aa}_k= \Fb^{+}(\Fb^{k-(n+1)}(\Aa_1))= [\Fb^{k-(n+1)}(\Aa_1), \Fb(\Fb^{k-(n+1)}(\Aa_1)),  \eta_{\Fb^{k-(n+1)}(\Aa_1)}]$ if $k \geq n+1$; where $\eta$ is the unit of the adjunction $\Fb \dashv \Ub$. Equivalently we have the inductive formula:
 $$\tld{\Aa}_{k}= \Fb^+(\tld{\Aa}_{k-1}^1)=[\tld{\Aa}_{k-1}^1,\Fb(\tld{\Aa}_{k-1}^1), \eta_{\tld{\Aa}_{k-1}^1}], \quad k \geq n+1 .$$
\end{itemize}
\item If $\sg=[\sg_0; \sg_1] \in \Hom_{\mum}(\Aa,\Ba)$ we define $\Upsilon_n(\sg)=(\tld{\sg}_k)_{k \in \Z}$ by the formulas:
\begin{itemize}
\item $\tld{\sg}_k= [\Id_{\emptyset}, \Id_{\emptyset}]$ if $k<n-1$
\item $\tld{\sg}_k= [\Id_{\emptyset},\sg_0]$  if $k=n-1$
\item $\tld{\sg}_k=\sg= [\sg_0,\sg_1]$ if $k=n$
\item $\tld{\sg}_k= \Fb^{+}(\Fb^{k-(n+1)}(\sg_1))= [\Fb^{k-(n+1)}(\sg_1), \Fb( \Fb^{k-(n+1)}(\sg_1)) ]$ if $k \geq n+1$. In a short form we have the inductive formula:
$$\tld{\sg}_k= \Fb^+(\tld{\sg}_{k-1}^1)=[\tld{\sg}_{k-1}^1, \Fb(\tld{\sg}_{k-1}^1)], \quad k \geq n+1.$$
\end{itemize}
\end{enumerate}
\end{lem}
Before we give the proof, let's take a moment to see what $\Upsilon_n$ really does. If we are given an object $\Aa=[\Aa_0,\Aa_1, \Aa_0 \xrw{\qi} \Ub(\Aa_1)]$, by adjunction the map $\qi$ corresponds to a unique map $\Fb(\Aa_0) \xrw{ \varrho(\qi)} \Aa_1$. Then  $\Upsilon_n(\Aa)$ is the prespectrum given by the sequence:
$$\emptyset, \cdots, \emptyset, \underbrace{\Fb(\Aa_0)}_{n \tx{th}},\underbrace{\Aa_1}_{(n+1)\tx{th}}, \underbrace{\Fb(\Aa_1)}_{(n+2)\tx{th}}, \Fb^2(\Aa_1),\cdots, \underbrace{\Fb^{k-(n+1)}(\Aa_1)}_{k \geq (n+3)} \cdots,$$
with the connecting morphisms:
 $$\cdots, \emptyset \xrw{\Id_{\emptyset}} \emptyset, \Fb(\Aa_0) \xrw{ \varrho(\qi)} \Aa_1, \Fb(\Aa_1) \xrw{\Id} \Fb(\Aa_1), \Fb^2(\Aa_1) \xrw{\Id} \Fb^2(\Aa_1),\cdots  \Id_{\Fb^{k-(n+1)}(\Aa_1)}, \cdots .$$ 
 
In our formalism, we consider the adjoint definition of the spectrum given by the following sequence and its connecting morphisms.
$$\emptyset, \cdots, \emptyset, \underbrace{\Aa_0}_{n \tx{th}},\underbrace{\Aa_1}_{(n+1)\tx{th}}, \underbrace{\Fb(\Aa_1)}_{(n+2)\tx{th}}, \Fb^2(\Aa_1),\cdots, \underbrace{\Fb^{k-(n+1)}(\Aa_1)}_{k \geq (n+3)} \cdots.$$
$$\eta_{\emptyset}, \cdots, \emptyset \xrw{!} \Ub\F(\Aa_0), \Aa_0 \xrw{ \qi} \Ub(\Aa_1), \Aa_1 \xrw{\eta_{\Aa_1}} \Ub\Fb(\Aa_1), \Fb(\Aa_1) \xrw{\eta_{\Fb(\Aa_1)}} \Ub\Fb^2(\Aa_1),\cdots  \eta_{\Fb^{k-(n+1)}(\Aa_1)} , \cdots .$$

\begin{proof}[Proof of Lemma \ref{adj-pn}]
It's clear that $\Upsilon_n(\Aa)$ is a linked sequence and that $\Upsilon_n(\sg)$ is a linked morphism. The constructions are clearly functorial so the previous data define a functor. \ \\
The lemma will follow as soon as we show that for every $\Aa \in \mum$ and for every $\Xa \in \spu$, there is a functorial isomorphism of hom-sets:
$$\varphi: \Hom_{\mum}(\Aa,\Xa_n) \xrw{\cong} \Hom_{\spu}(\Upsilon_n(\Aa), \Xa).$$
 This is straightforward but we give the proof for the reader's convenience. 
If $\sg \in \Hom_{\mum}(\Aa,\Xa_n)$ we define  $\varphi(\sg)= (\tld{\sg}_k) \in \Hom_{\spu}(\Upsilon_n(\Aa), \Xa)$ as the linked morphism given by the formulas:
\begin{itemize}
\item $\tld{\sg}_k: L_1(\emptyset) \to \Xa_k$ is adjoint map to the unique morphism $\emptyset \xrw{!} \Xa_k^1$, in the adjunction $L_1 \dashv \Piun$ of Theorem \ref{adjunction-thm}, if $k<n-1$.  The components of $\tld{\sg}_k$ are $\tld{\sg}_k^0=\emptyset \xrw{!} \Xa_k^0$  and $\tld{\sg}_k^1= \emptyset \xrw{!} \Xa_k^1$ the unique morphisms from the initial object. 
\item $\tld{\sg}_k=[\Id_{\emptyset},\sg_0]: L_1(\Aa_0) \to \Xa_{n-1}$ is the adjoint map to $\sg_0: \Aa_0 \to \Xa_{n}^0=\Xa_{n-1}^1$  if $k=n-1$. Thus this morphism is uniquely determined by $\sg_0$. 
\item $\tld{\sg}_k=\sg: \Aa \to \Xa_n$ if $k=n$
\item Inductively for $k \geq n+1$, $\tld{\sg}_k: \tld{\Aa}_k= \Fb^+(\tld{\Aa}_{k-1}^1) \to \Xa_k $ is the adjoint map to $\tld{\sg}_{k-1}^1:  \tld{\Aa}_{k-1}^1 \to \Xa_{k-1}^1=\Xa_{k}^0$.
\end{itemize}

The  function $\varphi$ is clearly  $1$-$1$ because the $n$th component is $\sg$. And we have an inverse function $\varphi^{-1}:\Hom_{\spu}(\Upsilon_n(\Aa), \Xa) \to \Hom_{\mum}(\Aa,\Xa_n)  $ that projects the $n$th component. And the lemma follows. 
\end{proof}

\begin{pdef}
 For $n \in \Z$, define the $n$th \emph{Dirac mass} functor 
 $$\delta_n: \mum \to \Hom(\zd,\mum) $$
as the left adjoint to the $n$th evaluation $\Ev_n: \Hom(\zd,\mum) \to \mum$ defined by $\Ev_n(\Xa)=\Xa_n$.\ \\
For an object $\Aa \in \mum$, $\delta(\Aa)$ is given by:
\begin{itemize}
\item $\delta_n(\Aa)_k= \Aa $ if $k = n$.
\item $\delta_n(\Aa)_k= \emptyset $ if $k \neq n$.
\end{itemize} 
\begin{enumerate}
\item We have an functorial isomorphism of hom-sets: 
$$\Hom_{\mum}(\A, \Xa_n) \cong \Hom_{(\mum)^{\zd}}(\delta(\Aa), \Xa).$$
\item For every $\Xa=(\Xa_n) \in (\mum)^{\zd}$ we have $\Xa \cong \sum_n \delta_n(\Xa_n) =  \coprod_n \delta_n(\Xa_n).$   
\end{enumerate}
\end{pdef}

\paragraph{Proof of Proposition \ref{prop-adj-P}}
We can give the proof of Proposition \ref{prop-adj-P} as follows. 
\begin{proof}
That $P:  Sp_{\Ub}(\M) \to (\mum)^{\zd} $ creates limits and colimits is obvious and we leave it as an exercise for the reader. So Assertion $(1)$ is clear. \ \\

Since $\spu$ is complete, coproducts exist and by a classical argument we can find a left adjoint to $P$ as follows. For $\Aa=(\Aa_n) \in \prod_{n\in \Z} \mum$, let $\Upsilon: \prod_{n\in \N} \mum \to \spu $ be the functor defined by:
$$\Upsilon(\Aa)= \sum_{n} \Upsilon_n(\Aa_n)= \coprod_n \Upsilon_n(\Aa_n).$$ 
For every $\Xa \in \spu$ we have the following isomorphism of hom-sets:

\begin{align*}
\Hom_{\spu}(\Upsilon(\Aa), \Xa) & = \Hom_{\spu}(\coprod_n \Upsilon_n(\Aa_n), \Xa) \cong   \prod_n \Hom_{\mum}(\Aa_n, P_n(\Xa)) \\
& \cong \prod_n \Hom_{\mum}(\Aa_n, \Xa_n) =  \prod_n \Hom_{\mum}(\Aa_n, \Ev_n(\Xa))  \\
& \cong \prod_n \Hom_{(\mum)^{\Z}}(\delta(\Aa_n), \Xa)\\
& \cong  \Hom_{(\mum)^{\zd}}(\coprod_n\delta(\Aa_n), \Xa)\\
& \cong  \Hom_{(\mum)^{\zd}}(\Aa, \Xa)\\
\end{align*}
Finally this adjunction is monadic by Beck monadicity theorem (see \cite{Bar_Well}). Indeed, $P$ reflects isomorphisms because if $\sg=(\sg_k)$ is a  linked morphism such that each $\sg_k$ has an inverse $\sg_k^{-1}$ in $\mum$, then the sequence $(\sg_k^{-1})$ is a linked morphism which is the inverse of $\sg$. Moreover as mentioned before the functor $P$ creates colimits in particular it creates coequalizers of $P$-split pairs. And with the left adjoint $\Upsilon$ we are in the hypotheses of Beck's theorem. 
This ends the proof of the Proposition.
\end{proof}

\subsection{Homotopy of $\Z$-sequences}
The category of $\Z$-sequences is the diagram category $\zmum$. By Theorem \ref{inj-thm} 	and Theorem \ref{proj-thm}, there is an injective and projective model structure on $\mum$. We will then consider the homotopy theory of the diagram category $\Hom(\zd, \mum_{inj})$ and $\Hom(\Z, \mum_{proj})$.  Model structures on diagram categories have been studied for decades and they are well known in the literature (see for example \cite{DHKS}, \cite{Hirsch-model-loc}, \cite{Simpson_HTHC}). If $\M$ is simply cofibrantly generated we will consider the projective model structure on $\zmum$. And if $\M$ is combinatorial then we will also consider the injective model structure on the diagram category $\zmum$.
\begin{note}
As before, given a subset $\od \subseteq \Z$ we have a restriction $\zmum \to \Hom(\ods,\mum)$ and the constructions that will follow hold for the diagram category  $\Hom(\ods,\mum)$
\end{note}

\begin{df}
Let $\sg: \Xa \to \Ya$ be a morphism in $\zmum$ and consider a model structure on $\mum$.
\begin{itemize}
\item Say that $\sg$ is a level weak equivalence if for every $n$, $\sg_n: \Xa_n \to \Ya_n$ is a level-weak equivalence.
\item  Say that $\sg$ is a level (trivial) cofibration  if for every $n$, $\sg_n: \Xa_n \to \Ya_n$ is a level (trivial) cofibration.
\item  Say that $\sg$ is a level (trivial) fibration  if for every $n$, $\sg_n: \Xa_n \to \Ya_n$ is a level (trivial) cofibration.
\end{itemize} 
\end{df}

\begin{thm}
Let $\M$ be a combinatorial model category and consider the category $\mum$ with the injective model structure or the projective model structure.
\begin{enumerate}
\item There is an injective model structure on the category $\zmum$ where the cofibrations and the weak equivalences are level-wise.
\item There is a projective model structure on the category $\zmum$ where the cofibrations and the weak equivalences are level-wise.
\end{enumerate}
Both model categories are combinatorial and left proper if $\M$ is in addition left proper.
\end{thm}

\begin{proof}
The existence of these model structures can be found for example in \cite{Simpson_HTHC}. The left properness follows from the fact that $\mum_{inj}$ and $\mum_{proj}$ are left proper since  pushouts are computed object wise. We've seen that $\mum$ is locally presentable, therefore $\zmum$ is also locally presentable.
\end{proof}

\subsubsection{Strict model structures for prespectra}
We wish to transfer the model structure from $\zmum$ to $\spu$. The ingredient that is needed is the following lemma.
\begin{lem}\label{lem-pushout}
Let $\sg: \Aa \xrw{[\sg_0,\sg_1]} \Ba$ be a map in $\mum$. Then the following hold. 
\begin{enumerate}
\item For every $n$, $\Upsilon_n(\sg) : \Upsilon_n(\Aa) \to  \Upsilon_n(\Ba)$ is a level (trivial) cofibration if $\sg$ is a level (trivial) cofibration
\item For every $n$ and for any pushout square as follows 
 \[
 \xy
(0,18)*+{\Upsilon_n(\Aa)}="W";
(0,0)*+{\Upsilon_n(\Ba)}="X";
(30,0)*+{\Upsilon_n(\Ba) \cup^{\Upsilon_n(\Aa)} \Xa}="Y";
(30,18)*+{\Xa}="E";
{\ar@{->}^-{}"X";"Y"};
{\ar@{->}^-{\Upsilon_n(\sg)}"W";"X"};
{\ar@{->}^-{}"W";"E"};
{\ar@{->}^-{u}"E";"Y"};
\endxy
\]
the map $u:\Xa \to \Upsilon_n(\Ba) \cup^{\Upsilon_n(\Aa)} \Xa$ is a level trivial cofibration if $\sg$ is.  In particular  it's a level weak equivalence. 
\end{enumerate}
\end{lem}

\begin{proof}
By definition given  $\sg: \Aa \xrw{[\sg_0,\sg_1]} \Ba$, 
$\Upsilon_n(\sg)=(\tld{\sg}_k)_{k \in \N}$ is given by the formulas:
\begin{itemize}
\item $\tld{\sg}_k= [\Id_{\emptyset}, \Id_{\emptyset}]$ if $k<n-1$
\item $\tld{\sg}_k= [\Id_{\emptyset},\sg_0]$  if $k=n-1$
\item $\tld{\sg}_k=\sg= [\sg_0,\sg_1]$ if $k=n$
\item $\tld{\sg}_k= \Fb^{+}(\Fb^{k-(n+1)}(\sg_1))= [\Fb^{k-(n+1)}(\sg_1), \Fb( \Fb^{k-(n+1)}(\sg_1)) ]$ if $k \geq n+1$.
\end{itemize}
Since $\Fb$ is a left Quillen functor, it preserves the cofibrations and the trivial cofibrations. Therefore for every $k\geq n+1$, $\Fb^{k-(n+1)}(\sg_1)$ and $\Fb( \Fb^{k-(n+1)}(\sg_1))$ are (trivial) cofibrations if $\sg_1$ is a (trivial) cofibration. This means that $\tld{\sg}_k$ is a level (trivial) cofibration if $\sg_1$ is for $k \geq n+1$. 

For $k \leq n$, the components of $\tld{\sg}_k$ are one of the maps $\Id_{\emptyset}$, $\sg_0$ or $\sg_1$. So clearly $\tld{\sg}_k$ is a level (trivial) cofibration if $\sg_0$ and $\sg_1$ are (trivial) cofibrations. This gives Assertion $(1)$. \ \\

Assertion $(2)$ is clear because colimits and in particular pushouts in the category  $\spu$ are computed in $\zmum$. Finally since $\mum_{inj}$ is a model category, the cobase change of any trivial cofibration is again a trivial cofibration.
\end{proof}
\paragraph{Projective model structure} 
In what follows, we endow $\mum$ with either the injective or the projective model structure. 

\begin{thm}\label{proj-spec}
Let $\M$ be a cofibrantly generated model category and consider the projective model structure on $\zmum$. Then there is a right induced model structure on the category $\spu$ which is cofibrantly generated which may be described as follows. 
\begin{enumerate}
\item A map $\sg$ is a weak equivalence if $P(\sg)$ is a weak equivalence in $\zmum_{proj}$
\item A map $\sg$ is a fibration if $P(\sg)$ is a fibration in $\zmum_{proj}$
\item A map $\sg$ is a cofibration if  it posses the LLP with respect to any map that is simultaneously a fibration and a weak equivalence.
\end{enumerate}
The adjunction $\Upsilon: \zmum \rightleftarrows \spu : P$ is a Quillen adjunction and for every $n$, $\T_n= \Piar \circ \Ev_n \circ P_n: \spu \to \arm$ is right Quillen.
\end{thm}

\begin{proof}
This is a classical argument of Quillen \cite{Quillen_HA}. See also Goerss-Jardine \cite{Jardine-Goerss} or Beke \cite{Beke_1} \cite{Beke_2}. The main ingredient is precisely Lemma \ref{lem-pushout} above.
\end{proof}

\begin{cor}
If in addition $\M$ is left proper (resp. combinatorial) then  $\spu$ is left proper (resp. combinatorial).
\end{cor}

\begin{proof}
Left properness follows from the fact that pushouts in $\spu$ are computed object wise in $\mum$, and clearly $\mum_{inj}$ and $\mum_{proj}$ is left proper if $\M$ is. To prove that $\spu$ is combinatorial boils down to show that it's locally presentable. This can be proved directly but we can take a shortcut using the fact that the adjunction $\Upsilon \dashv P$ is monadic and the monad is clearly finitary i.e, it preserves filtered colimits. 
\end{proof}

\begin{rmk}
If $\M$ is combinatorial, we have a similar theorem for the injective model structure on $\zmum$.
\end{rmk}
\paragraph{Bousfield-Friedlander strict model structure}
We can be more specific when we consider the ``projective-projective'' model category $\Hom(\nd,\mum_{proj})_{proj}$. We do not require $\M$ to be cofibrantly generated. And the argument holds for a bounded below ordinal $\od \subset \Z$. 

\begin{thm}\label{strict-proj-natural}
Let $\Ub: \M \to \M$ be a right Quillen functor and let $\mum_{proj}$ be  the projective model structure  $\mum$ as in Theorem \ref{proj-thm}.
Then there is a model structure on $\spu$ called the \textbf{projective strict model structure} which may be described as follows. 

\begin{itemize}
\item A map $\sg : \Xa \to \Ya$ is a weak equivalence if for every $n$, $\sg_n: \Xa_n \to \Ya_n$ is a weak equivalence in $\mum_{proj}$, that is a level weak equivalence. 
\item  A map $\sg : \Xa \to \Ya$ is a fibration if for every $n$, $\sg_n: \Xa_n \to \Ya_n$ is a fibration in $\mum_{proj}$, that is a level fibration. 
\item  A map $\sg : \Xa \to \Ya$ is a cofibration if for every $n$, $\sg_n: \Xa_n \xrw{[\sg_n^0,\sg_n^1]} \Ya_n$ $n$ is a cofibration in $\mum_{proj}$, that is for every $n$ , $\sg_n^0$ is a cofibration and the canonical map $\delta:\Xa_{n+1}^1 \cup^{\Fb(\Xa_n^0)} \Fb(\Ya_n^0) \to \Ya_{n+1}^1 $ is a cofibration in $\M$.
\item If $\M$ is cofibrantly generated (resp. left proper, resp. combinatorial) then $\spu$ is cofibrantly generated (resp. left proper, resp. combinatorial). We will denote it by $\spu_{proj}$.
\end{itemize}
\end{thm} 

\begin{proof}
One get the factorization axioms and the lifting properties inductively because $\N$ is bounded below and is a direct Reedy category.  The main reason is that given a map $\sg_n: \Xa_n \xrw{[\sg_n^0,\sg_n^1]} \Ya_n$ in $\mum$, when we factor it in the projective  model structure we always start with a factorization of $\sg_n^0$. We also do the same thing when we construct a solution for lifting properties. Proceeding this way we can factor any linked morphism $\sg$ being careful to take the factorization of $\sg_n^1$ and apply them to $\sg_{n+1}^0=\sg_n^1$ to build in $\mum$ a factorization of $\sg_{n+1}=[\sg_{n+1}^0,\sg_{n+1}^1]$. One gets inductively a the required factorization (resp lifting) at each step by linked morphisms.The same  argument is used by Rosick\'y \cite[Section 3]{Rosicky_spec}.\ \\

If $\M$ is cofibrantly generated (resp. combinatorial, resp left proper) then this model structure coincides with the right induced model structure transferred from $\Hom(\nd,\mum_{proj})_{proj}$ with Theorem \ref{proj-spec}. 
\end{proof}

\begin{cor}
If $\M=\sset_{\ast}$ with the standard Kan-Quillen model structure, then the ``projective-projective'' model structure on $\spu$ of Theorem \ref{proj-spec}  coincides with the strict model structure of Bousfield-Friedlander that is also obtained by Hovey \cite{Hov_stable}, Schwede \cite{Schwede_stable}.
\end{cor}
\begin{proof}
Clear. 
\end{proof}

\subsubsection{Stable homotopy category}
We have a projective Quillen-Segal theory $\T$ on $\spu_{proj}$ with $\T_n: \spu_{proj} \to \arm_{proj}$, defined by $\T_n(\Xa)=[\varepsilon: \Xa_n^0 \to \Ub(\Xa_{n+1}^1]$. Each $\T_n$ is right Quillen and the left adjoint is $\Upsilon_n \circ \Gamma$. We assume that  $\M$ is a tractable and left proper model category such as $\sset_{\ast}$. In virtue of Theorem \ref{local-proj-theory} we can localize the theory and get the following. 
\begin{thm}\label{stable-ho}
Let $\Ub: \M \to \M$ be a right Quillen functor for a tractable model category which is left proper. Then the following hold. 
\begin{enumerate}
\item There is a left proper model structure on the category $\spu$ such that every fibrant object $\Xa$ is a $\Ub$-spectrum.
\item If $\M=\sset_{\ast}$ the the homotopy category is equivalent to the stable homotopy category of Bousfield and Friedlander. 
\end{enumerate}
\end{thm}

To prove the last assertion of the theorem we need to outline a result.  Let $\I$ be the set of generating cofibration of $\M$. For every $i \in \M$ we've introduced in Notation \ref{nota-alpha} and Notation \ref{nota-zeta} two maps $\alpha_i$ and $\zeta_i$ in $\arm$. The map $\zeta_i$ is a projective cofibration whereas $\alpha_i$ is an injective cofibration.
\begin{lem}\label{lem-stab-alpha-zeta}
For every $n$, consider the functor $\Upsilon_n \circ \Gamma \in \Hom(\armpj,\spu)$. Then for every $i \in \I$ the following hold. 
\begin{enumerate}
\item The map $\Upsilon \Gamma(\alpha_i) =(\sg_k)$ is such that for every $k\geq n+1$, $\sg_k$ is an isomorphism in $\mum$
\item The map $\Upsilon \Gamma(\zeta_i) =(\sg_k)$ is such that for every $k\geq n+1$, $\sg_k$ is a level trivial cofibration $\mum$. 
\item Taking $\M= \sset_{\ast}$ we find that for every $n$ and for every $i \in \I$ the maps $\Upsilon \Gamma(\alpha_i)$, $\Upsilon \Gamma(\zeta_i)$ are stable equivalence. In particular since $\Upsilon \Gamma(\zeta_i)$ is already a cofibration, then it's a trivial stable cofibration. 
\end{enumerate}
\end{lem}

\begin{proof}
The first assertions follow directly by construction of $\Gamma$ (Proposition \ref{adj-prop-gamma}) and $\Upsilon_n$ (Lemma \ref{adj-pn}). \ \\

The third assertion follows by definition of the stable homotopy groups (see \cite{Adams_stable}, \cite{Schwede_stable}). These groups are defined as direct colimit so what matters is the stabilization for greater $k$. The last assertion is clear
\end{proof}
We are now able to prove Theorem \ref{stable-ho}. 

\begin{proof}[Proof of Theorem \ref{stable-ho}]
The existence of the model structure and the fact that every fibrant object is an $\Ub$-spectrum is simply given by Theorem \ref{local-proj-theory}. \ \\

To prove that that the homotopy category is the usual stable homotopy category we will show that we have the same fibrant objects in the model structure in Hovey \cite[Theorem 3.4]{Hov_stable} which is also a left Bousfield localization of the same strict model structure.\ \\

If $\Xa \in \spu$ is fibrant in our new model category, then $\Xa$ is level fibrant and satisfies the Segal conditions so it's a $\Ub$-spectrum, whence it's a fibrant object in the model structure of Hovey. Conversely assume that $\Xa\in \spu$ is fibrant in Hovey's model structure. Then in a tautological way $\Xa$ is $\sg$-local for any stable equivalence $\sg$. By Lemma \ref{lem-stab-alpha-zeta}, we know that for every $i$ and for every $n$, $\Upsilon \Gamma(\zeta_i)$ is a stable weak equivalence and indeed a trivial cofibration. This means $\Xa$ is $\Upsilon \Gamma(\zeta_i)$-local for every $i$ and for every $n$, therefore $\Xa$ is fibrant in the model structure of our Theorem \ref{local-proj-theory}.
\end{proof}

\begin{cor}
If $\M$ is the category of pointed simplicial sets, our model structure on $\spu$ coincides with the stable model structure of Bousfield-Friedlander \cite{Bous_Fried}, Hovey \cite{Hov_stable}, Schwede \cite{Schwede_stable}. 
\end{cor}

\begin{proof}
Indeed we're localizing the same original model structure and we have the same fibrant objects  as Hovey. It turns out that the new weak equivalences are also the same since we have the same original model structure, thus the same function complexes.
\end{proof}

\begin{rmk}
It remains to compare the previous theorems when we consider the various model structure on $\zmum$. We also need to extend these results when $\M$ is cellular and left proper. This will be done later. There is much that is left to be done on the subject.
\end{rmk}

\subsubsection{$\Ub$-chain complexes}
In this section we consider that $\M$ has a zero object and let $\Ub: \M \to \M$ be right adjoint. 

\begin{df}
Let $\spu$ be the category of linked $\Z$-sequence as before. Say that $\Xa \in \spu$ is an $\Ub$-chain complex if for every $n$ the following composite is the zero map:
$$\Xa_n^0 \xrw{\varepsilon} \Ub(\Xa_{n+1}^0) \xrw{\Ub(\varepsilon)}  \Ub^2(\Xa_{n+2}^0).$$
Denote by $Ch_{\Ub}(\M)$ the full subcategory of $\spu$ spanned by $\Ub$-chain complexes. We have a forgetful functor $P: Ch_{\Ub}(\M) \to \zmum$
\end{df}
Obviously if $\M$ is an abelian category and $\Ub=\Id$ we find here the classical definition of chain complexes. We will develop the homotopy theory of them following  the classical case. And these constructions hold for an ordinal $\od \subseteq \Z$. 
\bibliographystyle{plain}
\bibliography{Bibliography_Ho_QS}
\end{document}